\documentclass[12pt,a4paper, british]{article}
\usepackage[latin1]{inputenc}
\usepackage[small,pagestyles,compact]{titlesec}
\usepackage{aliascnt}
\usepackage{amsmath}
\usepackage{amsfonts}
\usepackage{amssymb}
\usepackage{makeidx}
\usepackage{graphicx}
\usepackage{commath}
\usepackage{mathtools}
\usepackage{time}
\usepackage{etoolbox}
\usepackage{hhline}
\usepackage{tocbasic}
\usepackage[switch,pagewise,mathlines,running]{lineno}
\usepackage{url}
\usepackage{abstract}
\usepackage{appendix}
\usepackage{index}
\usepackage{mathtext}
\usepackage{hypertoc}
\usepackage{titlesec}
\usepackage{allfonts}
\usepackage{hyperref}
\usepackage{fancyhdr}
\usepackage{hypertoc}
\usepackage{sectsty} %
\usepackage{fmtcount}
\usepackage{cases}
\usepackage{braket}
\usepackage{setspace}
\usepackage{setspace}
\usepackage{footnote}
\usepackage{empheq}
\usepackage[width=2.00cm, height=2.00cm, left=2.00cm,
right=2.00cm, top=2.00cm, bottom=2.00cm]{geometry}
\usepackage{censor}
\usepackage{breqn}
\usepackage{authblk}
\newcommand*{\justifyheading}{\centering}
\chapterfont{\normalfont\small\bfseries}
\titleformat{\chapter}[display]{\normalfont\large\rmfamily\bfseries
\singlespacing\justifyheading}
{\MakeUppercase\chaptertitlename\  \thechapter}{20pt}{\uppercase}
\titlespacing*{\chapter}{0pt}{-50pt}{-3.5pt}
\titlespacing{\section}{-10pt}{-15pt}{-7.5pt}
\titlespacing{\subsection}{-10pt}{-10pt}{-5pt}
\titlespacing{\subsubsection}{-10pt}{-5pt}{-2.5pt}
\AtBeginDocument{\addtolength\abovedisplayskip{-0.25\baselineskip}
\addtolength\belowdisplayskip{-0.25\baselineskip}}

\newtheorem{remarks}{Remarks}[section]

\newtheorem{definition}{Definition}[section]
\newtheorem{theorem}{Theorem}[section]
\newtheorem{proposition}[theorem]{Proposition}

\newtheorem{corollary}[theorem]{Corollary}

\newtheorem{remark}{Remark}[section]

\newenvironment{proof}{{\bf Proof.}}{\hspace{\stretch{1}}
	$\square$}
\newaliascnt{lemma}{theorem}

\title{Fixed point theorems for asymptotically $T$-regular
mappings in preordered modular G-metric
spaces applied to solving nonlinear %
integral equations}
\date{\vspace{-5ex}}
\author[1]{\small Godwin Amechi Okeke\footnote{Corresponding author: G.A. Okeke}}
\author[2]{\small Daniel Francis}

\affil[1,2]{\small Department of Mathematics, School
of Physical Sciences, Federal University of Technology
 Owerri, P.M.B. 1526 Owerri, Imo State, Nigeria}

\affil[$\ast1$]{\small E-mail address: gaokeke1@yahoo.co.uk, godwin.okeke@futo.edu.ng}
\affil[2]{\small E-mail address: francis.daniel@mouau.edu.ng}

\begin{document}
	\maketitle
	\allsectionsfont{\centering}
	\allowdisplaybreaks
	\numberwithin{equation}{section}
	\pagenumbering{roman}
	\hypersetup{linkcolor=red}
	\pagenumbering{arabic}
	\pagestyle{myheadings}

\begin{abstract}
Our aim in this paper is to prove some interesting fixed point theorems for the class of
asymptotically $T$-regular mappings in the framework of %
 preordered modular G-metric spaces. Our results are novel and generalizes several know results.
 Furthermore we apply our results in solving nonlinear %
integral equations. %
\end{abstract}

\section{Introduction}

\qquad  Geraghty \cite{13} in 1973 came up with an %
interesting generalization of Banach contraction %
mapping principle using the concept of class %
$\mathcal{S}$  of functions, that is $\kappa: %
\mathbb{R}_{+}\rightarrow [0,1)$ with the %
condition that $\kappa(t_{n})\rightarrow %
1$ implies that $t_{n}\rightarrow 0$ where $\mathbb{R}_{+}$ %
is the set of all non-negative real number %
and $t\in\mathbb{R}_{+}$ for all $n\in\mathbb{N}$ and proved %
that fixed point in this class exists.%

\qquad In 2012, Gordji {\it et al.} \cite{27} proved some %
fixed point theorems for generalized  Geraghty's %
contraction in partially ordered complete metric spaces and %
Bhaskar and Lakshmikantham \cite{22}  proved a fixed point %
 theorem for a mixed %
 monotone mapping in a metric space endowed with partial %
 order, using a weak contractivity type  of assumption. %

 Yolacan \cite{e} established some new fixed point %
 theorems in 0-complete ordered partial metric spaces, %
 and also remarked on coupled generalized Banach %
 contraction mapping. Faraji {\it et al.} \cite{h} %
 extended some fixed point theorems for Geraghty %
 contractive mappings in b-complete b-metric spaces. %
 However, Gupta {\it et al.} \cite{v}, established %
some fixed point theorems in an ordered complete metric %
spaces using distance function. Chaipunya {\it et al.} \cite{1} %
proved some fixed point theorems of Geraghty-type %
contractions concerning the existence and uniqueness %
of fixed points under the setting of modular metric %
spaces which also generalized the results in %
Gordji {\it et al.}\cite{27} %
under the influence of a modular metric spaces. %

 Geraghty-type contractive mappings in %
 metric spaces was generalized to the concept of %
 preordered  G-metric spaces in \cite{G} and %
 obtained unique fixed point results. 

 \qquad In 2010, a remarkable study by Chistyakov %
 \cite{12} introduced  an aspect of metric spaces called  %
 modular metric spaces or parameterized metric spaces %
 with the time parameter $\lambda$ (say) and his %
 purpose was  to define the notion of a modular on an %
 arbitrary set, and developed the theory of metric spaces %
 generated by modulars, called modular metric spaces %
 and, on the basis of it, defined new metric spaces of %
 (multi-valued) functions of bounded generalized variation %
  of a real variable with values in metric semi-groups and %
  abstract convex cones. %

 \qquad In the same year, Chistyakov \cite{23}, as an %
 application presented an exhausting description %
 of Lipschitz continuous and some other classes %
 of superposition (Nemytskii) operators, acting %
 in these modular metric spaces. He developed %
 the theory of metric spaces generated by modulars, %
 and extended the results given by Nakano \cite{x6}, %
 Musielak and Orlicz \cite{x3},  Musielak \cite{x2} %
 to modular metric spaces. %
Modular spaces are extensions of Lebesgue, %
Riesz, and Orlicz spaces of integrable functions. %

 \qquad Modular theories on linear spaces can be found %
 by Nakano in his two monographs \cite{x6,x7}, %
 where he developed a spectral theory in semi-ordered %
 linear spaces (vector lattices) and established the %
 integral representation for projections acting in %
 this modular spaces. %
Nakano \cite{x6} established some modulars on %
 real linear spaces which are convex functional. %
 Non-convex modulars and the corresponding modular %
 linear spaces were constructed by Musielak and %
 Orlicz \cite{x3}. Orlicz spaces and modular linear %
 spaces have already become classical tools in modern %
 nonlinear functional analysis. %

 \qquad The development of theory of metric spaces %
generated by modulars, called modular metric spaces %
attracted  many research mathematicians still  %
investigating fixed point results in this area including %
Chistyakov himself. Chistyakov \cite{3} also established some %
fixed point theorems for contractive maps in modular spaces. %
It is related to contracting, rather generalized average %
velocities than metric distances, and the successive %
approximations of fixed points converge to the %
fixed points in a weaker sense as compared to the %
metric convergence in \cite{3} and other fixed point results %
in modular metric spaces can be  found in \cite{18} and \cite{11}. %
 Considering applicability, these fixed point results are
 applied in solving the fixed point of nonlinear %
 integral equations  see \cite{18,19} and references therein, while %
 \cite{1} deals with application to partial differential equation, %
 in modular metric spaces. Interested reader should see \cite{ok18,ok9,28} 
and the references therein for further studies in modular function spaces.%

\qquad In 2013, Azadifar {\it et al}. \cite{100} introduced the %
concept of modular G-metric space and obtained some fixed point %
theorems of contractive mappings defined on modular G-metric spaces. %

\qquad Recently, Okeke {\it et al.} \cite{new}  proved %
some existence and uniqueness of fixed point %
theorems for mappings satisfying rational %
inequality in modular metric spaces of which %
Geraghty-type theorems in \cite{1} and Reich contraction %
mapping principles are special cases. In this work, we will %
examine some asymptotically $T$-regular mapping theorems %
in preordered modular G-metric spaces which is an %
extension of \cite{new}. %

\qquad In this present paper, we  %
extend the fixed point result obtained in \cite{new} to %
asymptotically $T$-regular mapping theorems %
in preordered modular G-metric spaces. Also  our %
result shows that  the system of nonlinear  %
integral equations have a unique solution in %
preordered modular G-metric %
spaces, $X_{\omega^{G}}.$ %

\section{Preliminaries}

\begin{definition}\cite{ana}%
A perorder set $X$ is a relation $\preceq$ that is both, %
\begin{enumerate}%
\item[(i)] transitive  i.e; $x\preceq y$ and %
$y\preceq z$ implies $x\preceq z$ and, %

\item[(ii)] reflexive i.e; $x\preceq x.$ %
\end{enumerate}%
A preordered set is a pair $(X, \preceq)$  consisting %
 of  a set $X$ and a preorder $\preceq$ on $X.$ %
\end{definition}%

\begin{remark}%
  If a preorder $\preceq$ is antisymmetic i.e;  %
  $x\preceq y$ and $y\preceq x$ implies %
  $x= y$, then $\preceq$ %
  is called a partial order. %
\end{remark}

\begin{definition}\cite{13}%
Let $\mathcal{S}$ be the family of all %
Geraghty functions, that is functions %
 $\alpha:[0,\infty)\rightarrow [0,1)$ %
 satisfying the  condition %
 $\{\alpha(t_{n})\}\rightarrow 1\implies %
 \{t_{n}\}\rightarrow 0$.	%
\end{definition}

We will in this work  take $\mathcal{S}_{Ger}$ %
as the class of all Geraghty functions. %
Such Geraghty class was discussed in \cite{1}.%

\begin{definition}\cite{1}%
	Let $\mathcal{S}$ be the family of all %
	Geraghty functions, that is functions %
	$\beta_{i}:\mathbb{R}_{+}\cup \{\infty\} %
	\rightarrow [0,1)$ satisfying the  condition %
	 $\beta_{i}(t_{k})\rightarrow \frac{1}{n}$ %
	implies that $\{t_{k}\}\rightarrow 0$ for all $i$.	%
\end{definition}%

\begin{definition}\cite{1}%
	Let $\Psi$ be the class of functions $\psi: %
	\mathbb{R}_{+}\rightarrow\mathbb{R}_{+}$ such %
	that the following conditions hold; %
	\begin{enumerate}%
		
		\item[(1)] $\psi$ is decreasing, %
		
		\item[(2)] $\psi$ is continuous, %
		
		\item[(3)] $\psi(t)=0$ if and only if $t=0.$ %
		
	\end{enumerate}	%
\end{definition}%

 Extension of Definition 2.2 above is as follows:%

\begin{definition}\cite{1}%
	Let $\bar{\Psi}$ be the class of functions %
	 $\psi:\mathbb{R}_{+}\cup\{\infty\}\rightarrow %
	\mathbb{R}_{+}\cup\{\infty\}$ such that the following conditions hold; %
	\begin{enumerate}%
		\item[(a)] $\psi$ is sub-additive, %
		
		\item[(b)] $\psi(t)$ is finite for $0<t<\infty,$ %
		
		\item[(c)] $\psi\lvert_{\mathbb{R}_{+}}\in\bar{\Psi}.$ %
	\end{enumerate}	%
\end{definition}%

\begin{definition}\cite{100}\label{2.1}%
	Let $X$ be a nonempty set, and let %
	 $\omega^{G}:(0,\infty)\times X\times X\times X %
	 \rightarrow [0,\infty]$ be a function satisfying; %
	\begin{enumerate}%
		\item[(1)] $\omega^{G}_{\lambda}(x,y,z)=0$ for %
		all $x,y,z\in X$ and $\lambda>0$ if $x=y=z,$ %
		
		\item[(2)] $\omega_{\lambda}^{G}(x,x,y)>0$ for %
		all $x,y\in X$ and $\lambda>0$ with $x\neq y,$ %
		
		\item[(3)] $\omega_{\lambda}^{G}(x,x,y)\le\omega_{ %
			\lambda}^{G}(x,y,z)$ for all $x,y,z\in X$ and %
		$\lambda>0$ with $z\neq y,$ %
		
		\item[(4)]  $\omega_{\lambda}^{G}(x,y,z)=\omega_{ %
			\lambda}^{G}(x,z,y) %
		=\omega_{\lambda}^{G}(y,z,x)=\cdots$ for all %
		 $\lambda>0$ (symmetry in all three variables), %
		
		\item[(5)] $\omega_{\lambda+\mu}^{G}(x,y,z)\le\omega_{ %
			\lambda}^{G}(x,a,a)+\omega_{\mu}^{G}(a,y,z)$, %
		for all $x,y,z,a\in X$ and $\lambda, \nu>0$, %
	\end{enumerate}%

then the function $\omega^{G}_{\lambda}$ %
is called a modular G-metric on $X$. %
\end{definition}%

\begin{remarks}%
	
	\begin{enumerate}%
		
		\item[(a)] The pair $(X, \omega^{G}) $ is %
		called a modular G-metric space, and without %
		any confusion we will take $X_{\omega^{G}}$ as a %
		modular G-metric space. %
		 From condition (5), if $\omega^{G}$ is convex, %
		then we have a strong form  as, %
		
		\item[(b)] $\omega^{G}_{\lambda+\mu}(x,y,z)\le\omega^{G}_{ %
			\frac{\lambda}{\lambda+\mu}}(x,a,a)+\omega^{G}_{ %
			\frac{\mu}{\lambda+\mu}}(a,y,z)$, %
		
		\item[(c)]	If $x=a$, then (5) above becomes %
		 $\omega_{\lambda+\mu}^{G}(a,y,z)\le\omega_{\mu}^{G}(a,y,z)$, %
		
		\item[(d)] Condition (5) is called rectangle inequality. %
		
	\end{enumerate}%
	
\end{remarks}	%

\begin{definition}\cite{100}%
	Let $(X, \omega^{G})$ be a modular G-metric space. %
	The sequence $\{x_{n}\}_{n\in\mathbb{N}}$ in $X$ is %
	 modular G-convergent to $x$, if it converges to %
	  $x$ in the topology $\tau(\omega^{G}_{\lambda})$. %
	  \newline
	A function $T: X_{\omega^{G}}\rightarrow X_{\omega^{G}} $ %
	at $x\in X_{\omega^{G}}$ is called modular G-continuous if %
	 $\omega^{G}_{\lambda}(x_{n},x,x)\rightarrow 0$,  then %
	$\omega^{G}_{\lambda}(Tx_{n},Tx,Tx)\rightarrow 0$, for all %
	 $\lambda>0 .$ %
\end{definition}

\begin{remark}%
	 $\{x_{n}\}_{n\in\mathbb{N}}$  modular G-converges to %
	  $x$ as $n\rightarrow\infty$, if %
	   $\lim\limits_{n\rightarrow\infty}\omega^{G}_{\lambda} %
	  (x_{n},x_{m},x)=0$. That is for all $\epsilon>0$ there %
	  exists $n_{0}\in\mathbb{N}$ such that %
	   $\omega^{G}_{\lambda}(x_{n},x_{m},x)<\epsilon$ for all %
	  $n,m\geq n_{0}$. Here we say that $x$ is modular G-limit of %
	   $\{x_{n}\}_{n\in\mathbb{N}}$. %
\end{remark}%

\begin{definition}\cite{100}%
	Let $(X, \omega^{G})$ be a modular G-metric space, then $\{ %
	 x_{n}\}_{n\in\mathbb{N}}\subseteq X_{\omega^{G}}$ is %
	 said to be  modular G-Cauchy if for every  $\epsilon>0$, %
	 there exists $n_{\epsilon}\in \mathbb{N}$ such that %
	$\omega^{G}_{\lambda}(x_{n},x_{m},x_{l})<\epsilon$ for all %
	 $n,m,l\geq n_{\epsilon}$ and $\lambda>0$.\\ %
	A modular G-metric space $X_{\omega^{G}}$ is said to be %
	modular G-complete if every modular G-Cauchy sequence %
	in $X_{\omega^{G}}$ is  modular G-convergent in $X_{\omega^{G}}$. %
\end{definition}%

\begin{proposition}\cite{100}\label{2.2}%
	Let $(X, \omega^{G})$ be a modular G-metric space, %
	for any $x,y,z,a\in X$, it follows that: %
	\begin{enumerate}
		\item [(1)] If $\omega_{\lambda}^{G}(x,y,z)=0$ for all %
		 $\lambda>0$, then $x=y=z.$ %
		
		\item [(2)] $ \omega_{\lambda}^{G}(x,y,z)\le %
		 \omega_{\frac{\lambda}{2}}^{G}(x,x,y)+\omega_{ %
			\frac{\lambda}{2}}^{G}(x,x,z)$ for  all $\lambda>0 .$ %
		
		\item [(3)]  $\omega_{\lambda}^{G}(x,y,y)\le %
		 2\omega_{\frac{\lambda}{2}}^{G}(x,x,y)$ for  all $\lambda>0.$ %
		
		\item [(4)] $\omega_{\lambda}^{G}(x,y,z)\le %
		 \omega_{\frac{\lambda}{2}}^{G}(x,a,z)+\omega_{\frac{ %
				\lambda}{2}}^{G}(a,y,z)$ for  all $\lambda>0.$ %
			
		\item[(5)] $\omega_{\lambda}^{G}(x,y,z)\le %
		 \frac{2}{3}(\omega_{\frac{\lambda}{2}}^{G}(x,y,a)+ %
		\omega_{\frac{\lambda}{2}}^{G}(x,a,z) +\omega_{ %
			\frac{\lambda}{2}}^{G}(a,y,z))$ for  all $\lambda>0.$ %
		
		\item[(6)] $\omega_{\lambda}^{G}(x,y,z)\le %
		 \omega_{\frac{\lambda}{2}}^{G}(x,a,a)+\omega_{\frac{ %
				\lambda}{4}}^{G}(y,a,a) %
		+\omega_{\frac{\lambda}{4}}^{G}(z,a,a)$ for  all $\lambda>0.$ %
	\end{enumerate}%
\end{proposition}%

\begin{proposition}\cite{100}\label{2.3}
	Let $(X, \omega^{G})$ be a modular G-metric space and  %
	$ \{x_{n}\}_{n\in\mathbb{N}}$ be a sequence in %
	$X_{\omega^{G}}.$  Then the following are equivalent: %
	\begin{enumerate}%
		\item[(1)] $\{x_{n}\}_{n\in\mathbb{N}}$ is $\omega^{G}$-convergent to $x$, %
		
		\item[(2)] $\omega^{G}_{\lambda}(x_{n},x)\rightarrow 0$ as $n\rightarrow \infty,$ i.e; %
		 $\{x_{n}\}_{n\in\mathbb{N}}$ converges to $x$ relative to modular metric %
		  $\omega^{G}_{\lambda}(.)$, %
		
		\item[(3)] $\omega^{G}_{\lambda}(x_{n},x_{n},x)\rightarrow0$ as $n\rightarrow %
		 \infty$ for all $\lambda>0,$ %
		
		\item[(4)] $\omega^{G}_{\lambda}(x_{n},x,x)\rightarrow 0$ as $n\rightarrow %
		 \infty$ for all $\lambda>0,$ %
		
		\item[(5)] $\omega^{G}_{\lambda}(x_{m},x_{n},x)\rightarrow 0$ as $m,n\rightarrow %
		 \infty$ for all $\lambda>0.$ 
	\end{enumerate}%
\end{proposition}%

We give the following definitions %
which will form the basis of the results %
 in this section 3 below.%

\begin{definition}%
	An ordered modular G-metric space is a %
triple $(X,\omega^{G},\preccurlyeq)$ where %
	 $(X,\omega)$ is a modular metric space %
and $\preccurlyeq$ is a partial order on %
	 $X_{\omega^{G}}.$ If $\preceq$ is a %
preorder on $X_{\omega^{G}}$, then %
	 $(X,\omega^{G},\preceq)$ is a %
preordered modular G-metric space. %
\end{definition}%

\begin{definition}%
	A sequence $\{T^{n}x_{0}\}_{n\in\mathbb{N}}$ in a preordered %
	 modular G-metric space  $(X,\omega^{G},\preceq)$ %
	 is asymptotically $T$-regular at $x_{0}\in X$ if for all %
	 $\lambda>0$, $\lim\limits_{n\rightarrow\infty}\omega_{\lambda} %
	 (T^{n}x_{0},T^{n+1}x_{0},T^{n+1}x_{0})=0$. %
\end{definition}%

\section{Main Results}
	
\begin{theorem}\label{3.1} %
Let $(X,\omega^{G},\preceq)$ be a complete preordered %
 modular G-metric space. %
 Let $T:X_{\omega^{G}}\rightarrow X_{\omega^{G}}$ %
  be a non-decreasing, asymptotically $T$-regular map on some point %
  $x_{0}\in X_{\omega^{G}}$  %
  such that for each $\lambda>0$, there %
 is $\nu(\lambda)\in\lbrack 0,\lambda)$ such that the following %
  conditions hold: %
\begin{align}\label{a} %
(1)~~\psi(\omega^{G}_{\lambda}(Tx,Ty,Ty))\notag%
\le&%
\kappa_{1}(\psi(\omega^{G}_{\lambda}(x,y,y)))%
\psi(\omega^{G}_{\lambda+\nu(\lambda)}(x,y,y))\\ \notag%
+&\kappa_{2}(\psi(\omega^{G}_{\lambda}(x,y,y)))%
\psi(\omega^{G}_{\lambda}(x,Tx,Tx))\\ \notag
+& \kappa_{3}(\psi(\omega^{G}_{\lambda}(x,y,y)))%
\psi(\omega^{G}_{\lambda}(y,Ty,Ty))\\ \notag
+&\kappa_{4}(\psi(\omega^{G}_{\lambda}(x,y,y)))%
\psi(\omega^{G}_{\lambda}(Tx,Ty,Ty))\\ \notag
+&\kappa_{5}(\psi(\omega^{G}_{\lambda}(x,y,y)))%
\psi(\omega^{G}_{\lambda}(Tx,y,y))\\ \notag
+&\kappa_{6}(\psi(\omega^{G}_{\lambda}(x,y,y)))\\\notag %
\times&\psi\Set{\frac{\omega^{G}_{\lambda}(y,Ty,Ty)%
[1+\omega^{G}_{\lambda}(x,y,y)]}{1+\omega^{G}_{\lambda}(x,Tx,Tx)}}\\ \notag%
+&\kappa_{7}(\psi(\omega^{G}_{\lambda}(x,y,y)))\\\notag%
\times&\psi\Set{\frac{\omega^{G}_{\lambda}(y,Ty,Ty)%
[1+\omega^{G}_{\lambda}(x,Tx,Tx)]}{1+\omega^{G}_{\lambda}(x,y,y)}}\\ \notag%
+&\kappa_{8}(\psi(\omega^{G}_{\lambda}(x,y,y)))\\\notag%
\times&\psi\Set{\frac{\omega^{G}_{\lambda}(x,Tx,Tx)%
[1+\omega^{G}_{\lambda}(x,y,y)]}{1+\omega^{G}_{\lambda}(y,Tx,Tx)%
+\omega^{G}_{\lambda}(x,y,y)}}\\ \notag%
+&\kappa_{9}(\psi(\omega^{G}_{\lambda}(x,y,y)))\\\notag%
\times&\psi\Set{\frac{\omega^{G}_{\lambda}(x,y,y)%
[1+\omega^{G}_{\lambda}(y,Ty,Ty)]}{1+\omega^{G}_{\lambda}(Tx,Tx,y)%
+\omega^{G}_{\lambda}(Tx,Ty,Ty)}}\\ \notag%
+&\kappa_{10}(\psi(\omega^{G}_{\lambda}(x,y,y)))\\\notag%
\times&\psi\Set{\frac{\omega^{G}_{\lambda}(Tx,Ty,Ty)%
[1+\omega^{G}_{\lambda}(x,Ty,Ty)]}{1+\omega^{G}_{\lambda}(Tx,Tx,y)%
+\omega^{G}_{\lambda}(x,Ty,Ty)}}\\ \notag%
+&\kappa_{11}(\psi(\omega^{G}_{\lambda}(x,y,y)))\\%
\times&\psi\Set{\frac{\omega^{G}_{\lambda}(Tx,Ty,Ty)%
[1+\omega^{G}_{\lambda}(y,Ty,Ty)+\omega^{G}_{\lambda}(Tx,Tx,y)]}%
{1+\omega^{G}_{\lambda}(Tx,Tx,y)+\omega^{G}_{\lambda}(Tx,Ty,Ty)}},%
\end{align}%

where $\psi\in \bar{\Psi}$ , $\{\kappa_{1},\kappa_{2}, %
\cdots, \kappa_{11}\}\in \mathcal{S}_{Ger}$ and %
$\max\{\sup_{t\geq 0}\kappa_{1}(t),\sup_{t\geq 0}\kappa_{2}(t), %
\cdots, \sup_{t\geq 0}\kappa_{11}(t)\}<1.$ %
Assume that if a non-decreasing sequence $\{x_{n}\}_{n\in\mathbb{N}}$ %
converges to $x^{\ast}$, then $x_{n}\preceq x^{\ast}$ for %
each $n\in\mathbb{N}$,%
\newline%
$(2)$ if $\psi$ is sub-additive and for any $x,y\in X_{\omega^{G}}$,%
there exists $z\in X_{\omega^{G}}$ with $z\preceq Tz$ and %
$\omega^{G}_{\lambda}(z,Tz,Tz)$ is finite  for all $\lambda>0$ %
such that $z$ is comparable to both $x~{\rm and}~y$.%
Then $T$ has a  fixed point  $x^{\ast}\in X_{\omega^{G}}$ %
and the sequence define by $\{T^{n}x_{0}\}_{n\ge 1}$ converges to $x^{\ast}$.%
Moreover the fixed point of $T$ is unique. %
\end{theorem}

\begin{proof} %
(a)
Let $x_{0}\in X_{\omega^{G}}$ be such %
that $x_{0}\preceq Tx_{0}$ and let %
$x_{n}=Tx_{n-1}=T^{n}x_{0}$ for all %
$n\in \mathbb{N}.$ Regarding that $T$ is %
non-decreasing mapping, we have that %
$x_{0}\preceq Tx_{0}=x_{1}$, %
which implies that $x_{1}=Tx_{0} %
\preceq Tx_{1}=x_{2}$. Inductively, %
we have %
\begin{equation}\label{1111}%
x_{0}\preceq x_{1}\preceq x_{2}\preceq \cdots \preceq x_{n-1} %
\preceq x_{n}\preceq x_{n+1}\preceq \cdots. %
\end{equation}%
	
Assume that there exists $n_{0}\in\mathbb{N}$ %
such that $x_{n_{0}}=x_{n_{0}+1}$. Since %
$x_{n_{0}}=x_{n_{0}+1}=Tx_{n_{0}}$, then %
$x_{n_{0}}$ is the fixed point of $T$. %
Now suppose that $x_{n} \precneqq x_{n+1}$ for %
all $n\in\mathbb{N}$, thus  by inequality%
\ref{1111}, we have that %
\begin{equation}\label{22}%
x_{0}\prec x_{1}\prec x_{2}\prec \cdots %
\prec x_{n-1}\prec x_{n}\prec x_{n+1}\prec \cdots. %
\end{equation}%
Now for each $\lambda>0,$ and $x_{0}\prec %
Tx_{0}$ for all $n\in\mathbb{N}$ %
implies that $\omega^{G}_{\lambda}(x_{0},Tx_{0},Tx_{0})>0.$ %
Again, let $x_{0}\in X_{\omega^{G}}$  such that %
$\omega^{G}_{\lambda}(x_{0},Tx_{0},Tx_{0}) %
<\infty~\forall~\lambda>0$.  %
\newline
(b)	
Now, we show  that for all $n\in \mathbb{N}$, the sequence %
$\omega^{G}_{\lambda}(T^{n}x_{0},T^{n+1}x_{0},T^{n+1}x_{0})= 0$ %
for all $\lambda>0$, as  $n\rightarrow\infty$. This means that
the sequence $\{T^{n}x_{0}\}_{n\ge 1}$ is asymptotically $T$-regular %
on some point $x_{0}\in X_{\omega^{G}}$. Assume that, %
for each $n\in\mathbb{N}$, there exists $\lambda_{n}>0$ such that %
$\omega^{G}_{\lambda_{n}}(T^{n}x_{0},T^{n+1}x_{0},T^{n+1}x_{0}) %
\neq 0$.  Otherwise we are done. %
For each $n\ge 1$, if $0<\lambda<\lambda_{n}$, then we have  %
$\omega^{G}_{\lambda}(T^{n}x_{0},T^{n}x_{0},T^{n+1}x_{0})\neq 0$. %
Since $T^{n}x_{0}\preceq T^{n+1}x_{0}$, we have from condition (1) %
of \autoref{3.1},then we have that $\psi(\omega^{G}_{\lambda_{n}}(T^{n}x_{0}, %
T^{n+1}x_{0},T^{n+1}x_{0}))\le \psi(\omega^{G}_{\lambda}(T^{n}x_{0}, %
T^{n+1}x_{0},T^{n+1}x_{0})) %
=\psi(\omega^{G}_{\lambda}(TT^{n-1}x_{0},TT^{n}x_{0},TT^{n}x_{0}))$. %
Take $x=T^{n-1}x_{0}$  and $ y=T^{n}x_{0}$, assume that we are taking  %
$d_{n}=\psi(\omega^{G}_{\lambda} %
(T^{n}x_{0},T^{n+1}x_{0},T^{n+1}x_{0}))$, %
hence we have that %
\begin{align}\label{b}%
d_{n}\notag %
\le& %
\kappa_{1}(\psi(\omega^{G}_{\lambda} %
(T^{n-1}x_{0},T^{n}x_{0},T^{n}x_{0}))) %
\psi(\omega^{G}_{\lambda+\nu(\lambda)} %
(T^{n-1}x_{0},T^{n}x_{0},T^{n}x_{0}))\\\notag %
+&\kappa_{2}(\psi(\omega^{G}_{\lambda} %
(T^{n-1}x_{0},T^{n}x_{0},T^{n}x_{0}))) %
\psi(\omega^{G}_{\lambda} %
(T^{n-1}x_{0},TT^{n-1}x_{0},TT^{n-1}x_{0})) \\\notag %
+&\kappa_{3}(\psi(\omega^{G}_{\lambda} %
(T^{n-1}x_{0},T^{n}x_{0},T^{n}x_{0}))) %
\psi(\omega^{G}_{\lambda} %
(T^{n}x_{0},TT^{n}x_{0},TT^{n}x_{0}))\\ \notag %
+&\kappa_{4}(\psi(\omega^{G}_{\lambda} %
(T^{n-1}x_{0},T^{n}x_{0},T^{n}x_{0}))) %
\psi(\omega^{G}_{\lambda} %
(TT^{n-1}x_{0},TT^{n}x_{0},TT^{n}x_{0}))\\ \notag %
+&\kappa_{5}(\psi(\omega^{G}_{\lambda} %
(T^{n-1}x_{0},T^{n}x_{0},T^{n}x_{0}))) %
\psi(\omega^{G}_{\lambda} %
(TT^{n-1}x_{0},T^{n}x_{0},T^{n}x_{0}))\\ \notag %
+&\kappa_{6}(\psi(\omega^{G}_{\lambda} %
(T^{n-1}x_{0},T^{n}x_{0},T^{n}x_{0})))\\\notag %
\times&\psi\Set{\frac{\omega^{G}_{\lambda} %
(T^{n}x_{0},TT^{n}x_{0},TT^{n}x_{0}) %
[1+\omega^{G}_{\lambda} %
(T^{n-1}x_{0},T^{n}x_{0},T^{n}x_{0})]} %
{1+\omega^{G}_{\lambda} %
(T^{n-1}x_{0},TT^{n-1}x_{0},TT^{n-1}x_{0})}}\\ \notag %
+&\kappa_{7}(\psi(\omega^{G}_{\lambda} %
(T^{n-1}x_{0},T^{n}x_{0},T^{n}x_{0})))\\\notag %
\times&\psi\Set{\frac{\omega^{G}_{\lambda} %
(T^{n}x_{0},TT^{n}x_{0},TT^{n}x_{0}) %
[1+\omega^{G}_{\lambda} %
(T^{n-1}x_{0},TT^{n-1}x_{0},TT^{n-1}x_{0})]} %
{1+\omega^{G}_{\lambda} %
(T^{n-1}x_{0},T^{n}x_{0},T^{n}x_{0})}}\\\notag %
+&\kappa_{8}(\psi(\omega^{G}_{\lambda} %
(T^{n-1}x_{0},T^{n}x_{0},T^{n}x_{0})))\\\notag %
\times&\psi\Set{\frac{\omega^{G}_{\lambda} %
(T^{n-1}x_{0},TT^{n-1}x_{0},TT^{n-1}x_{0}) %
[1+\omega^{G}_{\lambda} %
(T^{n-1}x_{0},T^{n}x_{0},T^{n}x_{0})]} %
{1+\omega^{G}_{\lambda} %
(T^{n}x_{0},TT^{n-1}x_{0},TT^{n-1}x_{0}) %
+\omega^{G}_{\lambda} %
(T^{n-1}x_{0},T^{n}x_{0},T^{n}x_{0})}}\\ \notag %
+&\kappa_{9}(\psi(\omega^{G}_{\lambda} %
(T^{n-1}x_{0},T^{n}x_{0},T^{n}x_{0})))\\\notag %
\times&\psi\Set{\frac{\omega^{G}_{\lambda} %
(T^{n-1}x_{0},T^{n}x_{0},T^{n}x_{0}) %
[1+\omega^{G}_{\lambda} %
(T^{n}x_{0},TT^{n}x_{0},TT^{n}x_{0})]} %
{1+\omega^{G}_{\lambda} %
(TT^{n-1}x_{0},TT^{n-1}x_{0},T^{n}x_{0}) %
+\omega^{G}_{\lambda} %
(TT^{n-1}x_{0},TT^{n}x_{0},TT^{n}x_{0})}}\\ \notag %
+&\kappa_{10}(\psi(\omega^{G}_{\lambda} %
(T^{n-1}x_{0},T^{n}x_{0},T^{n}x_{0})))\\\notag %
\times&\psi\Set{\frac{\omega^{G}_{\lambda} %
(TT^{n-1}x_{0},TT^{n}x_{0},TT^{n}x_{0}) %
[1+\omega^{G}_{\lambda} %
(T^{n-1}x_{0},TT^{n}x_{0},TT^{n}x_{0})]} %
{1+\omega^{G}_{\lambda} %
(TT^{n-1}x_{0},TT^{n-1}x_{0},T^{n}x_{0}) %
+\omega^{G}_{\lambda} %
(T^{n-1}x_{0},TT^{n}x_{0},TT^{n}x_{0})}} \\\notag	 %
+&\kappa_{11}(\psi(\omega^{G}_{\lambda} %
(T^{n-1}x_{0},T^{n}x_{0},T^{n}x_{0}))) %
\times\notag
\end{align}%
\begin{equation}%
\psi\Set{\frac{\omega^{G}_{\lambda} %
(TT^{n-1}x_{0},TT^{n}x_{0},TT^{n}x_{0}) %
[1+\omega^{G}_{\lambda} %
(T^{n}x_{0},TT^{n}x_{0},TT^{n}x_{0}) %
+\omega^{G}_{\lambda} %
(TT^{n-1}x_{0},TT^{n-1}x_{0},T^{n}x_{0})]} %
{1+\omega^{G}_{\lambda} %
(TT^{n-1}x_{0},TT^{n-1}x_{0},T^{n}x_{0}) %
+\omega^{G}_{\lambda} %
(TT^{n-1}x_{0},TT^{n}x_{0},TT^{n}x_{0})}}. %
\end{equation}%

 We get, %

\begin{align}\label{c}%
d_{n}\notag %
\le& %
\kappa_{1}\big(\psi\big(\omega^{G}_{\lambda} %
\big(T^{n-1}x_{0},T^{n}x_{0},T^{n}x_{0}\big)\big)\big) %
\psi\big(\omega^{G}_{\lambda+\nu(\lambda)} %
\big(T^{n-1}x_{0},T^{n}x_{0},T^{n}x_{0}))\\ \notag %
+&\kappa_{2}\big(\psi\big(\omega^{G}_{\lambda} %
\big(T^{n-1}x_{0},T^{n}x_{0},T^{n}x_{0}))) %
\psi\big(\omega^{G}_{\lambda} %
\big(T^{n-1}x_{0},T^{n}x_{0},T^{n}x_{0}))\\ \notag %
+&\kappa_{3}\big(\psi\big(\omega^{G}_{\lambda} %
\big(T^{n-1}x_{0},T^{n}x_{0},T^{n}x_{0}))) %
\psi\big(\omega^{G}_{\lambda} %
\big(T^{n}x_{0},T^{n+1}x_{0},T^{n+1}x_{0}))\\ \notag %
+&\kappa_{4}\big(\psi\big(\omega^{G}_{\lambda} %
\big(T^{n-1}x_{0},T^{n}x_{0},T^{n}x_{0}))) %
\psi\big(\omega^{G}_{\lambda} %
\big(T^{n}x_{0},T^{n+1}x_{0},T^{n+1}x_{0}))\\ \notag %
+&\kappa_{5}\big(\psi\big(\omega^{G}_{\lambda} %
\big(T^{n-1}x_{0},T^{n}x_{0},T^{n}x_{0}))) %
\psi\big(\omega^{G}_{\lambda} %
\big(T^{n}x_{0},T^{n}x_{0},T^{n}x_{0}))\\\notag %
+&\kappa_{6}\big(\psi\big(\omega^{G}_{\lambda} %
\big(T^{n-1}x_{0},T^{n}x_{0},T^{n}x_{0})))\\\notag %
\times&\psi\Set{\frac{\omega^{G}_{\lambda} %
\big(T^{n}x_{0},T^{n+1}x_{0},T^{n+1}x_{0}) %
\big[1+\omega^{G}_{\lambda} %
\big(T^{n-1}x_{0},T^{n}x_{0},T^{n}x_{0})]} %
{1+\omega^{G}_{\lambda} %
\big(T^{n-1}x_{0},T^{n}x_{0},T^{n}x_{0})}}\\ \notag %
+&\kappa_{7}\big(\psi\big(\omega^{G}_{\lambda} %
\big(T^{n-1}x_{0},T^{n}x_{0},T^{n}x_{0})))\\\notag %
\times&\psi\Set{\frac{\omega^{G}_{\lambda} %
\big(T^{n}x_{0},T^{n+1}x_{0},T^{n+1}x_{0}) %
\big[1+\omega^{G}_{\lambda} %
\big(T^{n-1}x_{0},T^{n}x_{0},T^{n}x_{0})]} %
{1+\omega^{G}_{\lambda} %
\big(T^{n-1}x_{0},T^{n}x_{0},T^{n}x_{0})}}\\\notag %
+&\kappa_{8}\big(\psi\big(\omega^{G}_{\lambda} %
\big(T^{n-1}x_{0},T^{n}x_{0},T^{n}x_{0})))\\\notag %
\times&\psi\Set{\frac{\omega^{G}_{\lambda} %
\big(T^{n-1}x_{0},T^{n}x_{0},T^{n}x_{0}) %
\big[1+\omega^{G}_{\lambda} %
\big(T^{n-1}x_{0},T^{n}x_{0},T^{n}x_{0})]} %
{1+\omega^{G}_{\lambda} %
\big(T^{n}x_{0},T^{n}x_{0},T^{n}x_{0}) %
+\omega^{G}_{\lambda} %
\big(T^{n-1}x_{0},T^{n}x_{0},T^{n}x_{0})}}\\ \notag %
+&\kappa_{9}\big(\psi\big(\omega^{G}_{\lambda} %
\big(T^{n-1}x_{0},T^{n}x_{0},T^{n}x_{0})))\\\notag %
\times&\psi\Set{\frac{\omega^{G}_{\lambda} %
\big(T^{n-1}x_{0},T^{n}x_{0},T^{n}x_{0}) %
\big[1+\omega^{G}_{\lambda} %
\big(T^{n}x_{0},T^{n+1}x_{0},T^{n+1}x_{0})]} %
{1+\omega^{G}_{\lambda}\big(T^{n}x_{0},T^{n}x_{0},T^{n}x_{0}) %
+\omega^{G}_{\lambda} %
\big(T^{n}x_{0},T^{n+1}x_{0},T^{n+1}x_{0})}}\\ \notag %
+&\kappa_{10}\big(\psi\big(\omega^{G}_{\lambda} %
\big(T^{n-1}x_{0},T^{n}x_{0},T^{n}x_{0})))\\\notag %
\times&\psi\Set{\frac{\omega^{G}_{\lambda} %
\big(T^{n}x_{0},T^{n+1}x_{0},T^{n+1}x_{0}) %
\big[1+\omega^{G}_{\lambda} %
\big(T^{n-1}x_{0},T^{n+1}x_{0},T^{n+1}x_{0})]} %
{1+\omega^{G}_{\lambda} %
\big(T^{n}x_{0},T^{n}x_{0},T^{n}x_{0}) %
+\omega^{G}_{\lambda} %
\big(T^{n-1}x_{0},T^{n+1}x_{0},T^{n+1}x_{0})}}\\ \notag %
+&\kappa_{11}\big(\psi\big(\omega^{G}_{\lambda} %
\big(T^{n-1}x_{0},T^{n}x_{0},T^{n}x_{0})))\notag %
\end{align}%
\begin{equation}%
\qquad\times\psi\Set{\frac{\omega^{G}_{\lambda} %
\big(T^{n}x_{0},T^{n+1}x_{0},T^{n+1}x_{0}) %
\big[1+\omega^{G}_{\lambda} %
\big(T^{n}x_{0},T^{n+1}x_{0},T^{n+1}x_{0}) %
+\omega^{G}_{\lambda} %
\big(T^{n}x_{0},T^{n}x_{0},T^{n}x_{0})]} %
{1+\omega^{G}_{\lambda} %
\big(T^{n}x_{0},T^{n}x_{0},T^{n}x_{0}) %
+\omega^{G}_{\lambda} %
\big(T^{n}x_{0},T^{n+1}x_{0},T^{n+1}x_{0})}}. %
\end{equation}%

So applying property 1 of Definition \ref{2.1}, we have %

\begin{align}%
\psi(\omega^{G}_{\lambda} %
(T^{n}x_{0}, T^{n+1}x_{0},T^{n+1}x_{0}))\notag %
\le& %
\kappa_{1}(\psi(\omega^{G}_{\lambda} %
(T^{n-1}x_{0}, T^{n}x_{0},T^{n}x_{0})))\\\notag %
\times&\psi(\omega^{G}_{\lambda+\nu(\lambda)} %
(T^{n-1}x_{0}, T^{n}x_{0},T^{n}x_{0}))\\\notag %
+&\kappa_{2}(\psi(\omega^{G}_{\lambda} %
(T^{n-1}x_{0},T^{n}x_{0},T^{n}x_{0})))\\\notag %
\times&\psi(\omega^{G}_{\lambda} %
(T^{n-1}x_{0}, T^{n}x_{0},T^{n}x_{0}))\\\notag %
+&\kappa_{3}(\psi(\omega^{G}_{\lambda} %
(T^{n-1}x_{0},T^{n}x_{0},T^{n}x_{0})))\\\notag %
\times&\psi(\omega^{G}_{\lambda} %
(T^{n}x_{0}, T^{n+1}x_{0},T^{n+1}x_{0}))\\\notag %
+&\kappa_{4}(\psi(\omega^{G}_{\lambda} %
(T^{n-1}x_{0},T^{n}x_{0},T^{n}x_{0})))\\\notag %
\times&\psi(\omega^{G}_{\lambda} %
(T^{n}x_{0},T^{n+1}x_{0},T^{n+1}x_{0}))\\\notag %
+&\kappa_{6}(\psi(\omega^{G}_{\lambda} %
(T^{n-1}x_{0},T^{n}x_{0},T^{n}x_{0})))\\\notag %
\times&\psi(\omega^{G}_{\lambda} %
(T^{n}x_{0}, T^{n+1}x_{0},T^{n+1}x_{0}))\\\notag %
+&\kappa_{7}(\psi(\omega^{G}_{\lambda} %
(T^{n-1}x_{0},T^{n}x_{0},T^{n}x_{0})))\\\notag %
\times&\psi(\omega^{G}_{\lambda} %
(T^{n}x_{0}, T^{n+1}x_{0},T^{n+1}x_{0}))\\\notag %
+&\kappa_{8}(\psi(\omega^{G}_{\lambda} %
(T^{n-1}x_{0},T^{n}x_{0},T^{n}x_{0})))\\\notag %
\times&\psi(\omega^{G}_{\lambda} %
(T^{n-1}x_{0}, T^{n}x_{0},T^{n}x_{0}))\\\notag %
+&\kappa_{9}(\psi(\omega^{G}_{\lambda} %
(T^{n-1}x_{0},T^{n}x_{0},T^{n}x_{0})))\\\notag %
\times&\psi(\omega^{G}_{\lambda} %
(T^{n-1}x_{0}, T^{n}x_{0},T^{n}x_{0}))\\\notag %
+&\kappa_{10}(\psi(\omega^{G}_{\lambda} %
(T^{n-1}x_{0},T^{n}x_{0},T^{n}x_{0})))\\\notag %
\times&\psi(\omega^{G}_{\lambda} %
(T^{n}x_{0}, T^{n+1}x_{0},T^{n+1}x_{0}))\\\notag %
+&\kappa_{11}(\psi(\omega^{G}_{\lambda} %
(T^{n-1}x_{0},T^{n}x_{0},T^{n}x_{0})))\\ %
\times&\psi(\omega^{G}_{\lambda} %
(T^{n}x_{0}, T^{n+1}x_{0},T^{n+1}x_{0})), %
\end{align}%

so that %

\begin{align}\label{3.7}%
\psi(\omega^{G}_{\lambda} %
(T^{n}x_{0}, T^{n+1}x_{0},T^{n+1}x_{0}))\notag %
\le& %
\kappa_{1}(\psi(\omega^{G}_{\lambda} %
(T^{n-1}x_{0}, T^{n}x_{0},T^{n}x_{0})))\\\notag %
\times&\psi(\omega^{G}_{\lambda} %
(T^{n-1}x_{0}, T^{n}x_{0},T^{n}x_{0}))\\\notag %
+&\kappa_{2}(\psi(\omega^{G}_{\lambda} %
(T^{n-1}x_{0},T^{n}x_{0},T^{n}x_{0})))\\\notag %
\times&\psi(\omega^{G}_{\lambda} %
(T^{n-1}x_{0}, T^{n}x_{0},T^{n}x_{0}))\\\notag %
+&\kappa_{3}(\psi(\omega^{G}_{\lambda} %
(T^{n-1}x_{0},T^{n}x_{0},T^{n}x_{0})))\\\notag %
\times&\psi(\omega^{G}_{\lambda} %
(T^{n}x_{0}, T^{n+1}x_{0},T^{n+1}x_{0}))\\\notag %
+&\kappa_{4}(\psi(\omega^{G}_{\lambda} %
(T^{n-1}x_{0},T^{n}x_{0},T^{n}x_{0})))\\\notag %
\times&\psi(\omega^{G}_{\lambda} %
(T^{n}x_{0},T^{n+1}x_{0},T^{n+1}x_{0}))\\\notag %
+&\kappa_{6}(\psi(\omega^{G}_{\lambda} %
(T^{n-1}x_{0},T^{n}x_{0},T^{n}x_{0})))\\\notag %
\times&\psi(\omega^{G}_{\lambda} %
(T^{n}x_{0}, T^{n+1}x_{0},T^{n+1}x_{0}))\\\notag %
+&\kappa_{7}(\psi(\omega^{G}_{\lambda} %
(T^{n-1}x_{0},T^{n}x_{0},T^{n}x_{0})))\\\notag %
\times&\psi(\omega^{G}_{\lambda} %
(T^{n}x_{0}, T^{n+1}x_{0},T^{n+1}x_{0}))\\\notag %
+&\kappa_{8}(\psi(\omega^{G}_{\lambda} %
(T^{n-1}x_{0},T^{n}x_{0},T^{n}x_{0})))\\\notag %
\times&\psi(\omega^{G}_{\lambda} %
(T^{n-1}x_{0}, T^{n}x_{0},T^{n}x_{0}))\\\notag %
+&\kappa_{9}(\psi(\omega^{G}_{\lambda} %
(T^{n-1}x_{0},T^{n}x_{0},T^{n}x_{0})))\\\notag %
\times&\psi(\omega^{G}_{\lambda} %
(T^{n-1}x_{0}, T^{n}x_{0},T^{n}x_{0}))\\\notag %
+&\kappa_{10}(\psi(\omega^{G}_{\lambda} %
(T^{n-1}x_{0},T^{n}x_{0},T^{n}x_{0})))\\\notag %
\times&\psi(\omega^{G}_{\lambda} %
(T^{n}x_{0}, T^{n+1}x_{0},T^{n+1}x_{0}))\\\notag %
+&\kappa_{11}(\psi(\omega^{G}_{\lambda} %
(T^{n-1}x_{0},T^{n}x_{0},T^{n}x_{0})))\\ %
\times&\psi(\omega^{G}_{\lambda} %
(T^{n}x_{0}, T^{n+1}x_{0},T^{n+1}x_{0})). %
\end{align}%
		
Thus, %

\begin{eqnarray}\notag %
&&\bigg(1-\bigg(\kappa_{3}(\psi(\omega^{G}_{\lambda} %
(T^{n-1}x_{0},T^{n}x_{0},T^{n}x_{0})))
+\kappa_{4}(\psi(\omega^{G}_{\lambda} %
(T^{n-1}x_{0},T^{n}x_{0},T^{n}x_{0})))\\\notag %
&&+\kappa_{6}(\psi(\omega^{G}_{\lambda} %
(T^{n-1}x_{0},T^{n}x_{0},T^{n}x_{0})))\notag %
+\kappa_{7}(\psi(\omega^{G}_{\lambda} %
(T^{n-1}x_{0},T^{n}x_{0},T^{n}x_{0})))\\\notag %
&&+\kappa_{10}(\psi(\omega^{G}_{\lambda} %
(T^{n-1}x_{0},T^{n}x_{0},T^{n}x_{0}))) %
+\kappa_{11}(\psi(\omega^{G}_{\lambda} %
(T^{n-1}x_{0},T^{n}x_{0},T^{n}x_{0})))\bigg) %
\bigg)\\\notag%
&&\times\psi(\omega^{G}_{\lambda} %
(T^{n}x_{0}, T^{n+1}x_{0},T^{n+1}x_{0}))\\\notag %
&&\le\notag %
\bigg(\kappa_{1}(\psi(\omega^{G}_{\lambda} %
(T^{n-1}x_{0},T^{n}x_{0},T^{n}x_{0})))\notag %
+\kappa_{2}(\psi(\omega^{G}_{\lambda} %
(T^{n-1}x_{0},T^{n}x_{0},T^{n}x_{0})))\\\notag %
&&+\kappa_{8}(\psi(\omega^{G}_{\lambda} %
(T^{n-1}x_{0},T^{n}x_{0},T^{n}x_{0}))) %
+\kappa_{9}(\psi(\omega^{G}_{\lambda} %
(T^{n-1}x_{0},T^{n}x_{0},T^{n}x_{0}))) %
\bigg)\\&&\times\psi(\omega^{G}_{\lambda} %
(T^{n-1}x_{0}, T^{n}x_{0},T^{n}x_{0})). %
\end{eqnarray}%

Therefore, %

\begin{align}\label{3.9}%
\psi(\omega^{G}_{\lambda} %
(T^{n}x_{0}, T^{n+1}x_{0},T^{n+1}x_{0}))\notag %
\le& \rho \psi(\omega^{G}_{\lambda} %
(T^{n-1}x_{0}, T^{n}x_{0},T^{n}x_{0}))\\\notag %
\le&\psi(\omega^{G}_{\lambda} %
(T^{n-1}x_{0}, T^{n}x_{0},T^{n}x_{0}))\\\notag %
\le&\psi(\omega^{G}_{\lambda} %
(T^{n-2}x_{0}, T^{n-1}x_{0},T^{n-1}x_{0}))\\\notag %
\vdots&\\\notag %
\le&\psi(\omega^{G}_{\lambda}(x_{0}, Tx_{0},Tx_{0}))\\ %
< &\infty, %
\end{align}%

where $\rho=\dfrac{\rho_{1}}{\rho_{2}},$ where, %

\begin{align}%
\rho_{1}:=\notag %
&1-\big(\kappa_{3}(\psi(\omega^{G}_{\lambda} %
(T^{n-1}x_{0},T^{n}x_{0},T^{n}x_{0})))\notag %
+\kappa_{4}(\psi(\omega^{G}_{\lambda} %
(T^{n-1}x_{0},T^{n}x_{0},T^{n}x_{0})))\\\notag %
&+\kappa_{6}(\psi(\omega^{G}_{\lambda} %
(T^{n-1}x_{0},T^{n}x_{0},T^{n}x_{0})))\notag %
+\kappa_{7}(\psi(\omega^{G}_{\lambda} %
(T^{n-1}x_{0},T^{n}x_{0},T^{n}x_{0})))\\\notag %
&+\kappa_{10}(\psi(\omega^{G}_{\lambda} %
(T^{n-1}x_{0},T^{n}x_{0},T^{n}x_{0}))) %
+\kappa_{11}(\psi(\omega^{G}_{\lambda} %
(T^{n-1}x_{0},T^{n}x_{0},T^{n}x_{0}))), %
\end{align}%

and

\begin{align}%
\rho_{2}:=\notag %
&\kappa_{1}(\psi(\omega^{G}_{\lambda} %
(T^{n-1}x_{0},T^{n}x_{0},T^{n}x_{0})))\notag %
+\kappa_{2}(\psi(\omega^{G}_{\lambda} %
(T^{n-1}x_{0},T^{n}x_{0},T^{n}x_{0})))\\\notag %
&+\kappa_{8}(\psi(\omega^{G}_{\lambda} %
(T^{n-1}x_{0},T^{n}x_{0},T^{n}x_{0}))) %
+\kappa_{9}(\psi(\omega^{G}_{\lambda} %
(T^{n-1}x_{0},T^{n}x_{0},T^{n}x_{0}))). %
\end{align}%

Hence $\{\psi(\omega_{\lambda} %
(T^{n}x_{0}, T^{n+1}x_{0},T^{n+1}x_{0}))\}_{n\ge 1}$ %
is non-increasing and bounded below for which the  %
sequence converges to some real number %
$\gamma\ge 0$ (say); that is $\gamma\in [0,\infty)$.  %
We claim that $\gamma=0$. %
Suppose not, then  $\gamma>0$, then  for any $\lambda>0$, %
using inequalities \ref{3.9} and \ref{3.7},  %
we can see clearly that %
\begin{align}%
\psi(\omega^{G}_{\lambda} %
(T^{n}x_{0}, T^{n+1}x_{0},T^{n+1}x_{0}))\notag %
&\le %
\bigg(\kappa_{1}\big(\psi\big( %
\omega^{G}_{\lambda}\big(T^{n-1}x_{0}, T^{n}x_{0}, %
T^{n}x_{0}\big)\big)\big)\notag %
+ \kappa_{2}\big(\psi\big(\omega^{G}_{\lambda} %
\big(T^{n-1}x_{0},T^{n}x_{0}, %
T^{n}x_{0}\big)\big)\big)\\\notag %
&+ \kappa_{3}\big(\psi\big(\omega^{G}_{\lambda} %
\big(T^{n-1}x_{0},T^{n}x_{0}, %
T^{n}x_{0}\big)\big)\big)\notag
+\kappa_{4} %
\big(\psi\big(\omega^{G}_{\lambda} %
\big(T^{n-1}x_{0},T^{n}x_{0},T^{n}x_{0}\big)\big)\big)\\\notag %
&+\kappa_{6}\big(\psi\big(\omega^{G}_{\lambda} %
\big(T^{n-1}x_{0},T^{n}x_{0},T^{n}x_{0}\big)\big)\big)\\\notag %
&+\sum_{k=7}^{11}\bigg(\kappa_{k}\big(\psi\big( %
\omega^{G}_{\lambda} %
\big(T^{n-1}x_{0},T^{n}x_{0},T^{n}x_{0}\big)\big)\big)\bigg)\bigg)\\ %
&\times\psi(\omega^{G}_{\lambda} %
(T^{n-1}x_{0},T^{n}x_{0},T^{n}x_{0})).\\\notag %
\end{align}%
So that %
\begin{align}
\psi(\omega_{\lambda}(T^{n}x_{0}, T^{n+1}x_{0},T^{n+1}x_{0})) %
&\le\notag %
\sum_{k=1}^{4}\kappa_{k}\big(\psi\big(\omega^{G}_{\lambda} %
\big(T^{n-1}x_{0},T^{n}x_{0},T^{n}x_{0}\big)\big)\big)\\\notag  %
&+\sum_{k=6}^{11}\kappa_{k}\big(\psi\big(\omega^{G}_{\lambda} %
\big(T^{n-1}x_{0},T^{n}x_{0},T^{n}x_{0}\big)\big)\big)\\ %
&\times\psi(\omega^{G}_{\lambda}(T^{n-1}x_{0},T^{n}x_{0},T^{n}x_{0})), %
\end{align}%

as $n\rightarrow\infty$, %

\begin{equation}%
1\le\liminf_{n\rightarrow\infty}\biggl(\sum_{k=1}^{4} %
\kappa_{k}\big(\psi\big(\omega^{G}_{\lambda} %
\big(T^{n-1}x_{0},T^{n}x_{0},T^{n}x_{0}\big)\big)\big)  %
+\sum_{k=6}^{11}\kappa_{k}\big(\psi\big(\omega^{G}_{\lambda} %
\big(T^{n-1}x_{0},T^{n}x_{0},T^{n}x_{0}\big)\big)\big)\biggr). %
\end{equation}%
	
So, we have that %

\begin{equation}%
\lim\limits_{n\rightarrow\infty} %
\psi(\omega^{G}_{\lambda} %
(T^{n-1}x_{0},T^{n}x_{0},T^{n}x_{0})) %
=0~\forall ~\lambda>0, %
\end{equation}	%

hence, %

\begin{equation}%
\lim\limits_{n\rightarrow\infty}\omega^{G}_{\lambda} %
(T^{n-1}x_{0},T^{n}x_{0},T^{n}x_{0})=0~~  %
\forall ~\lambda>0. %
\end{equation}%
	
This is a  contradiction to our claim. %
Hence, %

\begin{equation}%
\lim\limits_{n\rightarrow\infty} %
\psi(\omega^{G}_{\lambda} %
(T^{n}x_{0},T^{n+1}x_{0},T^{n+1}x_{0})) %
=0~ \forall ~\lambda>0, %
\end{equation}%

moreover, we have %

\begin{equation}\label{A}%
\lim\limits_{n\rightarrow\infty}\omega^{G}_{\lambda} %
(T^{n}x_{0},T^{n+1}x_{0},T^{n+1}x_{0})=0 ~\forall~\lambda>0. %
\end{equation}%

This shows that the sequence %
$\{T^{n}x_{0}\}_{n\in\mathbb{N}}$ is asymptotically $T$-regular
on some point $x_{0}\in X_{\omega^{G}}$. %
\newline
(c)	
We now show that $\{T^{n}x_{0}\}_{n\in\mathbb{N}}$ %
is a modular G-Cauchy sequence in $X_{\omega^{G}}$. %
By contradiction, suppose, if possible, that there %
 exists real number, %
$\epsilon>0$ and two modular G-sub-sequences %
$\{T^{m_{k}}x_{0}\}_{k\in\mathbb{N}}$ %
and $\{T^{n_{k}}x_{0}\}_{k\in\mathbb{N}}$ of the modular G-sequence %
$\{T^{n}x_{0}\}_{n\in\mathbb{N}}$ %
such that for $\lambda_{0}>0$ and $n_{k}>m_{k}>k$,  %
we have that $\omega^{G}_{\lambda_{0}} %
(T^{m_{k}}x_{0},T^{n_{k}}x_{0},T^{n_{k}}x_{0}) %
\ge6\epsilon$, but %
$\omega^{G}_{\lambda_{0}} %
(T^{m_{k}}x_{0},T^{n_{k}-1}x_{0},T^{n_{k}-1}x_{0}) %
<\epsilon$. %
Now, $T^{m_{k}}x_{0}\preceq T^{n_{k}}x_{0}$, %
we have $6\epsilon\le \omega^{G}_{\lambda_{0}} %
(T^{m_{k}}x_{0},T^{n_{k}}x_{0},T^{n_{k}}x_{0})$ %
which implies that  $6\psi(\epsilon)\le\psi( %
\omega^{G}_{\lambda_{0}} %
(T^{m_{k}}x_{0},T^{n_{k}}x_{0},T^{n_{k}}x_{0})) %
=\psi(\omega^{G}_{\lambda_{0}} %
(TT^{m_{k}-1}x_{0},TT^{n_{k}-1}x_{0},TT^{n_{k}-1}x_{0}))$. %
Set $x=T^{m_{k}-1}x_{0}$ and $y=T^{n_{k}-1}x_{0}$  %
into inequality \ref{a}, then we have %
\begin{align}\label{d}%
6\psi(\epsilon)\notag %
&\le %
\psi(\omega^{G}_{\lambda_{0}} %
(T^{m_{k}}x_{0},T^{n_{k}}x_{0},T^{n_{k}}x_{0}))\\\notag %
&\le %
\kappa_{1}(\psi(\omega^{G}_{\lambda_{0}} %
(T^{m_{k}-1}x_{0},T^{n_{k}-1}x_{0},T^{n_{k}-1}x_{0}))) %
\psi(\omega^{G}_{\lambda_{0}+\nu(\lambda_{0})} %
(T^{m_{k}-1}x_{0},T^{n_{k}-1}x_{0},T^{n_{k}-1}x_{0}))\\\notag %
&+\kappa_{2}(\psi(\omega^{G}_{\lambda_{0}} %
(T^{m_{k}-1}x_{0},T^{n_{k}-1}x_{0},T^{n_{k}-1}x_{0}))) %
\psi(\omega^{G}_{\lambda_{0}} %
(T^{m_{k}-1}x_{0},TT^{m_{k}-1}x_{0},TT^{m_{k}-1}x_{0}))\\ \notag %
&+\kappa_{3}(\psi(\omega^{G}_{\lambda_{0}} %
(T^{m_{k}-1}x_{0},T^{n_{k}-1}x_{0},T^{n_{k}-1}x_{0}))) %
\psi(\omega^{G}_{\lambda_{0}} %
(T^{n_{k}-1}x_{0},TT^{n_{k}-1}x_{0},TT^{n_{k}-1}x_{0}))\\\notag %
&+\kappa_{4}(\psi(\omega^{G}_{\lambda_{0}} %
(T^{m_{k}-1}x_{0},T^{n_{k}-1}x_{0},T^{n_{k}-1}x_{0}))) %
\psi(\omega^{G}_{\lambda_{0}} %
(TT^{m_{k}-1}x_{0},TT^{n_{k}-1}x_{0},TT^{n_{k}-1}x_{0}))\\ \notag %
&+\kappa_{5}(\psi(\omega^{G}_{\lambda_{0}} %
(T^{m_{k}-1}x_{0},T^{n_{k}-1}x_{0},T^{n_{k}-1}x_{0}))) %
\psi(\omega^{G}_{\lambda_{0}} %
(TT^{m_{k}-1}x_{0},T^{n_{k}-1}x_{0},T^{n_{k}-1}x_{0}))\\\notag %
&+\kappa_{6}(\psi(\omega^{G}_{\lambda_{0}} %
(T^{m_{k}-1}x_{0},T^{n_{k}-1}x_{0},T^{n_{k}-1}x_{0})))\\\notag %
&\times\psi\Set{\frac{\omega^{G}_{\lambda_{0}} %
(T^{n_{k}-1}x_{0},TT^{n_{k}-1}x_{0},TT^{n_{k}-1}x_{0}) %
[1+\omega^{G}_{\lambda_{0}} %
(T^{m_{k}-1}x_{0},T^{n_{k}-1}x_{0},T^{n_{k}-1}x_{0})]} %
{1+\omega^{G}_{\lambda_{0}} %
(T^{m_{k}-1}x_{0},TT^{m_{k}-1}x_{0},TT^{m_{k}-1}x_{0})}}\\ \notag %
&+\kappa_{7}(\psi(\omega^{G}_{\lambda_{0}} %
(T^{m_{k}-1}x_{0},T^{n_{k}-1}x_{0},T^{n_{k}-1}x_{0})))\\\notag %
&\times\psi\Set{\frac{\omega^{G}_{\lambda_{0}} %
(T^{n_{k}-1}x_{0},TT^{n_{k}-1}x_{0},TT^{n_{k}-1}x_{0}) %
[1+\omega^{G}_{\lambda_{0}} %
(T^{m_{k}-1}x_{0},TT^{m_{k}-1}x_{0},TT^{m_{k}-1}x_{0})]} %
{1+\omega^{G}_{\lambda_{0}} %
(T^{m_{k}-1}x_{0},T^{n_{k}-1}x_{0},T^{n_{k}-1}x_{0})}}\\ \notag %
&+\kappa_{8}(\psi(\omega^{G}_{\lambda_{0}} %
(T^{m_{k}-1}x_{0},T^{n_{k}-1}x_{0},T^{n_{k}-1}x_{0})))\\\notag %
&\times\psi\Set{\frac{\omega^{G}_{\lambda_{0}} %
(T^{m_{k}-1}x_{0},TT^{m_{k}-1}x_{0},TT^{m_{k}-1}x_{0}) %
[1+\omega^{G}_{\lambda_{0}} %
(T^{m_{k}-1}x_{0},T^{n_{k}-1}x_{0},T^{n_{k}-1}x_{0})]} %
{1+\omega^{G}_{\lambda_{0}} %
(T^{n_{k}-1}x_{0},TT^{m_{k}-1}x_{0},TT^{m_{k}-1}x_{0}) %
+\omega^{G}_{\lambda_{0}} %
(T^{m_{k}-1}x_{0},T^{n_{k}-1}x_{0},T^{n_{k}-1}x_{0})}}\\ \notag %
&+\kappa_{9}(\psi(\omega^{G}_{\lambda_{0}} %
(T^{m_{k}-1}x_{0},T^{n_{k}-1}x_{0},T^{n_{k}-1}x_{0})))\\\notag %
&\times\psi\Set{\frac{\omega^{G}_{\lambda_{0}} %
(T^{m_{k}-1}x_{0},T^{n_{k}-1}x_{0},T^{n_{k}-1}x_{0}) %
[1+\omega^{G}_{\lambda_{0}} %
(T^{n_{k}-1}x_{0},TT^{n_{k}-1}x_{0},TT^{n_{k}-1}x_{0})]} %
{1+\omega^{G}_{\lambda_{0}} %
(TT^{m_{k}-1}x_{0},TT^{m_{k}-1}x_{0},T^{n_{k}-1}x_{0}) %
+\omega^{G}_{\lambda_{0}} %
(TT^{m_{k}-1}x_{0},TT^{n_{k}-1}x_{0},TT^{n_{k}-1}x_{0})}}\\ \notag %
&+\kappa_{10}(\psi(\omega^{G}_{\lambda_{o}} %
(T^{m_{k}-1}x_{0},T^{n_{k}-1}x_{0},T^{n_{k}-1}x_{0})))\\\notag %
&\times\psi\Set{\frac{\omega^{G}_{\lambda_{0}} %
(TT^{m_{k}-1}x_{0},TT^{n_{k}-1}x_{0},TT^{n_{k}-1}x_{0}) %
[1+\omega^{G}_{\lambda_{0}} %
(T^{m_{k}-1}x_{0},TT^{n_{k}-1}x_{0},TT^{n_{k}-1}x_{0})]} %
{1+\omega^{G}_{\lambda_{0}} %
(TT^{m_{k}-1}x_{0},TT^{m_{k}-1}x_{0},T^{n_{k}-1}x_{0}) %
+\omega^{G}_{\lambda_{0}} %
(T^{m_{k}-1}x_{0},TT^{n_{k}-1}x_{0},TT^{n_{k}-1}x_{0})}} \\ %
&+\kappa_{11}(\psi(\omega^{G}_{\lambda_{0}} %
(T^{m_{k}-1}x_{0},T^{n_{k}-1}x_{0},T^{n_{k}-1}x_{0}))) %
\end{align}
$ \times\psi\Set{\frac{\omega^{G}_{\lambda_{0}} %
(TT^{m_{k}-1}x_{0},TT^{n_{k}-1}x_{0},TT^{n_{k}-1}x_{0}) %
[1+\omega^{G}_{\lambda_{0}} %
(T^{n_{k}-1}x_{0},TT^{n_{k}-1}x_{0},TT^{n_{k}-1}x_{0}) %
+\omega^{G}_{\lambda_{0}}
(TT^{m_{k}-1}x_{0},TT^{m_{k}-1}x_{0},T^{n_{k}-1}x_{0})]} %
{1+\omega^{G}_{\lambda_{0}} %
(TT^{m_{k}-1}x_{0},TT^{m_{k}-1}x_{0},T^{n_{k}-1}x_{0}) %
+\omega^{G}_{\lambda_{0}} %
(TT^{m_{k}-1}x_{0},TT^{n_{k}-1}x_{0},TT^{n_{k}-1}x_{0})}}.$ %
So, we have that %
\begin{align}\label{}%
6\psi(\epsilon)\notag %
&\le %
\psi(\omega^{G}_{\lambda_{0}} %
(T^{m_{k}}x_{0},T^{n_{k}}x_{0},T^{n_{k}}x_{0}))\\\notag %
&\le%
\kappa_{1}(\psi(\omega^{G}_{\lambda_{0}} %
(T^{m_{k}-1}x_{0},T^{n_{k}-1}x_{0},T^{n_{k}-1}x_{0}))) %
\psi(\omega^{G}_{\lambda_{0}+\nu(\lambda_{0})} %
(T^{m_{k}-1}x_{0},T^{n_{k}-1}x_{0},T^{n_{k}-1}x_{0}))\\\notag %
&+\kappa_{2}(\psi(\omega^{G}_{\lambda_{0}} %
(T^{m_{k}-1}x_{0},T^{n_{k}-1}x_{0},T^{n_{k}-1}x_{0}))) %
\psi(\omega^{G}_{\lambda_{0}} %
(T^{m_{k}-1}x_{0},T^{m_{k}}x_{0},T^{m_{k}}x_{0}))\\ \notag %
&+ \kappa_{3}(\psi(\omega^{G}_{\lambda_{0}} %
(T^{m_{k}-1}x_{0},T^{n_{k}-1}x_{0},T^{n_{k}-1}x_{0}))) %
\psi(\omega^{G}_{\lambda_{0}}
(T^{n_{k}-1}x_{0},T^{n_{k}}x_{0},T^{n_{k}}x_{0}))\\\notag
&+\kappa_{4}(\psi(\omega^{G}_{\lambda_{0}}
(T^{m_{k}-1}x_{0},T^{n_{k}-1}x_{0},T^{n_{k}-1}x_{0})))
\psi(\omega^{G}_{\lambda_{0}} %
(T^{m_{k}}x_{0},T^{n_{k}}x_{0},T^{n_{k}}x_{0}))\\ \notag %
&+\kappa_{5}(\psi(\omega^{G}_{\lambda_{0}} %
(T^{m_{k}-1}x_{0},T^{n_{k}-1}x_{0},T^{n_{k}-1}x_{0}))) %
\psi(\omega^{G}_{\lambda_{0}} %
(T^{m_{k}}x_{0},T^{n_{k}-1}x_{0},T^{n_{k}-1}x_{0}))\\\notag %
&+\kappa_{6}(\psi(\omega^{G}_{\lambda_{0}} %
(T^{m_{k}-1}x_{0},T^{n_{k}-1}x_{0},T^{n_{k}-1}x_{0})))\\\notag %
&\times\psi\Set{\frac{\omega^{G}_{\lambda_{0}} %
(T^{n_{k}-1}x_{0},T^{n_{k}}x_{0},T^{n_{k}}x_{0}) %
[1+\omega^{G}_{\lambda_{0}} %
(T^{m_{k}-1}x_{0},T^{n_{k}-1}x_{0},T^{n_{k}-1}x_{0})]} %
{1+\omega^{G}_{\lambda_{0}} %
(T^{m_{k}-1}x_{0},T^{m_{k}}x_{0},T^{m_{k}}x_{0})}}\\ \notag %
&+\kappa_{7}(\psi(\omega^{G}_{\lambda_{0}} %
(T^{m_{k}-1}x_{0},T^{n_{k}-1}x_{0},T^{n_{k}-1}x_{0})))\\\notag %
&\times\psi\Set{\frac{\omega^{G}_{\lambda_{0}} %
(T^{n_{k}-1}x_{0},T^{n_{k}}x_{0},T^{n_{k}}x_{0}) %
[1+\omega^{G}_{\lambda_{0}} %
(T^{m_{k}-1}x_{0},T^{m_{k}}x_{0},T^{m_{k}}x_{0})]} %
{1+\omega^{G}_{\lambda_{0}} %
(T^{m_{k}-1}x_{0},T^{n_{k}-1}x_{0},T^{n_{k}-1}x_{0})}}\\ \notag %
&+\kappa_{8}(\psi(\omega^{G}_{\lambda_{0}} %
(T^{m_{k}-1}x_{0},T^{n_{k}-1}x_{0},T^{n_{k}-1}x_{0})))\\\notag %
&\times\psi\Set{\frac{\omega^{G}_{\lambda_{0}} %
(T^{m_{k}-1}x_{0},T^{m_{k}}x_{0},T^{m_{k}}x_{0}) %
[1+\omega^{G}_{\lambda_{0}} %
(T^{m_{k}-1}x_{0},T^{n_{k}-1}x_{0},T^{n_{k}-1}x_{0})]} %
{1+\omega^{G}_{\lambda_{0}} %
(T^{n_{k}-1}x_{0},T^{m_{k}}x_{0},T^{m_{k}}x_{0}) %
+\omega^{G}_{\lambda_{0}} %
(T^{m_{k}-1}x_{0},T^{n_{k}-1}x_{0},T^{n_{k}-1}x_{0})}}\\ \notag %
&+\kappa_{9}(\psi(\omega^{G}_{\lambda_{0}} %
(T^{m_{k}-1}x_{0},T^{n_{k}-1}x_{0},T^{n_{k}-1}x_{0})))\\\notag %
&\times\psi\Set{\frac{\omega^{G}_{\lambda_{0}} %
(T^{m_{k}-1}x_{0},T^{n_{k}-1}x_{0},T^{n_{k}-1}x_{0}) %
[1+\omega^{G}_{\lambda_{0}} %
(T^{n_{k}-1}x_{0},T^{n_{k}}x_{0},T^{n_{k}}x_{0})]} %
{1+\omega^{G}_{\lambda_{0}} %
(T^{m_{k}}x_{0},T^{m_{k}}x_{0},T^{n_{k}-1}x_{0}) %
+\omega^{G}_{\lambda_{0}} %
(T^{m_{k}}x_{0},T^{n_{k}}x_{0},T^{n_{k}}x_{0})}}\\ \notag %
&+\kappa_{10}(\psi(\omega^{G}_{\lambda_{0}} %
(T^{m_{k}-1}x_{0},T^{n_{k}-1}x_{0},T^{n_{k}-1}x_{0})))\\\notag %
&\times\psi\Set{\frac{\omega^{G}_{\lambda_{0}} %
(T^{m_{k}}x_{0},T^{n_{k}}x_{0},T^{n_{k}}x_{0}) %
[1+\omega^{G}_{\lambda_{0}} %
(T^{m_{k}-1}x_{0},T^{n_{k}}x_{0},T^{n_{k}}x_{0})]} %
{1+\omega^{G}_{\lambda_{0}} %
(T^{m_{k}}x_{0},T^{m_{k}}x_{0},T^{n_{k}-1}x_{0}) %
+\omega^{G}_{\lambda_{0}} %
(T^{m_{k}-1}x_{0},T^{n_{k}}x_{0},T^{n_{k}}x_{0})}} \\\notag %
&+\kappa_{11}(\psi(\omega^{G}_{\lambda_{0}} %
(T^{m_{k}-1}x_{0},T^{n_{k}-1}x_{0},T^{n_{k}-1}x_{0})))\notag %
\end{align}%
\begin{equation}\label{e}%
\times\psi\Set{\frac{\omega^{G}_{\lambda_{0}} %
(T^{m_{k}}x_{0},T^{n_{k}}x_{0},T^{n_{k}}x_{0}) %
[1+\omega^{G}_{\lambda_{0}} %
(T^{n_{k}-1}x_{0},T^{n_{k}}x_{0},T^{n_{k}}x_{0}) %
+\omega^{G}_{\lambda_{0}} %
(T^{m_{k}}x_{0},T^{m_{k}}x_{0},T^{n_{k}-1}x_{0})]} %
{1+\omega^{G}_{\lambda_{0}} %
(T^{m_{k}}x_{0},T^{m_{k}}x_{0},T^{n_{k}-1}x_{0}) %
+\omega^{G}_{\lambda_{0}} %
(T^{m_{k}}x_{0},T^{n_{k}}x_{0},T^{n_{k}}x_{0})}}. %
\end{equation}%

We rewrite inequality \ref{e} as %

\begin{align}\label{f}%
6\psi(\epsilon)\notag %
\le& %
\psi(\omega^{G}_{\lambda_{0}} %
(T^{m_{k}}x_{0},T^{n_{k}}x_{0},T^{n_{k}}x_{0}))\\\notag %
\le& %
\psi(\omega^{G}_{\lambda_{0}+\nu(\lambda_{0})} %
(T^{m_{k}-1}x_{0},T^{n_{k}-1}x_{0},T^{n_{k}-1}x_{0}))\\\notag %
+&\psi(\omega^{G}_{\lambda_{0}} %
(T^{m_{k}-1}x_{0},T^{m_{k}}x_{0},T^{m_{k}}x_{0}))\\ \notag %
+& \psi(\omega^{G}_{\lambda_{0}} %
(T^{n_{k}-1}x_{0},T^{n_{k}}x_{0},T^{n_{k}}x_{0}))\\\notag %
+&\psi(\omega^{G}_{\lambda_{0}} %
(T^{m_{k}}x_{0},T^{n_{k}}x_{0},T^{n_{k}}x_{0}))\\ \notag %
+&\psi(\omega^{G}_{\lambda_{0}} %
(T^{m_{k}}x_{0},T^{n_{k}-1}x_{0},T^{n_{k}-1}x_{0}))\\\notag %
+&\psi\Set{\frac{\omega^{G}_{\lambda_{0}} %
(T^{n_{k}-1}x_{0},T^{n_{k}}x_{0},T^{n_{k}}x_{0}) %
[1+\omega^{G}_{\lambda_{0}} %
(T^{m_{k}-1}x_{0},T^{n_{k}-1}x_{0},T^{n_{k}-1}x_{0})]} %
{1+\omega^{G}_{\lambda_{0}} %
(T^{m_{k}-1}x_{0},T^{m_{k}}x_{0},T^{m_{k}}x_{0})}}\\ \notag %
+&\psi\Set{\frac{\omega^{G}_{\lambda_{0}} %
(T^{n_{k}-1}x_{0},T^{n_{k}}x_{0},T^{n_{k}}x_{0}) %
[1+\omega^{G}_{\lambda_{0}} %
(T^{m_{k}-1}x_{0},T^{m_{k}}x_{0},T^{m_{k}}x_{0})]} %
{1+\omega^{G}_{\lambda_{0}} %
(T^{m_{k}-1}x_{0},T^{n_{k}-1}x_{0},T^{n_{k}-1}x_{0})}}\\ \notag %
+&\psi\Set{\frac{\omega^{G}_{\lambda_{0}} %
(T^{m_{k}-1}x_{0},T^{m_{k}}x_{0},T^{m_{k}}x_{0}) %
[1+\omega^{G}_{\lambda_{0}} %
(T^{m_{k}-1}x_{0},T^{n_{k}-1}x_{0},T^{n_{k}-1}x_{0})]} %
{1+\omega^{G}_{\lambda_{0}} %
(T^{n_{k}-1}x_{0},T^{m_{k}}x_{0},T^{m_{k}}x_{0}) %
+\omega^{G}_{\lambda_{0}} %
(T^{m_{k}-1}x_{0},T^{n_{k}-1}x_{0},T^{n_{k}-1}x_{0})}}\\ \notag %
+&\psi\Set{\frac{\omega^{G}_{\lambda_{0}} %
(T^{m_{k}-1}x_{0},T^{n_{k}-1}x_{0},T^{n_{k}-1}x_{0}) %
[1+\omega^{G}_{\lambda_{0}} %
(T^{n_{k}-1}x_{0},T^{n_{k}}x_{0},T^{n_{k}}x_{0})]} %
{1+\omega^{G}_{\lambda_{0}} %
(T^{m_{k}}x_{0},T^{m_{k}}x_{0},T^{n_{k}-1}x_{0}) %
+\omega^{G}_{\lambda_{0}} %
(T^{m_{k}}x_{0},T^{n_{k}}x_{0},T^{n_{k}}x_{0})}}\\ \notag %
+&\psi\Set{\frac{\omega^{G}_{\lambda_{0}} %
(T^{m_{k}}x_{0},T^{n_{k}}x_{0},T^{n_{k}}x_{0}) %
[1+\omega^{G}_{\lambda_{0}} %
(T^{m_{k}-1}x_{0},T^{n_{k}}x_{0},T^{n_{k}}x_{0})]} %
{1+\omega^{G}_{\lambda_{0}} %
(T^{m_{k}}x_{0},T^{m_{k}}x_{0},T^{n_{k}-1}x_{0}) %
+\omega^{G}_{\lambda_{0}} %
(T^{m_{k}-1}x_{0},T^{n_{k}}x_{0},T^{n_{k}}x_{0})}} %
\end{align}%
\begin{equation}%
+\psi\Set{\frac{\omega^{G}_{\lambda_{0}} %
(T^{m_{k}}x_{0},T^{n_{k}}x_{0},T^{n_{k}}x_{0}) %
[1+\omega^{G}_{\lambda_{0}} %
(T^{n_{k}-1}x_{0},T^{n_{k}}x_{0},T^{n_{k}}x_{0}) %
+\omega^{G}_{\lambda_{0}} %
(T^{m_{k}}x_{0},T^{m_{k}}x_{0},T^{n_{k}-1}x_{0})]} %
{1+\omega^{G}_{\lambda_{0}} %
(T^{m_{k}}x_{0},T^{m_{k}}x_{0},T^{n_{k}-1}x_{0}) %
+\omega^{G}_{\lambda_{0}} %
(T^{m_{k}}x_{0},T^{n_{k}}x_{0},T^{n_{k}}x_{0})}}. %
\end{equation}%

Using the fact that $\omega^{G}_{\lambda_{0}} %
(T^{m_{k}-1}x_{0},T^{m_{k}}x_{0},T^{m_{k}}x_{0})\le  %
\omega^{G}_{\lambda_{0}} %
(T^{m_{k}-1}x_{0},T^{n_{k-1}}x_{0},T^{n_{k-1}}x_{0}),$ %
\newline%
$\omega^{G}_{\lambda_{0}}(T^{m_{k}}x_{0}, %
T^{n_{k}-1}x_{0},T^{n_{k}-1}x_{0})<\epsilon$ %
and using property 5 of Definition \ref{2.1}, we have,%
\begin{align}%
6\psi(\epsilon)\notag %
&\le %
\psi(\omega^{G}_{\lambda_{0}} %
(T^{m_{k}}x_{0},T^{n_{k}}x_{0},T^{n_{k}}x_{0}))\\\notag %
&\le\psi(\omega^{G}_{\nu(\lambda_{0})} %
(T^{m_{k}-1}x_{0},T^{m_{k}}x_{0},T^{m_{k}}x_{0}) %
+\omega^{G}_{\lambda_{0}} %
(T^{m_{k}}x_{0},T^{n_{k}-1}x_{0},T^{n_{k}-1}x_{0}))\\\notag %
&+\psi(\omega^{G}_{\lambda_{0}} %
(T^{m_{k}-1}x_{0},T^{m_{k}}x_{0},T^{m_{k}}x_{0})) %
+\psi(\omega^{G}_{\lambda_{0}} %
(T^{n_{k}-1}x_{0},T^{n_{k}}x_{0},T^{n_{k}}x_{0}))\\\notag %
&+\psi(\omega^{G}_{\frac{\lambda_{0}}{2}} %
(T^{m_{k}}x_{0},T^{n_{k}-1}x_{0},T^{n_{k}-1}x_{0})) %
+\psi(\omega^{G}_{\frac{\lambda_{0}}{2}} %
(T^{n_{k}-1}x_{0},T^{n_{k}}x_{0},T^{n_{k}}x_{0}))\\\notag %
&+\psi(\omega^{G}_{\lambda_{0}} %
(T^{m_{k}}x_{0},T^{n_{k}-1}x_{0},T^{n_{k}-1}x_{0})) %
+\psi(\omega^{G}_{\lambda_{0}} %
(T^{n_{k}-1}x_{0},T^{n_{k}}x_{0},T^{n_{k}}x_{0}))\\\notag %
&+\psi(\omega^{G}_{\lambda_{0}} %
(T^{n_{k}-1}x_{0},T^{n_{k}}x_{0},T^{n_{k}}x_{0})) %
+\psi(\omega^{G}_{\lambda_{0}} %
(T^{m_{k}-1}x_{0},T^{m_{k}}x_{0},T^{m_{k}}x_{0}))\\\notag %
&+\psi(\omega^{G}_{\frac{\lambda_{0}}{2}} %
(T^{m_{k}-1}x_{0},T^{m_{k}}x_{0},T^{m_{k}}x_{0}) %
+\omega^{G}_{\frac{\lambda_{0}}{2}} %
(T^{m_{k}}x_{0},T^{n_{k}-1}x_{0},T^{n_{k}-1}x_{0}))\\\notag %
&+\psi(\omega^{G}_{\frac{\lambda_{0}}{2}} %
(T^{m_{k}}x_{0},T^{n_{k}-1}x_{0},T^{n_{k}-1}x_{0}) %
+\omega^{G}_{\frac{\lambda_{0}}{2}} %
(T^{n_{k}-1}x_{0},T^{n_{k}}x_{0},T^{n_{k}}x_{0}))\\ %
&+\psi(\omega^{G}_{\frac{\lambda_{0}}{2}} %
(T^{m_{k}}x_{0},T^{n_{k}-1}x_{0},T^{n_{k}-1}x_{0}) %
+\omega^{G}_{\frac{\lambda_{0}}{2}} %
(T^{n_{k}-1}x_{0},T^{n_{k}}x_{0},T^{n_{k}}x_{0})), %
\end{align}%
so that
\begin{align}%
6\psi(\epsilon)\notag %
&\le %
\psi(\omega^{G}_{\lambda_{0}} %
(T^{m_{k}}x_{0},T^{n_{k}}x_{0},T^{n_{k}}x_{0}))\\\notag %
&\le\psi(\omega^{G}_{\nu(\lambda_{0})} %
(T^{m_{k}-1}x_{0},T^{m_{k}}x_{0},T^{m_{k}}x_{0})) %
+\psi(\epsilon)\\\notag %
&+\psi(\omega^{G}_{\lambda_{0}} %
(T^{m_{k}-1}x_{0},T^{m_{k}}x_{0},T^{m_{k}}x_{0})) %
+\psi(\omega^{G}_{\lambda_{0}} %
(T^{n_{k}-1}x_{0},T^{n_{k}}x_{0},T^{n_{k}}x_{0}))\\\notag %
&+\psi(\epsilon)+\psi(\omega^{G}_{\frac{\lambda_{0}}{2}} %
(T^{n_{k}-1}x_{0},T^{n_{k}}x_{0},T^{n_{k}}x_{0})) %
+\psi(\epsilon)\\\notag %
&+\psi(\omega^{G}_{\lambda_{0}} %
(T^{n_{k}-1}x_{0},T^{n_{k}}x_{0},T^{n_{k}}x_{0})) %
+\psi(\omega^{G}_{\lambda_{0}} %
(T^{n_{k}-1}x_{0},T^{n_{k}}x_{0},T^{n_{k}}x_{0}))\\\notag %
&+\psi(\omega^{G}_{\lambda_{0}} %
(T^{m_{k}-1}x_{0},T^{m_{k}}x_{0},T^{m_{k}}x_{0}))\\\notag %
&+\psi(\omega^{G}_{\frac{\lambda_{0}}{2}} %
(T^{m_{k}-1}x_{0},T^{m_{k}}x_{0},T^{m_{k}}x_{0})) %
+\psi(\epsilon)\\\notag %
&+\psi(\epsilon)+\psi(\omega^{G}_{\frac{\lambda_{0}}{2}} %
(T^{n_{k}-1}x_{0},T^{n_{k}}x_{0},T^{n_{k}}x_{0})\\ %
&+\psi(\epsilon)+\psi(\omega^{G}_{\frac{\lambda_{0}}{2}} %
(T^{n_{k}-1}x_{0},T^{n_{k}}x_{0},T^{n_{k}}x_{0})), %
\end{align}%

as $k\rightarrow\infty$, we obtain that %
$6\psi(\epsilon)\le\lim\limits_{k\rightarrow\infty} %
\psi(\omega^{G}_{\lambda_{0}} %
(T^{m_{k}}x_{0},T^{n_{k}}x_{0},T^{n_{k}}x_{0})) %
\le6\psi(\epsilon)$ %
which implies that $\lim\limits_{k\rightarrow\infty} %
\psi(\omega^{G}_{\lambda_{0}} %
(T^{m_{k}}x_{0},T^{n_{k}}x_{0},T^{n_{k}}x_{0})) %
=6\psi(\epsilon)$.  %
\newline%

Therefore, we have %

\begin{equation} %
\lim\limits_{k\rightarrow\infty}
\omega^{G}_{\lambda_{0}} %
(T^{m_{k}}x_{0},T^{n_{k}}x_{0},T^{n_{k}}x_{0}) %
=6\epsilon. %
\end{equation}%

Again, we can see that %

\begin{align}%
\psi(\omega^{G}_{\lambda_{0}} %
(T^{m_{k}}x_{0},T^{n_{k}}x_{0},T^{n_{k}}x_{0}))\notag %
&\le\psi(\omega^{G}_{\lambda_{0}+\nu(\lambda_{0})} %
(T^{m_{k}-1}x_{0},T^{n_{k}-1}x_{0},T^{n_{k}-1}x_{0}))\\\notag %
&+\psi(\omega^{G}_{\lambda_{0}} %
(T^{m_{k}-1}x_{0},T^{m_{k}}x_{0},T^{m_{k}}x_{0}))\\\notag %
&+\psi(\omega^{G}_{\lambda_{0}} %
(T^{n_{k}-1}x_{0},T^{n_{k}}x_{0},T^{n_{k}}x_{0}))\\\notag %
&+\psi(\omega^{G}_{\lambda_{0}} %
(T^{m_{k}}x_{0},T^{n_{k}}x_{0},T^{n_{k}}x_{0}))\\\notag %
&+\psi(\omega^{G}_{\lambda_{0}} %
(T^{m_{k}}x_{0},T^{n_{k}-1}x_{0},T^{n_{k}-1}x_{0}))\\\notag %
&+\psi(\omega^{G}_{\lambda_{0}} %
(T^{n_{k}-1}x_{0},T^{n_{k}}x_{0},T^{n_{k}}x_{0}))\\\notag %
&+\psi(\omega^{G}_{\lambda_{0}} %
(T^{n_{k}-1}x_{0},T^{n_{k}}x_{0},T^{n_{k}}x_{0}))\\\notag %
&+\psi(\omega^{G}_{\lambda_{0}} %
(T^{m_{k}-1}x_{0},T^{m_{k}}x_{0},T^{m_{k}}x_{0}))\\\notag %
&+\psi(\omega^{G}_{\lambda_{0}} %
(T^{m_{k}-1}x_{0},T^{n_{k}-1}x_{0},T^{n_{k}-1}x_{0}))\\\notag %
&+\psi(\omega^{G}_{\lambda_{0}} %
(T^{m_{k}}x_{0},T^{n_{k}}x_{0},T^{n_{k}}x_{0}))\\\notag %
&+\psi(\omega^{G}_{\lambda_{0}} %
(T^{m_{k}}x_{0},T^{n_{k}}x_{0},T^{n_{k}}x_{0}))\\\notag %
&\le %
\psi(\omega^{G}_{\lambda_{0}} %
(T^{m_{k}-1}x_{0},T^{n_{k}-1}x_{0},T^{n_{k}-1}x_{0}))\\\notag %
&+\psi(\omega^{G}_{\lambda_{0}} %
(T^{m_{k}-1}x_{0},T^{m_{k}}x_{0},T^{m_{k}}x_{0}))\\\notag %
&+\psi(\omega^{G}_{\lambda_{0}} %
(T^{n_{k}-1}x_{0},T^{n_{k}}x_{0},T^{n_{k}}x_{0}))\\\notag %
&+\psi(\omega^{G}_{\lambda_{0}} %
(T^{m_{k}}x_{0},T^{n_{k}}x_{0},T^{n_{k}}x_{0}))\\\notag %
&+\psi(\omega^{G}_{\lambda_{0}} %
(T^{m_{k}}x_{0},T^{n_{k}-1}x_{0},T^{n_{k}-1}x_{0}))\\\notag %
&+\psi(\omega^{G}_{\lambda_{0}} %
(T^{n_{k}-1}x_{0},T^{n_{k}}x_{0},T^{n_{k}}x_{0}))\\\notag %
&+\psi(\omega^{G}_{\lambda_{0}} %
(T^{n_{k}-1}x_{0},T^{n_{k}}x_{0},T^{n_{k}}x_{0}))\\\notag %
&+\psi(\omega^{G}_{\lambda_{0}} %
(T^{m_{k}-1}x_{0},T^{m_{k}}x_{0},T^{m_{k}}x_{0}))\\\notag %
&+\psi(\omega^{G}_{\lambda_{0}} %
(T^{m_{k}-1}x_{0},T^{n_{k}-1}x_{0},T^{n_{k}-1}x_{0}))\\\notag %
&+\psi(\omega^{G}_{\lambda_{0}} %
(T^{m_{k}}x_{0},T^{n_{k}}x_{0},T^{n_{k}}x_{0}))\\\notag %
&+\psi(\omega^{G}_{\lambda_{0}} %
(T^{m_{k}}x_{0},T^{n_{k}}x_{0},T^{n_{k}}x_{0}))\\\notag %
&\le %
\psi(\omega^{G}_{\frac{\lambda_{0}}{2}} %
(T^{m_{k}-1}x_{0},T^{m_{k}}x_{0},T^{m_{k}}x_{0}))\\\notag %
&+\psi(\omega^{G}_{\frac{\lambda_{0}}{2}} %
(T^{m_{k}}x_{0},T^{n_{k}-1}x_{0},T^{n_{k}-1}x_{0}))\\\notag %
&+\psi(\omega^{G}_{\lambda_{0}} %
(T^{m_{k}-1}x_{0},T^{m_{k}}x_{0},T^{m_{k}}x_{0}))\\\notag %
&+\psi(\omega^{G}_{\lambda_{0}} %
(T^{n_{k}-1}x_{0},T^{n_{k}}x_{0},T^{n_{k}}x_{0}))\\\notag %
&+\psi(\omega^{G}_{\frac{\lambda_{0}}{2}} %
(T^{m_{k}}x_{0},T^{n_{k}-1}x_{0},T^{n_{k}-1}x_{0}))\\\notag %
&+\psi(\omega^{G}_{\frac{\lambda_{0}}{2}} %
(T^{n_{k}-1}x_{0},T^{n_{k}}x_{0},T^{n_{k}}x_{0}))\\\notag %
&+\psi(\omega^{G}_{\lambda_{0}} %
(T^{m_{k}}x_{0},T^{n_{k}-1}x_{0},T^{n_{k}-1}x_{0}))\\\notag %
&+\psi(\omega^{G}_{\lambda_{0}} %
(T^{n_{k}-1}x_{0},T^{n_{k}}x_{0},T^{n_{k}}x_{0}))\\\notag %
&+\psi(\omega^{G}_{\lambda_{0}} %
(T^{n_{k}-1}x_{0},T^{n_{k}}x_{0},T^{n_{k}}x_{0}))\\\notag %
&+\psi(\omega^{G}_{\lambda_{0}} %
(T^{m_{k}-1}x_{0},T^{m_{k}}x_{0},T^{m_{k}}x_{0}))\\\notag %
&+\psi(\omega^{G}_{\lambda_{0}} %
(T^{m_{k}-1}x_{0},T^{n_{k}-1}x_{0},T^{n_{k}-1}x_{0}))\\\notag %
&+\psi(\omega^{G}_{\frac{\lambda_{0}}{2}} %
(T^{m_{k}}x_{0},T^{n_{k}-1}x_{0},T^{n_{k}-1}x_{0}))\\\notag %
&+\psi(\omega^{G}_{\frac{\lambda_{0}}{2}} %
(T^{n_{k}-1}x_{0},T^{n_{k}}x_{0},T^{n_{k}}x_{0}))\\\notag %
&+\psi(\omega^{G}_{\frac{\lambda_{0}}{2}} %
(T^{m_{k}}x_{0},T^{n_{k}-1}x_{0},T^{n_{k}-1}x_{0}))\\\notag %
&+\psi(\omega^{G}_{\frac{\lambda_{0}}{2}} %
(T^{n_{k}-1}x_{0},T^{n_{k}}x_{0},T^{n_{k}}x_{0}))\\\notag %
&\le %
\psi(\omega^{G}_{\lambda_{0}} %
(T^{m_{k}-1}x_{0},T^{m_{k}}x_{0},T^{m_{k}}x_{0}))\\\notag %
&+\psi(\omega^{G}_{\lambda_{0}} %
(T^{m_{k}}x_{0},T^{n_{k}-1}x_{0},T^{n_{k}-1}x_{0}))\\\notag %
&+\psi(\omega^{G}_{\lambda_{0}} %
(T^{m_{k}-1}x_{0},T^{m_{k}}x_{0},T^{m_{k}}x_{0}))\\\notag %
&+\psi(\omega^{G}_{\lambda_{0}} %
(T^{n_{k}-1}x_{0},T^{n_{k}}x_{0},T^{n_{k}}x_{0}))\\\notag %
&+\psi(\omega^{G}_{\lambda_{0}} %
(T^{m_{k}}x_{0},T^{n_{k}-1}x_{0},T^{n_{k}-1}x_{0}))\\\notag %
&+\psi(\omega^{G}_{\lambda_{0}} %
(T^{n_{k}-1}x_{0},T^{n_{k}}x_{0},T^{n_{k}}x_{0}))\\\notag %
&+\psi(\omega^{G}_{\lambda_{0}} %
(T^{m_{k}}x_{0},T^{n_{k}-1}x_{0},T^{n_{k}-1}x_{0}))\\\notag %
&+\psi(\omega^{G}_{\lambda_{0}} %
(T^{n_{k}-1}x_{0},T^{n_{k}}x_{0},T^{n_{k}}x_{0}))\\\notag %
&+\psi(\omega^{G}_{\lambda_{0}} %
(T^{n_{k}-1}x_{0},T^{n_{k}}x_{0},T^{n_{k}}x_{0}))\\\notag
&+\psi(\omega^{G}_{\lambda_{0}}
(T^{m_{k}-1}x_{0},T^{m_{k}}x_{0},T^{m_{k}}x_{0}))\\\notag %
&+\psi(\omega^{G}_{\lambda_{0}} %
(T^{m_{k}-1}x_{0},T^{m_{k}}x_{0},T^{m_{k}}x_{0}))\\\notag %
&+\psi(\omega^{G}_{\lambda_{0}} %
(T^{m_{k}}x_{0},T^{n_{k}-1}x_{0},T^{n_{k}-1}x_{0}))\\\notag %
&+\psi(\omega^{G}_{\lambda_{0}} %
(T^{m_{k}}x_{0},T^{n_{k}-1}x_{0},T^{n_{k}-1}x_{0}))\\\notag %
&+\psi(\omega^{G}_{\lambda_{0}} %
(T^{n_{k}-1}x_{0},T^{n_{k}}x_{0},T^{n_{k}}x_{0}))\\\notag %
&+\psi(\omega^{G}_{\lambda_{0}} %
(T^{m_{k}}x_{0},T^{n_{k}-1}x_{0},T^{n_{k}-1}x_{0}))\\ %
&+\psi(\omega^{G}_{\lambda_{0}} %
(T^{n_{k}-1}x_{0},T^{n_{k}}x_{0},T^{n_{k}}x_{0})). %
\end{align}%

Letting $k\rightarrow\infty$, we deduce that %
$\psi(6\epsilon)\le\lim\limits_{k\rightarrow\infty} %
\psi(\omega^{G}_{\lambda_{0}}(T^{m_{k}-1}x_{0}, %
T^{n_{k}-1}x_{0},T^{n_{k}-1}x_{0})) %
\le\psi(6\epsilon)$, for which we have that %
$\lim\limits_{k\rightarrow\infty} %
\psi(\omega^{G}_{\lambda_{0}}(T^{m_{k}-1}x_{0}, %
T^{n_{k}-1}x_{0},T^{n_{k}-1}x_{0})) %
=\psi(6\epsilon)$.  %
\newline %
Hence $\lim\limits_{k\rightarrow %
\infty}\omega^{G}_{\lambda_{0}} %
(T^{m_{k}-1}x_{0},T^{n_{k}-1}x_{0}, %
T^{n_{k}-1}x_{0})=6\epsilon.$ %
Thus, it follows that %
\begin{align}%
1\notag %
&\le\lim\limits_{k\rightarrow %
\infty}\inf\biggl(\kappa_{1} %
(\psi(\omega^{G}_{\lambda_{0}} %
(T^{m_{k}-1}x_{0},T^{n_{k}-1}x_{0},T^{n_{k}-1}x_{0})))\notag %
+\kappa_{2}(\psi(\omega^{G}_{\lambda_{0}} %
(T^{m_{k}-1}x_{0},T^{n_{k}-1}x_{0},T^{n_{k}-1}x_{0})))\\\notag %
&+\kappa_{3}(\psi(\omega^{G}_{\lambda_{0}} %
(T^{m_{k}-1}x_{0},T^{n_{k}-1}x_{0},T^{n_{k}-1}x_{0})))\notag %
+\kappa_{4}(\psi(\omega^{G}_{\lambda_{0}} %
(T^{m_{k}-1}x_{0},T^{n_{k}-1}x_{0},T^{n_{k}-1}x_{0})))\\\notag %
&+\kappa_{5}(\psi(\omega^{G}_{\lambda_{0}} %
(T^{m_{k}-1}x_{0},T^{n_{k}-1}x_{0},T^{n_{k}-1}x_{0})))\notag %
+\kappa_{6}(\psi(\omega^{G}_{\lambda_{0}} %
(T^{m_{k}-1}x_{0},T^{n_{k}-1}x_{0},T^{n_{k}-1}x_{0})))\\\notag %
&+\kappa_{7}(\psi(\omega^{G}_{\lambda_{0}} %
(T^{m_{k}-1}x_{0},T^{n_{k}-1}x_{0},T^{n_{k}-1}x_{0})))\notag %
+\kappa_{8}(\psi(\omega^{G}_{\lambda_{0}} %
(T^{m_{k}-1}x_{0},T^{n_{k}-1}x_{0},T^{n_{k}-1}x_{0})))\\\notag %
&+\kappa_{9}(\psi(\omega^{G}_{\lambda_{0}} %
(T^{m_{k}-1}x_{0},T^{n_{k}-1}x_{0},T^{n_{k}-1}x_{0})))\notag %
+\kappa_{10}(\psi(\omega^{G}_{\lambda_{0}} %
(T^{m_{k}-1}x_{0},T^{n_{k}-1}x_{0},T^{n_{k}-1}x_{0})))\\ %
&+\kappa_{11}(\psi(\omega^{G}_{\lambda_{0}} %
(T^{m_{k}-1}x_{0},T^{n_{k}-1}x_{0},T^{n_{k}-1}x_{0})))\bigg). %
\end{align}%

Therefore, %
$\lim\limits_{k\rightarrow\infty}\psi(\omega^{G}_{\lambda_{0}} %
(T^{m_{k}-1}x_{0},T^{n_{k}-1}x_{0},T^{n_{k}-1}x_{0}))=0$, %
which implies that %
\begin{equation} %
\lim\limits_{k\rightarrow\infty}\omega^{G}_{\lambda_{0}} %
(T^{m_{k}-1}x_{0},T^{n_{k}-1}x_{0},T^{n_{k}-1}x_{0})=0, %
\end{equation}%

which is a contradiction. 
Hence $\{T^{n}x_{0}\}_{n\in\mathbb{N}}$ is a %
modular G-Cauchy sequence in $X_{\omega^{G}}$. %
\newline
(d)	
Since $X_{\omega^{G}}$ is complete modular G-metric space, %
there exists $x^{\ast}\in X_{\omega^{G}}$ such that %
$T^{n}x_{0}\rightarrow x^{\ast}\in X_{\omega^{G}}$. %
Now we show that $x^{\ast}$ is a fixed point of $T$ %
for $\lambda>0$, suppose that $Tx^{\ast}\neq x^{\ast}$, %
applying condition 5 of Definition \ref{2.1}, we have that %
\begin{align} \label{do}%
\psi(\omega^{G}_{\lambda} %
(T^{n}x_{0},Tx^{\ast},Tx^{\ast}))\notag %
&\le\psi(\omega^{G}_{\frac{\lambda}{2}} %
(T^{n+1}x_{0},Tx^{\ast},Tx^{\ast}))\notag %
+\psi(\omega^{G}_{\frac{\lambda}{2}} %
(T^{n}x_{0},T^{n+1}x_{0},T^{n+1}x_{0}))\\\notag %
&=\psi(\omega^{G}_{\frac{\lambda}{2}} %
(T^{n+1}x_{0},Tx^{\ast},Tx^{\ast}))\notag %
+\psi(\omega^{G}_{\frac{\lambda}{2}} %
(T^{n}x_{0},T^{n+1}x_{0},T^{n+1}x_{0}))\\ %
&\le\psi(\omega^{G}_{\lambda} %
(T^{n+1}x_{0},Tx^{\ast},Tx^{\ast})) %
+\psi(\omega^{G}_{\lambda} %
(T^{n}x_{0},T^{n+1}x_{0},T^{n+1}x_{0})). %
\end{align}%

Putting  $x=T^{n}x_{0}, y=x^{\ast}$ in the fist 
quantity in inequality \ref{do} 
into inequality \ref{a}, take $g_{n}=\psi(\omega^{G}_{\lambda} %
(T^{n}x_{0},Tx^{\ast},Tx^{\ast}))$, we have %
\begin{align}\label{g}%
g_{n}\notag %
\le&\kappa_{1}(\psi(\omega^{G}_{\lambda}(T^{n}x_{0},x^{\ast},x^{\ast}))) %
\psi(\omega^{G}_{\lambda+\nu(\lambda)}(T^{n}x_{0},x^{\ast},x^{\ast}))\\ \notag %
+&\kappa_{2}(\psi(\omega^{G}_{\lambda}(T^{n}x_{0},x^{\ast},x^{\ast}))) %
\psi(\omega^{G}_{\lambda}(T^{n}x_{0},TT^{n}x_{0},TT^{n}x_{0}))\\ \notag %
+& \kappa_{3}(\psi(\omega^{G}_{\lambda}(T^{n}x_{0},x^{\ast},x^{\ast}))) %
\psi(\omega^{G}_{\lambda}(x^{\ast},Tx^{\ast},Tx^{\ast}))\\ \notag %
+&\kappa_{4}(\psi(\omega^{G}_{\lambda}(T^{n}x_{0},x^{\ast},x^{\ast}))) %
\psi(\omega^{G}_{\lambda}(TT^{n}x_{0},Tx^{\ast},Tx^{\ast}))\\ \notag %
+&\kappa_{5}(\psi(\omega^{G}_{\lambda}(T^{n}x_{0},x^{\ast},x^{\ast}))) %
\psi(\omega^{G}_{\lambda}(TT^{n}x_{0},x^{\ast},x^{\ast}))\\ \notag %
+&\kappa_{6}(\psi(\omega^{G}_{\lambda}(T^{n}x_{0},x^{\ast},x^{\ast})))\\\notag %
\times&\psi\Set{\frac{\omega^{G}_{\lambda}(x^{\ast},Tx^{\ast},Tx^{\ast}) %
[1+\omega^{G}_{\lambda}(T^{n}x_{0},x^{\ast},x^{\ast})]} %
{1+\omega^{G}_{\lambda}(T^{n}x_{0},TT^{n}x_{0},TT^{n}x_{0})}}\\ \notag %
+&\kappa_{7}(\psi(\omega^{G}_{\lambda}(T^{n}x_{0},x^{\ast},x^{\ast})))\\\notag %
\times&\psi\Set{\frac{\omega^{G}_{\lambda}(x^{\ast},Tx^{\ast},Tx^{\ast}) %
[1+\omega^{G}_{\lambda}(T^{n}x_{0},TT^{n}x_{0},TT^{n}x_{0})]} %
{1+\omega^{G}_{\lambda}(T^{n}x_{0},x^{\ast},x^{\ast})}}\\ \notag %
+&\kappa_{8}(\psi(\omega^{G}_{\lambda}(T^{n}x_{0},x^{\ast},x^{\ast})))\\\notag %
\times&\psi\Set{\frac{\omega^{G}_{\lambda}(T^{n}x_{0},TT^{n}x_{0},TT^{n}x_{0}) %
[1+\omega^{G}_{\lambda}(T^{n}x_{0},x^{\ast},x^{\ast})]} %
{1+\omega^{G}_{\lambda}(x^{\ast},TT^{n}x_{0},TT^{n}x_{0}) %
+\omega^{G}_{\lambda}(T^{n}x_{0},x^{\ast},x^{\ast})}}\\ \notag %
+&\kappa_{9}(\psi(\omega^{G}_{\lambda}(T^{n}x_{0},x^{\ast},x^{\ast})))\\\notag %
\times&\psi\Set{\frac{\omega^{G}_{\lambda}(T^{n}x_{0},x^{\ast},x^{\ast}) %
[1+\omega^{G}_{\lambda}(x^{\ast},Tx^{\ast},Tx^{\ast})]} %
{1+\omega^{G}_{\lambda}(TT^{n}x_{0},TT^{n}x_{0},x^{\ast}) %
+\omega^{G}_{\lambda}(TT^{n}x_{0},Tx^{\ast},Tx^{\ast})}}\\ \notag %
+&\kappa_{10}(\psi(\omega^{G}_{\lambda}(T^{n}x_{0},x^{\ast},x^{\ast})))\\\notag %
\times&\psi\Set{\frac{\omega^{G}_{\lambda}(TT^{n}x_{0},Tx^{\ast},Tx^{\ast}) %
[1+\omega^{G}_{\lambda}(T^{n}x_{0},Tx^{\ast},Tx^{\ast})]} %
{1+\omega^{G}_{\lambda}(TT^{n}x_{0},TT^{n}x_{0},x^{\ast}) %
+\omega^{G}_{\lambda}(T^{n}x_{0},Tx^{\ast},Tx^{\ast})}}\\ \notag %
+&\kappa_{11}(\psi(\omega^{G}_{\lambda}(T^{n}x_{0},x^{\ast},x^{\ast})))\\\notag %
\times&\psi\Set{\frac{\omega^{G}_{\lambda}(TT^{n}x_{0},Tx^{\ast},Tx^{\ast}) %
[1+\omega^{G}_{\lambda}(x^{\ast},Tx^{\ast},Tx^{\ast}) %
+\omega^{G}_{\lambda}(TT^{n}x_{0},TT^{n}x_{0},x^{\ast})]} %
{1+\omega^{G}_{\lambda}(TT^{n}x_{0},TT^{n}x_{0},x^{\ast}) %
+\omega^{G}_{\lambda}(TT^{n}x_{0},Tx^{\ast},Tx^{\ast})}}\\\notag %
+&\psi(\omega^{G}_{\lambda} %
(T^{n}x_{0},T^{n+1}x_{0},T^{n+1}x_{0}))\\\notag %
=&\kappa_{1}(\psi(\omega^{G}_{\lambda}(T^{n}x_{0},x^{\ast},x^{\ast}))) %
\psi(\omega^{G}_{\lambda+\nu(\lambda)}(T^{n}x_{0},x^{\ast},x^{\ast}))\\ \notag %
+&\kappa_{2}(\psi(\omega^{G}_{\lambda}(T^{n}x_{0},x^{\ast},x^{\ast}))) %
\psi(\omega^{G}_{\lambda}(T^{n}x_{0},T^{n+1}x_{0},T^{n+1}x_{0}))\\ \notag %
+& \kappa_{3}(\psi(\omega^{G}_{\lambda}(T^{n}x_{0},x^{\ast},x^{\ast}))) %
\psi(\omega^{G}_{\lambda}(x^{\ast},Tx^{\ast},Tx^{\ast}))\\ \notag %
+&\kappa_{4}(\psi(\omega^{G}_{\lambda}(T^{n}x_{0},x^{\ast},x^{\ast}))) %
\psi(\omega^{G}_{\lambda}(T^{n+1}x_{0},Tx^{\ast},Tx^{\ast}))\\ \notag %
+&\kappa_{5}(\psi(\omega^{G}_{\lambda}(T^{n}x_{0},x^{\ast},x^{\ast}))) %
\psi(\omega^{G}_{\lambda}(T^{n+1}x_{0},x^{\ast},x^{\ast}))\\ \notag %
+&\kappa_{6}(\psi(\omega^{G}_{\lambda}(T^{n}x_{0},x^{\ast},x^{\ast})))\\\notag %
\times&\psi\Set{\frac{\omega^{G}_{\lambda}(x^{\ast},Tx^{\ast},Tx^{\ast}) %
[1+\omega^{G}_{\lambda}(T^{n}x_{0},x^{\ast},x^{\ast})]} %
{1+\omega^{G}_{\lambda}(T^{n}x_{0},T^{n+1}x_{0},T^{n+1}x_{0})}}\\ \notag %
+&\kappa_{7}(\psi(\omega^{G}_{\lambda}(T^{n}x_{0},x^{\ast},x^{\ast})))\\\notag %
\times&\psi\Set{\frac{\omega^{G}_{\lambda}(x^{\ast},Tx^{\ast},Tx^{\ast}) %
[1+\omega^{G}_{\lambda}(T^{n}x_{0},T^{n+1}x_{0},T^{n+1}x_{0})]} %
{1+\omega^{G}_{\lambda}(T^{n}x_{0},x^{\ast},x^{\ast})}}\\ \notag %
+&\kappa_{8}(\psi(\omega^{G}_{\lambda}(T^{n}x_{0},x^{\ast},x^{\ast})))\\\notag %
\times&\psi\Set{\frac{\omega^{G}_{\lambda}(T^{n}x_{0},T^{n+1}x_{0},T^{n+1}x_{0}) %
[1+\omega^{G}_{\lambda}(T^{n}x_{0},x^{\ast},x^{\ast})]} %
{1+\omega^{G}_{\lambda}(x^{\ast},T^{n+1}x_{0},T^{n+1}x_{0}) %
+\omega^{G}_{\lambda}(T^{n}x_{0},x^{\ast},x^{\ast})}}\\ \notag %
+&\kappa_{9}(\psi(\omega^{G}_{\lambda}(T^{n}x_{0},x^{\ast},x^{\ast})))\\\notag %
\times&\psi\Set{\frac{\omega^{G}_{\lambda}(T^{n}x_{0},x^{\ast},x^{\ast}) %
[1+\omega^{G}_{\lambda}(x^{\ast},Tx^{\ast},Tx^{\ast})]} %
{1+\omega^{G}_{\lambda}(T^{n+1}x_{0},T^{n+1}x_{0},x^{\ast}) %
+\omega^{G}_{\lambda}(T^{n+1}x_{0},Tx^{\ast},Tx^{\ast})}}\\ \notag %
+&\kappa_{10}(\psi(\omega^{G}_{\lambda}(T^{n}x_{0},x^{\ast},x^{\ast})))\\\notag %
\times&\psi\Set{\frac{\omega^{G}_{\lambda}(T^{n+1}x_{0},Tx^{\ast},Tx^{\ast}) %
[1+\omega^{G}_{\lambda}(T^{n}x_{0},Tx^{\ast},Tx^{\ast})]} %
{1+\omega^{G}_{\lambda}(T^{n+1}x_{0},T^{n+1}x_{0},x^{\ast}) %
+\omega^{G}_{\lambda}(T^{n}x_{0},Tx^{\ast},Tx^{\ast})}}\\ \notag %
+&\kappa_{11}(\psi(\omega^{G}_{\lambda}(T^{n}x_{0},x^{\ast},x^{\ast})))\\\notag %
\times&\psi\Set{\frac{\omega^{G}_{\lambda}(T^{n+1}x_{0},Tx^{\ast},Tx^{\ast}) %
[1+\omega^{G}_{\lambda}(x^{\ast},Tx^{\ast},Tx^{\ast}) %
+\omega^{G}_{\lambda}(T^{n+1}x_{0},T^{n+1}x_{0},x^{\ast})]} %
{1+\omega^{G}_{\lambda}(T^{n+1}x_{0},T^{n+1}x_{0},x^{\ast}) %
+\omega^{G}_{\lambda}(T^{n+1}x_{0},Tx^{\ast},Tx^{\ast})}}\\ %
+&\psi(\omega^{G}_{\lambda} %
(T^{n}x_{0},T^{n+1}x_{0},T^{n+1}x_{0})).\\\notag %
\end{align}%

Letting $n\rightarrow\infty$, we have that %

\begin{align} %
\psi(\omega^{G}_{\lambda}(x^{\ast},Tx^{\ast},Tx^{\ast}))\notag %
\le &\kappa_{1}(\psi(\omega^{G}_{\lambda}(x^{\ast},x^{\ast},x^{\ast}))) %
\psi(\omega^{G}_{\lambda+\nu(\lambda)}(x^{\ast},x^{\ast},x^{\ast}))\\ \notag %
+&\kappa_{2}(\psi(\omega^{G}_{\lambda}(x^{\ast},x^{\ast},x^{\ast}))) %
\psi(\omega^{G}_{\lambda}(x^{\ast},x^{\ast},x^{\ast}))\\ \notag %
+& \kappa_{3}(\psi(\omega^{G}_{\lambda}(x^{\ast},x^{\ast},x^{\ast}))) %
\psi(\omega^{G}_{\lambda}(x^{\ast},Tx^{\ast},Tx^{\ast}))\\ \notag %
+&\kappa_{4}(\psi(\omega^{G}_{\lambda}(x^{\ast},x^{\ast},x^{\ast}))) %
\psi(\omega^{G}_{\lambda}(x^{\ast},Tx^{\ast},Tx^{\ast}))\\ \notag %
+&\kappa_{5}(\psi(\omega^{G}_{\lambda}(x^{\ast},x^{\ast},x^{\ast}))) %
\psi(\omega^{G}_{\lambda}(x^{\ast},x^{\ast},x^{\ast}))\\ \notag %
+&\kappa_{6}(\psi(\omega^{G}_{\lambda}(x^{\ast},x^{\ast},x^{\ast})))\\\notag %
\times&\psi\Set{\frac{\omega^{G}_{\lambda}(x^{\ast},Tx^{\ast},Tx^{\ast}) %
[1+\omega^{G}_{\lambda}(x^{\ast},x^{\ast},x^{\ast})]} %
{1+\omega^{G}_{\lambda}(x^{\ast},x^{\ast},x^{\ast})}}\\ \notag %
+&\kappa_{7}(\psi(\omega^{G}_{\lambda}(x^{\ast},x^{\ast},x^{\ast})))\\\notag %
\times&\psi\Set{\frac{\omega^{G}_{\lambda}(x^{\ast},Tx^{\ast},Tx^{\ast}) %
[1+\omega^{G}_{\lambda}(x^{\ast},x^{\ast},x^{\ast})]} %
{1+\omega^{G}_{\lambda}(x^{\ast},x^{\ast},x^{\ast})}}\\ \notag %
+&\kappa_{8}(\psi(\omega^{G}_{\lambda}(x^{\ast},x^{\ast},x^{\ast})))\\\notag %
\times&\psi\Set{\frac{\omega^{G}_{\lambda}(x^{\ast},x^{\ast},x^{\ast}) %
[1+\omega^{G}_{\lambda}(x^{\ast},x^{\ast},x^{\ast})]} %
{1+\omega^{G}_{\lambda}(x^{\ast},x^{\ast},x^{\ast}) %
+\omega^{G}_{\lambda}(x^{\ast},x^{\ast},x^{\ast})}}\\ \notag %
+&\kappa_{9}(\psi(\omega^{G}_{\lambda}(x^{\ast},x^{\ast},x^{\ast})))\\\notag %
\times&\psi\Set{\frac{\omega^{G}_{\lambda}(x^{\ast},x^{\ast},x^{\ast}) %
[1+\omega^{G}_{\lambda}(x^{\ast},Tx^{\ast},Tx^{\ast})]} %
{1+\omega^{G}_{\lambda}(x^{\ast},x^{\ast},x^{\ast}) %
+\omega^{G}_{\lambda}(x^{\ast},Tx^{\ast},Tx^{\ast})}}\\ \notag %
+&\kappa_{10}(\psi(\omega^{G}_{\lambda}(x^{\ast},x^{\ast},x^{\ast})))\\\notag %
\times&\psi\Set{\frac{\omega^{G}_{\lambda}(x^{\ast},Tx^{\ast},Tx^{\ast}) %
[1+\omega^{G}_{\lambda}(x^{\ast},Tx^{\ast},Tx^{\ast})]} %
{1+\omega^{G}_{\lambda}(x^{\ast},x^{\ast},x^{\ast}) %
+\omega^{G}_{\lambda}(x^{\ast},Tx^{\ast},Tx^{\ast})}}\\ \notag %
+&\kappa_{11}(\psi(\omega^{G}_{\lambda}(x^{\ast},x^{\ast},x^{\ast})))\\\notag %
\times&\psi\Set{\frac{\omega^{G}_{\lambda}(x^{\ast},Tx^{\ast},Tx^{\ast}) %
[1+\omega^{G}_{\lambda}(x^{\ast},Tx^{\ast},Tx^{\ast}) %
+\omega^{G}_{\lambda}(x^{\ast},x^{\ast},x^{\ast})]} %
{1+\omega^{G}_{\lambda}(x^{\ast},x^{\ast},x^{\ast}) %
+\omega^{G}_{\lambda}(x^{\ast},Tx^{\ast},Tx^{\ast})}}\\ %
+&\psi(\omega^{G}_{\lambda} %
(x^{\ast},x^{\ast},x^{\ast})). %
\end{align}%

Using condition 1 of Definition \ref{2.1}, we have %
\begin{equation}\label{mee} %
  \psi(\omega^{G}_{\lambda}(x^{\ast},Tx^{\ast},Tx^{\ast})) %
  \le(\kappa_{3}(0)+\kappa_{4}(0)+ %
  \kappa_{6}(0)+\kappa_{7}(0)+\kappa_{10}(0)+\kappa_{11}(0)) %
  \psi(\omega^{G}_{\lambda}(x^{\ast},Tx^{\ast},Tx^{\ast})). %
\end{equation}%
So that %
\begin{equation}\label{m3} %
  (1-(\kappa_{3}(0)+\kappa_{4}(0)+ %
  \kappa_{6}(0)+\kappa_{7}(0)+\kappa_{10}(0)+\kappa_{11}(0)))  %
  \psi(\omega^{G}_{\lambda}(x^{\ast},Tx^{\ast},Tx^{\ast}))\le 0, %
\end{equation} %
where, $\kappa_{3}(0)+\kappa_{4}(0)+ %
  \kappa_{6}(0)+\kappa_{7}(0)+\kappa_{10}(0)+\kappa_{11}(0)<1$ %
  and for all $\lambda>0$. %
  Therefore, %
$\psi(\omega^{G}_{\lambda}(x^{\ast},Tx^{\ast},Tx^{\ast})) %
\le 0$ which implies that $\omega^{G}_{\lambda} %
(x^{\ast},Tx^{\ast},Tx^{\ast})\le 0$. %
Hence $Tx^{\ast}=x^{\ast}$. Therefore,  %
$x^{\ast}$ is a fixed point of $T$. %
Lastly, for the uniqueness, we know that $T$ has %
a fixed point $x^{\ast}\in X_{\omega^{G}}$.  %
Suppose that there is another fixed point of %
$T$ i.e; $y^{\ast}\in X_{\omega^{G}}$, thus  %
condition (2) of \autoref{3.1} says that if %
$z\in X_{\omega^{G}}$ with $z\preceq Tz$ and  %
it is comparable to both %
$x^{\ast}$ and $y^{\ast}$ and $T^{n}z$ is  %
also comparable to $x^{\ast}$ %
and $y^{\ast}$ for each $n\in\mathbb{N}$.  %
Now for $\lambda>0$, then %
$\psi(\omega^{G}_{\lambda}(T^{n+1}z,x^{\ast},x^{\ast}))$ and   %
$\psi(\omega^{G}_{\lambda}(T^{n+1}z,y^{\ast},y^{\ast}))$  %
are finite. %
\newline
(e)
We now show that $x^{\ast}=y^{\ast}$. Indeed, using inequality \ref{a}, %
we have by taking $x=T^{n}z$ and $y=x^{\ast}$. First consider  %
$\psi(\omega^{G}_{\lambda}(T^{n+1}z,x^{\ast},x^{\ast}))<\infty,$ %
so that we have the following; %
\begin{align}
\psi(\omega^{G}_{\lambda}(T^{n+1}z,x^{\ast},x^{\ast}))\notag %
=&\psi(\omega_{\lambda}(T^{n+1}z,Tx^{\ast},Tx^{\ast}))\\\notag %
=&\psi(\omega^{G}_{\lambda}(TT^{n}z,Tx^{\ast},Tx^{\ast}))\\\notag %
\le&\kappa_{1}(\psi(\omega^{G}_{\lambda}(T^{n}z,x^{\ast},x^{\ast}))) %
\psi(\omega^{G}_{\lambda+\nu(\lambda)}(T^{n}z,x^{\ast},x^{\ast})) \notag %
+\kappa_{2}(\psi(\omega^{G}_{\lambda}(T^{n}z,x^{\ast},x^{\ast}))) %
\\\notag\times&\psi(\omega^{G}_{\lambda}(T^{n}z,TT^{n}z,TT^{n}z)) \notag %
+\kappa_{3}(\psi(\omega^{G}_{\lambda}(T^{n}z,x^{\ast},x^{\ast}))) %
\psi(\omega^{G}_{\lambda}(x^{\ast},Tx^{\ast},Tx^{\ast}))\\ \notag %
+&\kappa_{4}(\psi(\omega^{G}_{\lambda}(T^{n}z,x^{\ast},x^{\ast}))) %
\psi(\omega^{G}_{\lambda}(TT^{n}z,Tx^{\ast},Tx^{\ast}))\\ \notag %
+&\kappa_{5}(\psi(\omega^{G}_{\lambda}(T^{n}z,x^{\ast},x^{\ast}))) %
\psi(\omega^{G}_{\lambda}(TT^{n}z,x^{\ast},x^{\ast}))\notag %
+\kappa_{6}(\psi(\omega^{G}_{\lambda}(T^{n}z,x^{\ast},x^{\ast})))\\ \notag %
\times&\psi\Set{\frac{\omega^{G}_{\lambda}(x^{\ast},Tx^{\ast},Tx^{\ast}) %
[1+\omega^{G}_{\lambda}(T^{n}z,x^{\ast},x^{\ast})]} %
{1+\omega^{G}_{\lambda}(T^{n}z,TT^{n}z,T^{n}z)}}\\ \notag %
+&\kappa_{7}(\psi(\omega^{G}_{\lambda}(T^{n}z,x^{\ast},x^{\ast})))\\ \notag %
\times&\psi\Set{\frac{\omega^{G}_{\lambda}(x^{\ast},Tx^{\ast},Tx^{\ast}) %
[1+\omega^{G}_{\lambda}(T^{n}z,TT^{n}z,TT^{n}z)]} %
{1+\omega^{G}_{\lambda}(T^{n}z,x^{\ast},x^{\ast})}}\\ \notag %
+&\kappa_{8}(\psi(\omega^{G}_{\lambda}(T^{n}z,x^{\ast},x^{\ast})))\\ \notag %
\times&\psi\Set{\frac{\omega^{G}_{\lambda}(T^{n}z,TT^{n}z,TT^{n}z) %
[1+\omega^{G}_{\lambda}(T^{n}z,x^{\ast},x^{\ast})]} %
{1+\omega^{G}_{\lambda}(x^{\ast},TT^{n}z,TT^{n}z) %
+\omega^{G}_{\lambda}(T^{n}z,x^{\ast},x^{\ast})}}\\ \notag %
+&\kappa_{9}(\psi(\omega^{G}_{\lambda}(T^{n}z,x^{\ast},x^{\ast})))\\ \notag %
\times&\psi\Set{\frac{\omega^{G}_{\lambda}(T^{n}z,x^{\ast},x^{\ast}) %
[1+\omega^{G}_{\lambda}(x^{\ast},Tx^{\ast},Tx^{\ast})]} %
{1+\omega^{G}_{\lambda}(TT^{n}z,TT^{n}z,x^{\ast}) %
+\omega^{G}_{\lambda}(TT^{n}z,Tx^{\ast},Tx^{\ast})}}\\ \notag %
+&\kappa_{10}(\psi(\omega^{G}_{\lambda}(T^{n}z,x^{\ast},x^{\ast})))\\ \notag %
\times&\psi\Set{\frac{\omega^{G}_{\lambda}(TT^{n}z,Tx^{\ast},Tx^{\ast}) %
[1+\omega^{G}_{\lambda}(T^{n}z,Tx^{\ast},Tx^{\ast})]} %
{1+\omega^{G}_{\lambda}(TT^{n}z,TT^{n}z,x^{\ast}) %
+\omega^{G}_{\lambda}(T^{n}z,Tx^{\ast},Tx^{\ast})}}\\ \notag %
+&\kappa_{11}(\psi(\omega^{G}_{\lambda}(T^{n}z,x^{\ast},x^{\ast})))\\ %
\times&\psi\Set{\frac{\omega^{G}_{\lambda}(TT^{n}z,Tx^{\ast},Tx^{\ast}) %
[1+\omega^{G}_{\lambda}(x^{\ast},Tx^{\ast},Tx^{\ast}) %
+\omega^{G}_{\lambda}(TT^{n}z,TT^{n}z,x^{\ast})]} %
{1+\omega^{G}_{\lambda}(TT^{n}z,TT^{n}z,x^{\ast}) %
+\omega^{G}_{\lambda}(TT^{n}z,Tx^{\ast},Tx^{\ast})}}. %
\end{align}%

So that, %

\begin{align}%
\psi(\omega^{G}_{\lambda}(T^{n+1}z,x^{\ast},x^{\ast}))\notag %
\le&\kappa_{1}(\psi(\omega^{G}_{\lambda}(T^{n}z,x^{\ast},x^{\ast}))) %
\psi(\omega^{G}_{\lambda+\nu(\lambda)}(T^{n}z,x^{\ast},x^{\ast}))\\ \notag %
+&\kappa_{2}(\psi(\omega^{G}_{\lambda}(T^{n}z,x^{\ast},x^{\ast}))) %
\psi(\omega^{G}_{\lambda}(T^{n}z,T^{n+1}z,T^{n+1}z))\\ \notag %
+&\kappa_{3}(\psi(\omega^{G}_{\lambda}(T^{n}z,x^{\ast},x^{\ast}))) %
\psi(\omega^{G}_{\lambda}(x^{\ast},Tx^{\ast},Tx^{\ast}))\\ \notag %
+&\kappa_{4}(\psi(\omega^{G}_{\lambda}(T^{n}z,x^{\ast},x^{\ast}))) %
\psi(\omega^{G}_{\lambda}(T^{n+1}z,Tx^{\ast},Tx^{\ast}))\\ \notag %
+&\kappa_{5}(\psi(\omega^{G}_{\lambda}(T^{n}z,x^{\ast},x^{\ast}))) %
\psi(\omega^{G}_{\lambda}(T^{n+1}z,x^{\ast},x^{\ast}))\\\notag %
+&\kappa_{6}(\psi(\omega^{G}_{\lambda}(T^{n}z,x^{\ast},x^{\ast})))\\ \notag %
\times&\psi\Set{\frac{\omega^{G}_{\lambda}(x^{\ast},Tx^{\ast},Tx^{\ast}) %
[1+\omega^{G}_{\lambda}(T^{n}z,x^{\ast},x^{\ast})]} %
{1+\omega^{G}_{\lambda}(T^{n}z,T^{n+1}z,T^{n}z)}}\\ \notag %
+&\kappa_{7}(\psi(\omega^{G}_{\lambda}(T^{n}z,x^{\ast},x^{\ast})))\\ \notag %
\times&\psi\Set{\frac{\omega^{G}_{\lambda}(x^{\ast},Tx^{\ast},Tx^{\ast}) %
[1+\omega^{G}_{\lambda}(T^{n}z,T^{n+1}z,T^{n+1}z)]} %
{1+\omega^{G}_{\lambda}(T^{n}z,x^{\ast},x^{\ast})}}\\ \notag %
+&\kappa_{8}(\psi(\omega^{G}_{\lambda}(T^{n}z,x^{\ast},x^{\ast})))\\ \notag %
\times&\psi\Set{\frac{\omega^{G}_{\lambda}(T^{n}z,T^{n+1}z,T^{n+1}z) %
[1+\omega^{G}_{\lambda}(T^{n}z,x^{\ast},x^{\ast})]} %
{1+\omega^{G}_{\lambda}(x^{\ast},T^{n+1}z,T^{n+1}z) %
+\omega^{G}_{\lambda}(T^{n}z,x^{\ast},x^{\ast})}}\\ \notag %
+&\kappa_{9}(\psi(\omega^{G}_{\lambda}(T^{n}z,x^{\ast},x^{\ast})))\\ \notag %
\times&\psi\Set{\frac{\omega^{G}_{\lambda}(T^{n}z,x^{\ast},x^{\ast}) %
[1+\omega^{G}_{\lambda}(x^{\ast},Tx^{\ast},Tx^{\ast})]} %
{1+\omega^{G}_{\lambda}(T^{n+1}z,T^{n+1}z,x^{\ast}) %
+\omega^{G}_{\lambda}(T^{n+1}z,Tx^{\ast},Tx^{\ast})}}\\ \notag %
+&\kappa_{10}(\psi(\omega^{G}_{\lambda}(T^{n}z,x^{\ast},x^{\ast})))\\ \notag %
\times&\psi\Set{\frac{\omega^{G}_{\lambda}(T^{n+1}z,Tx^{\ast},Tx^{\ast}) %
[1+\omega^{G}_{\lambda}(T^{n}z,Tx^{\ast},Tx^{\ast})]} %
{1+\omega^{G}_{\lambda}(T^{n+1}z,T^{n+1}z,x^{\ast}) %
+\omega^{G}_{\lambda}(T^{n}z,Tx^{\ast},Tx^{\ast})}}\\ \notag %
+&\kappa_{11}(\psi(\omega^{G}_{\lambda}(T^{n}z,x^{\ast},x^{\ast})))\\ %
\times&\psi\Set{\frac{\omega^{G}_{\lambda}(T^{n+1}z,Tx^{\ast},Tx^{\ast}) %
[1+\omega^{G}_{\lambda}(x^{\ast},Tx^{\ast},Tx^{\ast}) %
+\omega^{G}_{\lambda}(T^{n+1}z,T^{n+1}z,x^{\ast})]} %
{1+\omega^{G}_{\lambda}(T^{n+1}z,T^{n+1}z,x^{\ast}) %
+\omega^{G}_{\lambda}(T^{n+1}z,Tx^{\ast},Tx^{\ast})}}, %
\end{align} %

using the fact that $Tx^{\ast}=x^{\ast}$, we get %

\begin{align}\label{B}%
\psi(\omega^{G}_{\lambda}(T^{n+1}z,x^{\ast},x^{\ast}))\notag %
\le&\kappa_{1}(\psi(\omega^{G}_{\lambda}(T^{n}z,x^{\ast},x^{\ast}))) %
\psi(\omega^{G}_{\lambda+\nu(\lambda)}(T^{n}z,x^{\ast},x^{\ast})) \notag %
+\kappa_{2}(\psi(\omega^{G}_{\lambda}(T^{n}z,x^{\ast},x^{\ast}))) %
\\\notag\times&\psi(\omega^{G}_{\lambda}(T^{n}z,T^{n+1}z,T^{n+1}z)) \notag %
+\kappa_{4}(\psi(\omega^{G}_{\lambda}(T^{n}z,x^{\ast},x^{\ast}))) %
\psi(\omega^{G}_{\lambda}(T^{n+1}z,x^{\ast},x^{\ast}))\\ \notag %
+&\kappa_{5}(\psi(\omega^{G}_{\lambda}(T^{n}z,x^{\ast},x^{\ast}))) %
\psi(\omega^{G}_{\lambda}(T^{n+1}z,x^{\ast},x^{\ast}))\notag %
+\kappa_{8}(\psi(\omega^{G}_{\lambda}(T^{n}z,x^{\ast},x^{\ast}))) %
\\\notag\times&\psi(\omega^{G}_{\lambda}(T^{n}z,T^{n+1}z,T^{n+1}z))\notag %
+\kappa_{9}(\psi(\omega^{G}_{\lambda}(T^{n}z,x^{\ast},x^{\ast}))) %
\psi(\omega_{\lambda}(T^{n}z,x^{\ast},x^{\ast}))\\\notag %
+&\kappa_{10}(\psi(\omega^{G}_{\lambda}(T^{n}z,x^{\ast},x^{\ast}))) %
\psi(\omega^{G}_{\lambda}(T^{n+1}z,x^{\ast},x^{\ast})) %
+\kappa_{11}(\psi(\omega^{G}_{\lambda}(T^{n}z,x^{\ast},x^{\ast}))) %
\\\times&\psi(\omega^{G}_{\lambda}( %
T^{n}z,x^{\ast},x^{\ast})). %
\end{align}%

Applying condition 6 of Proposition \ref{2.2}, we get %

\begin{align}%
\omega^{G}_{\lambda} %
(T^{n}z,T^{n+1}z,T^{n+1}z)\notag %
\le&\omega^{G}_{\frac{\lambda}{2}} %
(T^{n}z,x^{\ast},x^{\ast}) %
+\omega^{G}_{\frac{\lambda}{4}} %
(T^{n+1}z,x^{\ast},x^{\ast}) %
+\omega^{G}_{\frac{\lambda}{4}} %
(T^{n+1}z,x^{\ast},x^{\ast})\\\notag %
=&\omega^{G}_{\frac{\lambda}{2}} %
(T^{n}z,x^{\ast},x^{\ast}) %
+2\omega^{G}_{\frac{\lambda}{4}} %
(T^{n+1}z,x^{\ast},x^{\ast})\\ %
\le& \omega^{G}_{\lambda} %
(T^{n}z,x^{\ast},x^{\ast}) %
+2\omega^{G}_{\lambda} %
(T^{n+1}z,x^{\ast},x^{\ast}). %
\end{align}%

Therefore inequality \ref{B} and  %
noticing that $\psi$(.) is sub-additive, we have %
\begin{align}\label{}%
\psi(\omega^{G}_{\lambda}(T^{n+1}z,x^{\ast},x^{\ast}))\notag %
\le&\kappa_{1}(\psi(\omega^{G}_{\lambda}(T^{n}z,x^{\ast},x^{\ast}))) %
\psi(\omega^{G}_{\lambda}(T^{n}z,x^{\ast},x^{\ast}))\\ \notag %
+&\kappa_{2}(\psi(\omega^{G}_{\lambda}(T^{n}z,x^{\ast},x^{\ast}))) %
(\psi(\omega^{G}_{\lambda}(T^{n}z,x^{\ast},x^{\ast})) %
+2\omega^{G}_{\lambda}(T^{n+1}z,x^{\ast},x^{\ast}))\\ \notag %
+&\kappa_{4}(\psi(\omega^{G}_{\lambda}(T^{n}z,x^{\ast},x^{\ast}))) %
\psi(\omega^{G}_{\lambda}(T^{n+1}z,x^{\ast},x^{\ast}))\\ \notag %
+&\kappa_{5}(\psi(\omega^{G}_{\lambda}(T^{n}z,x^{\ast},x^{\ast}))) %
\psi(\omega^{G}_{\lambda}(T^{n+1}z,x^{\ast},x^{\ast}))\\\notag
+&\kappa_{8}(\psi(\omega^{G}_{\lambda}(T^{n}z,x^{\ast},x^{\ast}))) %
\psi(\omega^{G}_{\lambda}(T^{n}z,x^{\ast},x^{\ast}) %
+2\omega^{G}_{\lambda}(T^{n+1}z,x^{\ast},x^{\ast}))\\\notag %
+&\alpha_{9}(\psi(\omega^{G}_{\lambda}(T^{n}z,x^{\ast},x^{\ast}))) %
\psi(\omega_{\lambda}(T^{n}z,x^{\ast},x^{\ast}))\\\notag %
+&\kappa_{10}(\psi(\omega^{G}_{\lambda}(T^{n}z,x^{\ast},x^{\ast}))) %
\psi(\omega^{G}_{\lambda}(T^{n+1}z,x^{\ast},x^{\ast}))\\\notag %
+&\kappa_{11}(\psi(\omega^{G}_{\lambda}(T^{n}z,x^{\ast},x^{\ast}))) %
\psi(\omega^{G}_{\lambda}(T^{n}z,x^{\ast},x^{\ast}))\\\notag %
\le&\kappa_{1}(\psi(\omega^{G}_{\lambda}(T^{n}z,x^{\ast},x^{\ast}))) %
\psi(\omega^{G}_{\lambda}(T^{n}z,x^{\ast},x^{\ast}))\\ \notag %
+&\kappa_{2}(\psi(\omega^{G}_{\lambda}(T^{n}z,x^{\ast},x^{\ast}))) %
(\psi(\omega^{G}_{\lambda}(T^{n}z,x^{\ast},x^{\ast}))) %
+2\psi(\omega^{G}_{\lambda}(T^{n+1}z,x^{\ast},x^{\ast}))\\ \notag %
+&\kappa_{4}(\psi(\omega^{G}_{\lambda}(T^{n}z,x^{\ast},x^{\ast}))) %
\psi(\omega^{G}_{\lambda}(T^{n+1}z,x^{\ast},x^{\ast}))\\ \notag %
+&\kappa_{5}(\psi(\omega^{G}_{\lambda}(T^{n}z,x^{\ast},x^{\ast})))
\psi(\omega^{G}_{\lambda}(T^{n+1}z,x^{\ast},x^{\ast}))\\\notag %
+&\kappa_{8}(\psi(\omega^{G}_{\lambda}(T^{n}z,x^{\ast},x^{\ast}))) %
(\psi(\omega^{G}_{\lambda}(T^{n}z,x^{\ast},x^{\ast})) %
+2\psi(\omega^{G}_{\lambda}(T^{n+1}z,x^{\ast},x^{\ast})))\\\notag %
+&\kappa_{9}(\psi(\omega^{G}_{\lambda}(T^{n}z,x^{\ast},x^{\ast}))) %
\psi(\omega_{\lambda}(T^{n}z,x^{\ast},x^{\ast}))\\\notag %
+&\kappa_{10}(\psi(\omega^{G}_{\lambda}(T^{n}z,x^{\ast},x^{\ast}))) %
\psi(\omega^{G}_{\lambda}(T^{n+1}z,x^{\ast},x^{\ast}))\\\notag %
+&\kappa_{11}(\psi(\omega^{G}_{\lambda}(T^{n}z,x^{\ast},x^{\ast}))) %
\psi(\omega^{G}_{\lambda}(T^{n}z,x^{\ast},x^{\ast}))\\\notag %
=&\kappa_{1}(\psi(\omega^{G}_{\lambda}(T^{n}z,x^{\ast},x^{\ast}))) %
\psi(\omega^{G}_{\lambda}(T^{n}z,x^{\ast},x^{\ast}))\\ \notag %
+&\kappa_{2}(\psi(\omega^{G}_{\lambda}(T^{n}z,x^{\ast},x^{\ast}))) %
\psi(\omega^{G}_{\lambda}(T^{n}z,x^{\ast},x^{\ast}))\\\notag %
+&2\kappa_{2}(\psi(\omega^{G}_{\lambda}(T^{n}z,x^{\ast},x^{\ast}))) %
\psi(\omega^{G}_{\lambda}(T^{n+1}z,x^{\ast},x^{\ast}))\\ \notag %
+&\kappa_{4}(\psi(\omega^{G}_{\lambda}(T^{n}z,x^{\ast},x^{\ast}))) %
\psi(\omega^{G}_{\lambda}(T^{n+1}z,x^{\ast},x^{\ast}))\\ \notag %
+&\kappa_{5}(\psi(\omega^{G}_{\lambda}(T^{n}z,x^{\ast},x^{\ast}))) %
\psi(\omega^{G}_{\lambda}(T^{n+1}z,x^{\ast},x^{\ast}))\\\notag %
+&\kappa_{8}(\psi(\omega^{G}_{\lambda}(T^{n}z,x^{\ast},x^{\ast}))) %
\psi(\omega^{G}_{\lambda}(T^{n}z,x^{\ast},x^{\ast}))\\\notag %
+&2\kappa_{8}(\psi(\omega^{G}_{\lambda}(T^{n}z,x^{\ast},x^{\ast}))) %
\psi(\omega^{G}_{\lambda}(T^{n+1}z,x^{\ast},x^{\ast}))\\\notag %
+&\kappa_{9}(\psi(\omega^{G}_{\lambda}(T^{n}z,x^{\ast},x^{\ast}))) %
\psi(\omega_{\lambda}(T^{n}z,x^{\ast},x^{\ast}))\\\notag %
+&\kappa_{10}(\psi(\omega^{G}_{\lambda}(T^{n}z,x^{\ast},x^{\ast}))) %
\psi(\omega^{G}_{\lambda}(T^{n+1}z,x^{\ast},x^{\ast}))\\ %
+&\kappa_{11}(\psi(\omega^{G}_{\lambda}(T^{n}z,x^{\ast},x^{\ast}))) %
\psi(\omega^{G}_{\lambda}(T^{n}z,x^{\ast},x^{\ast})), %
\end{align}%

so that	%
		
\begin{eqnarray}%
&&\bigg(1-\bigg(2\kappa_{2}(\psi(\omega^{G}_{\lambda} %
(T^{n}z,x^{\ast},x^{\ast})))\notag %
+\kappa_{4}(\psi(\omega^{G}_{\lambda} %
(T^{n}z,x^{\ast},x^{\ast})))\notag %
+\kappa_{5}(\psi(\omega^{G}_{\lambda} %
(T^{n}z,x^{\ast},x^{\ast})))\\ %
&&+2\kappa_{8}(\psi(\omega^{G}_{\lambda} %
(T^{n}z,x^{\ast},x^{\ast})))\notag %
+\kappa_{10}(\psi(\omega^{G}_{\lambda} %
(T^{n}z,x^{\ast},x^{\ast})))\bigg)\bigg)\notag %
\psi(\omega^{G}_{\lambda} %
(T^{n+1}z, x^{\ast},x^{\ast}))\\\notag %
&&\le\notag %
\bigg(\kappa_{1}(\psi(\omega^{G}_{\lambda} %
(T^{n},x^{\ast},x^{\ast})))\notag %
+\kappa_{2}(\psi(\omega^{G}_{\lambda} %
(T^{n}z,x^{\ast},x^{\ast})))\notag %
+\kappa_{8}(\psi(\omega^{G}_{\lambda} %
(T^{n}z,x^{\ast},x^{\ast})))\\ %
&&+\kappa_{9}(\psi(\omega^{G}_{\lambda} %
(T^{n}z,x^{\ast},x^{\ast}))) %
+\kappa_{11}(\psi(\omega^{G}_{\lambda} %
(T^{n}z,x^{\ast},x^{\ast})))\bigg) %
\psi(\omega^{G}_{\lambda} %
(T^{n}z,x^{\ast},x^{\ast})). %
\end{eqnarray}%
		
Therefore, %

\begin{align}%
\psi(\omega^{G}_{\lambda} %
(T^{n+1}z,x^{\ast},x^{\ast}))\notag %
\le& \delta\psi(\omega^{G}_{\lambda} %
(T^{n}z,x^{\ast},x^{\ast}))\\\notag %
\le&\psi(\omega^{G}_{\lambda} %
(T^{n}z,x^{\ast},x^{\ast}))\\\notag %
\le&\psi(\omega^{G}_{\lambda} %
(T^{n-1}z,x^{\ast},x^{\ast}))\\\notag %
\vdots&\\\notag %
\le&\psi(\omega^{G}_{\lambda} %
(z,x^{\ast},x^{\ast}))\\ %
< &\infty, %
\end{align}

where $\delta=\dfrac{\delta_{1}}{\delta_{2}}.$ %

\begin{eqnarray}%
\delta_{1}:=&& 1-\bigg(2\kappa_{2}( %
\psi(\omega^{G}_{\lambda} %
(T^{n}z,x^{\ast},x^{\ast})))\notag %
+\kappa_{4}(\psi(\omega^{G}_{\lambda} %
(T^{n}z,x^{\ast},x^{\ast})))\notag %
+\kappa_{5}(\psi(\omega^{G}_{\lambda} %
(T^{n}z,x^{\ast},x^{\ast})))\\ %
&&+2\kappa_{8}(\psi(\omega^{G}_{\lambda} %
(T^{n}z,x^{\ast},x^{\ast}))) %
+\kappa_{10}(\psi(\omega^{G}_{\lambda} %
(T^{n}z,x^{\ast},x^{\ast})))\bigg), %
\end{eqnarray}%

and%

\begin{eqnarray}%
\delta_{2}:=&&\bigg(\kappa_{1}(\psi(\omega^{G}_{\lambda} %
(T^{n},x^{\ast},x^{\ast})))\notag %
+\kappa_{2}(\psi(\omega^{G}_{\lambda} %
(T^{n}z,x^{\ast},x^{\ast})))\notag %
+\kappa_{8}(\psi(\omega^{G}_{\lambda} %
(T^{n}z,x^{\ast},x^{\ast})))\\ %
&&+\kappa_{9}(\psi(\omega^{G}_{\lambda} %
(T^{n}z,x^{\ast},x^{\ast}))) %
+\kappa_{11}(\psi(\omega^{G}_{\lambda} %
(T^{n}z,x^{\ast},x^{\ast})))\bigg). %
\end{eqnarray}%
		
Therefore, $\{\psi(\omega^{G}_{\lambda} %
(T^{n+1}z,x^{\ast},x^{\ast}))\}_{n\ge 1}$ %
is non-increasing sequence and bounded below. %
 So, it converges to some real number $\beta\ge 0$ (say). %
For $\beta>0$. Recall that $\lim\limits_{n\rightarrow\infty} %
\omega^{G}_{\lambda}(T^{n}x_{0},T^{n+1}x_{0},T^{n+1}x_{0})=0 $ %
for all $\lambda>0$, that is $\{T^{n}x_{0}\}_{n\in\mathbb{N}}$ %
is asymptotically $T$-regular on some point $x_{0}\in %
X_{\omega^{G}}$, Thus, letting $n\rightarrow \infty$  %
in the inequality \ref{B}, %
and using the fact that $\psi(\omega^{G}_{\lambda} %
(T^{n}z,x^{\ast},x^{\ast}))\leq \psi(\omega^{G}_{\lambda} %
(T^{n+1}z,x^{\ast},x^{\ast}))$, we have  that %
\begin{align}%
1\notag &\le\lim\limits_{n\rightarrow\infty} %
\inf\bigg(\kappa_{1}(\psi(\omega^{G}_{\lambda} %
(T^{n}z,x^{\ast},x^{\ast})))\notag %
+\kappa_{4}(\psi(\omega^{G}_{\lambda} %
(T^{n}z,x^{\ast},x^{\ast})))\\\notag %
&+\kappa_{5}(\psi(\omega^{G}_{\lambda} %
(T^{n}z,x^{\ast},x^{\ast})))\notag %
+\kappa_{9}(\psi(\omega^{G}_{\lambda} %
(T^{n}z,x^{\ast},x^{\ast})))\\\notag %
&+\kappa_{10}(\psi(\omega^{G}_{\lambda} %
(T^{n}z,x^{\ast},x^{\ast}))) %
+\kappa_{11}(\psi(\omega^{G}_{\lambda} %
(T^{n}z,x^{\ast},x^{\ast})))\bigg). %
\end{align}%

Therefore, $\lim\limits_{n\rightarrow\infty}  %
\psi(\omega^{G}_{\lambda}(T^{n}z,x^{\ast},x^{\ast}))=0$,  %
for which by condition 4 of Proposition %
\ref{2.3}, we have  that %

\begin{equation}%
\lim\limits_{n\rightarrow\infty} %
\omega^{G}_{\lambda}(T^{n}z,x^{\ast},x^{\ast}))=0  %
~\forall~\lambda>0. %
\end{equation}%

 Hence $T^{n}z\rightarrow x^{\ast}$ as %
  $n\rightarrow\infty.$ %
		
Again, $T^{n}z$ is also comparable with %
$y^{\ast}$ for each $n\in\mathbb{N}$. %
Now for $\lambda>0$, %
using inequality \ref{a}, we have by  %
taking $x=T^{n}z$ and $y=y^{\ast}.$ %
Again, consider $\psi(\omega^{G}_{ %
\lambda}(T^{n+1}z,y^{\ast},y^{\ast}))<\infty,$ %
so that we have the following %

\begin{align}%
\psi(\omega^{G}_{\lambda} %
(T^{n+1}z,y^{\ast},y^{\ast}))\notag %
=&\psi(\omega_{\lambda} %
(T^{n+1}z,Ty^{\ast},Ty^{\ast}))\\\notag %
=&\psi(\omega^{G}_{\lambda} %
(TT^{n}z,Ty^{\ast},Ty^{\ast}))\\\notag %
\le&\kappa_{1}(\psi(\omega^{G}_{\lambda} %
(T^{n}z,y^{\ast},y^{\ast}))) %
\psi(\omega^{G}_{\lambda+\nu(\lambda)} %
(T^{n}z,y^{\ast},y^{\ast}))\\ \notag %
+&\kappa_{2}(\psi(\omega^{G}_{\lambda} %
(T^{n}z,y^{\ast},y^{\ast}))) %
\psi(\omega^{G}_{\lambda} %
(T^{n}z,TT^{n}z,TT^{n}z))\\ \notag %
+&\kappa_{3}(\psi(\omega^{G}_{\lambda} %
(T^{n}z,y^{\ast},y^{\ast}))) %
\psi(\omega^{G}_{\lambda} %
(y^{\ast},Ty^{\ast},Ty^{\ast}))\\ \notag %
+&\kappa_{4}(\psi(\omega^{G}_{\lambda} %
(T^{n}z,y^{\ast},y^{\ast}))) %
\psi(\omega^{G}_{\lambda} %
(TT^{n}z,Ty^{\ast},Ty^{\ast}))\\ \notag %
+&\kappa_{5}(\psi(\omega^{G}_{\lambda} %
(T^{n}z,y^{\ast},y^{\ast}))) %
\psi(\omega^{G}_{\lambda} %
(TT^{n}z,y^{\ast},y^{\ast}))\\\notag %
+&\kappa_{6}(\psi(\omega^{G}_{\lambda} %
(T^{n}z,y^{\ast},y^{\ast})))\\ \notag %
\times&\psi\Set{\frac{\omega^{G}_{\lambda} %
(y^{\ast},Ty^{\ast},Ty^{\ast}) %
[1+\omega^{G}_{\lambda} %
(T^{n}z,y^{\ast},y^{\ast})]} %
{1+\omega^{G}_{\lambda} %
(T^{n}z,TT^{n}z,T^{n}z)}}\\ \notag %
+&\kappa_{7}(\psi(\omega^{G}_{\lambda} %
(T^{n}z,y^{\ast},y^{\ast})))\\ \notag %
\times&\psi\Set{\frac{\omega^{G}_{\lambda} %
(y^{\ast},Ty^{\ast},Ty^{\ast}) %
[1+\omega^{G}_{\lambda}(T^{n}z,TT^{n}z,TT^{n}z)]} %
{1+\omega^{G}_{\lambda} %
(T^{n}z,y^{\ast},y^{\ast})}}\\ \notag %
+&\kappa_{8}(\psi(\omega^{G}_{\lambda} %
(T^{n}z,y^{\ast},y^{\ast})))\\ \notag %
\times&\psi\Set{\frac{\omega^{G}_{\lambda} %
(T^{n}z,TT^{n}z,TT^{n}z) %
[1+\omega^{G}_{\lambda}(T^{n}z,y^{\ast},y^{\ast})]} %
{1+\omega^{G}_{\lambda}(y^{\ast},TT^{n}z,TT^{n}z) %
+\omega^{G}_{\lambda}(T^{n}z,y^{\ast},y^{\ast})}}\\ \notag %
+&\kappa_{9}(\psi(\omega^{G}_{\lambda} %
(T^{n}z,y^{\ast},y^{\ast})))\\ \notag %
\times&\psi\Set{\frac{\omega^{G}_{\lambda} %
(T^{n}z,y^{\ast},y^{\ast})  %
[1+\omega^{G}_{\lambda}(y^{\ast},Ty^{\ast},Ty^{\ast})]} %
{1+\omega^{G}_{\lambda}(TT^{n}z,TT^{n}z,y^{\ast}) %
+\omega^{G}_{\lambda}(TT^{n}z,Ty^{\ast},Ty^{\ast})}}\\ \notag %
+&\kappa_{10}(\psi(\omega^{G}_{\lambda} %
(T^{n}z,y^{\ast},y^{\ast})))\\ \notag %
\times&\psi\Set{\frac{\omega^{G}_{\lambda} %
(TT^{n}z,Ty^{\ast},Ty^{\ast}) %
[1+\omega^{G}_{\lambda}(T^{n}z,Ty^{\ast},Ty^{\ast})]} %
{1+\omega^{G}_{\lambda}(TT^{n}z,TT^{n}z,y^{\ast}) %
+\omega^{G}_{\lambda}(T^{n}z,Ty^{\ast},Ty^{\ast})}}\\ \notag %
+&\kappa_{11}(\psi(\omega^{G}_{\lambda} %
(T^{n}z,y^{\ast},y^{\ast})))\\%
\times&\psi\Set{\frac{\omega^{G}_{\lambda}( %
TT^{n}z,Ty^{\ast},Ty^{\ast}) %
[1+\omega^{G}_{\lambda}(y^{\ast},Ty^{\ast},Ty^{\ast})
+\omega^{G}_{\lambda}(TT^{n}z,TT^{n}z,y^{\ast})]} %
{1+\omega^{G}_{\lambda}(TT^{n}z,TT^{n}z,y^{\ast}) %
+\omega^{G}_{\lambda}(TT^{n}z,Ty^{\ast},Ty^{\ast})}}, %
\end{align}%

so that %

\begin{align}%
\psi(\omega^{G}_{\lambda}(T^{n+1}z,y^{\ast},y^{\ast}))\notag %
\le&\kappa_{1}(\psi(\omega^{G}_{\lambda}(T^{n}z,y^{\ast},y^{\ast}))) %
\psi(\omega^{G}_{\lambda+\nu(\lambda)}(T^{n}z,y^{\ast},y^{\ast}))\\ \notag %
+&\kappa_{2}(\psi(\omega^{G}_{\lambda}(T^{n}z,y^{\ast},y^{\ast}))) %
\psi(\omega^{G}_{\lambda}(T^{n}z,T^{n+1}z,T^{n+1}z))\\ \notag %
+&\kappa_{3}(\psi(\omega^{G}_{\lambda}(T^{n}z,y^{\ast},y^{\ast}))) %
\psi(\omega^{G}_{\lambda}(y^{\ast},Ty^{\ast},Ty^{\ast}))\\ \notag %
+&\kappa_{4}(\psi(\omega^{G}_{\lambda}(T^{n}z,y^{\ast},y^{\ast}))) %
\psi(\omega^{G}_{\lambda}(T^{n+1}z,Ty^{\ast},Ty^{\ast}))\\ \notag %
+&\kappa_{5}(\psi(\omega^{G}_{\lambda}(T^{n}z,y^{\ast},y^{\ast}))) %
\psi(\omega^{G}_{\lambda}(T^{n+1}z,y^{\ast},y^{\ast}))\\\notag %
+&\kappa_{6}(\psi(\omega^{G}_{\lambda}(T^{n}z,y^{\ast},y^{\ast})))\\ \notag %
\times&\psi\Set{\frac{\omega^{G}_{\lambda}(y^{\ast},Ty^{\ast},Ty^{\ast}) %
[1+\omega^{G}_{\lambda}(T^{n}z,y^{\ast},y^{\ast})]}{1 %
+\omega^{G}_{\lambda}(T^{n}z,T^{n+1}z,T^{n}z)}}\\ \notag %
+&\kappa_{7}(\psi(\omega^{G}_{\lambda}(T^{n}z,y^{\ast},y^{\ast})))\\ \notag %
\times&\psi\Set{\frac{\omega^{G}_{\lambda}(y^{\ast},Ty^{\ast},Ty^{\ast}) %
[1+\omega^{G}_{\lambda}(T^{n}z,T^{n+1}z,T^{n+1}z)]}{1 %
+\omega^{G}_{\lambda}(T^{n}z,y^{\ast},y^{\ast})}}\\ \notag %
+&\kappa_{8}(\psi(\omega^{G}_{\lambda}(T^{n}z,y^{\ast},y^{\ast})))\\ \notag %
\times&\psi\Set{\frac{\omega^{G}_{\lambda}(T^{n}z,T^{n+1}z,T^{n+1}z) %
[1+\omega^{G}_{\lambda}(T^{n}z,y^{\ast},y^{\ast})]} %
{1+\omega^{G}_{\lambda}(y^{\ast},T^{n+1}z,T^{n+1}z) %
+\omega^{G}_{\lambda}(T^{n}z,y^{\ast},y^{\ast})}}\\ \notag %
+&\kappa_{9}(\psi(\omega^{G}_{\lambda}(T^{n}z,y^{\ast},y^{\ast})))\\ \notag %
\times&\psi\Set{\frac{\omega^{G}_{\lambda}(T^{n}z,y^{\ast},y^{\ast}) %
[1+\omega^{G}_{\lambda}(y^{\ast},Ty^{\ast},Ty^{\ast})]} %
{1+\omega^{G}_{\lambda}(T^{n+1}z,T^{n+1}z,y^{\ast}) %
+\omega^{G}_{\lambda}(T^{n+1}z,Ty^{\ast},Ty^{\ast})}}\\\notag %
+&\kappa_{10}(\psi(\omega^{G}_{\lambda}(T^{n}z,y^{\ast},y^{\ast})))\\ \notag %
\times&\psi\Set{\frac{\omega^{G}_{\lambda}(T^{n+1}z,Ty^{\ast},Ty^{\ast}) %
[1+\omega^{G}_{\lambda}(T^{n}z,Ty^{\ast},Ty^{\ast})]} %
{1+\omega^{G}_{\lambda}(T^{n+1}z,T^{n+1}z,y^{\ast}) %
+\omega^{G}_{\lambda}(T^{n}z,Ty^{\ast},Ty^{\ast})}}\\ \notag %
+&\kappa_{11}(\psi(\omega^{G}_{\lambda}(T^{n}z,y^{\ast},y^{\ast})))\\ %
\times&\psi\Set{\frac{\omega^{G}_{\lambda}(T^{n+1}z,Ty^{\ast},Ty^{\ast}) %
[1+\omega^{G}_{\lambda}(y^{\ast},Ty^{\ast},Ty^{\ast}) %
+\omega^{G}_{\lambda}(T^{n+1}z,T^{n+1}z,y^{\ast})]} %
{1+\omega^{G}_{\lambda}(T^{n+1}z,T^{n+1}z,y^{\ast}) %
+\omega^{G}_{\lambda}(T^{n+1}z,Ty^{\ast},Ty^{\ast})}}, %
\end{align}%
	
using the fact that $Ty^{\ast}=y^{\ast}$, we get %

\begin{align}\label{C}%
\psi(\omega^{G}_{\lambda}(T^{n+1}z,y^{\ast},y^{\ast}))\notag %
\le&\kappa_{1}(\psi(\omega^{G}_{\lambda}(T^{n}z,y^{\ast},y^{\ast}))) %
\psi(\omega^{G}_{\lambda+\nu(\lambda)}(T^{n}z,y^{\ast},y^{\ast}))\\ \notag %
+&\kappa_{2}(\psi(\omega^{G}_{\lambda}(T^{n}z,y^{\ast},y^{\ast}))) %
\psi(\omega^{G}_{\lambda}(T^{n}z,T^{n+1}z,T^{n+1}z))\\\notag %
+&\kappa_{4}(\psi(\omega^{G}_{\lambda}(T^{n}z,y^{\ast},y^{\ast}))) %
\psi(\omega^{G}_{\lambda}(T^{n+1}z,y^{\ast},y^{\ast}))\\ \notag %
+&\kappa_{5}(\psi(\omega^{G}_{\lambda}(T^{n}z,y^{\ast},y^{\ast}))) %
\psi(\omega^{G}_{\lambda}(T^{n+1}z,y^{\ast},y^{\ast}))\\\notag %
+&\kappa_{8}(\psi(\omega^{G}_{\lambda}(T^{n}z,y^{\ast},y^{\ast}))) %
\psi(\omega^{G}_{\lambda}(T^{n}z,T^{n+1}z,T^{n+1}z))\\\notag %
+&\kappa_{9}(\psi(\omega^{G}_{\lambda}(T^{n}z,y^{\ast},y^{\ast}))) %
\psi(\omega_{\lambda}(T^{n}z,y^{\ast},y^{\ast}))\\\notag %
+&\kappa_{10}(\psi(\omega^{G}_{\lambda}(T^{n}z,y^{\ast},y^{\ast}))) %
\psi(\omega^{G}_{\lambda}(T^{n+1}z,y^{\ast},y^{\ast}))\\ %
+&\kappa_{11}(\psi(\omega^{G}_{\lambda}(T^{n}z,y^{\ast},y^{\ast}))) %
\psi(\omega^{G}_{\lambda}(T^{n}z,y^{\ast},y^{\ast})). %
\end{align}%

Now, applying condition 6 of Proposition \ref{2.2}, we get %

\begin{align}%
\omega^{G}_{\lambda}(T^{n}z,T^{n+1}z,T^{n+1}z)\notag  %
\le&\omega^{G}_{\frac{\lambda}{2}}(T^{n}z,y^{\ast},y^{\ast}) %
+\omega^{G}_{\frac{\lambda}{4}}(T^{n+1}z,y^{\ast},y^{\ast}) %
+\omega^{G}_{\frac{\lambda}{4}}(T^{n+1}z,y^{\ast},y^{\ast})\\\notag %
=&\omega^{G}_{\frac{\lambda}{2}}(T^{n}z,y^{\ast},y^{\ast}) %
+2\omega^{G}_{\frac{\lambda}{4}}(T^{n+1}z,y^{\ast},y^{\ast})\\ %
\le& \omega^{G}_{\lambda}(T^{n}z,y^{\ast},y^{\ast}) %
+2\omega^{G}_{\lambda}(T^{n+1}z,y^{\ast},y^{\ast}). %
\end{align}%

Therefore, from inequality \ref{C} and  %
that $\psi$(.) is sub-additive, we have %
	
\begin{align}\label{}
\psi(\omega^{G}_{\lambda}(T^{n+1}z,y^{\ast},y^{\ast}))\notag %
\le&\kappa_{1}(\psi(\omega^{G}_{\lambda}(T^{n}z,y^{\ast},y^{\ast}))) %
\psi(\omega^{G}_{\lambda}(T^{n}z,y^{\ast},y^{\ast}))\\ \notag %
+&\kappa_{2}(\psi(\omega^{G}_{\lambda}(T^{n}z,y^{\ast},y^{\ast}))) %
(\psi(\omega^{G}_{\lambda}(T^{n}z,y^{\ast},y^{\ast})) %
+2\omega^{G}_{\lambda}(T^{n+1}z,y^{\ast},y^{\ast}))\\ \notag %
+&\kappa_{4}(\psi(\omega^{G}_{\lambda}(T^{n}z,y^{\ast},y^{\ast}))) %
\psi(\omega^{G}_{\lambda}(T^{n+1}z,y^{\ast},y^{\ast}))\\ \notag %
+&\kappa_{5}(\psi(\omega^{G}_{\lambda}(T^{n}z,y^{\ast},y^{\ast}))) %
\psi(\omega^{G}_{\lambda}(T^{n+1}z,y^{\ast},y^{\ast}))\\\notag %
+&\kappa_{8}(\psi(\omega^{G}_{\lambda}(T^{n}z,y^{\ast},y^{\ast}))) %
\psi(\omega^{G}_{\lambda}(T^{n}z,y^{\ast},y^{\ast}) %
+2\omega^{G}_{\lambda}(T^{n+1}z,y^{\ast},y^{\ast}))\\\notag %
+&\kappa_{9}(\psi(\omega^{G}_{\lambda}(T^{n}z,y^{\ast},y^{\ast}))) %
\psi(\omega_{\lambda}(T^{n}z,y^{\ast},y^{\ast})) %
\\\notag %
+&\kappa_{10}(\psi(\omega^{G}_{\lambda}(T^{n}z,y^{\ast},y^{\ast}))) %
\psi(\omega^{G}_{\lambda}(T^{n+1}z,y^{\ast},y^{\ast}))\\\notag %
+&\kappa_{11}(\psi(\omega^{G}_{\lambda}(T^{n}z,y^{\ast},y^{\ast}))) %
\psi(\omega^{G}_{\lambda}(T^{n}z,y^{\ast},y^{\ast}))\\\notag %
\le&\kappa_{1}(\psi(\omega^{G}_{\lambda}(T^{n}z,y^{\ast},y^{\ast}))) %
\psi(\omega^{G}_{\lambda}(T^{n}z,y^{\ast},y^{\ast}))\\ \notag %
+&\kappa_{2}(\psi(\omega^{G}_{\lambda}(T^{n}z,y^{\ast},y^{\ast}))) %
(\psi(\omega^{G}_{\lambda}(T^{n}z,y^{\ast},y^{\ast}))) %
+2\psi(\omega^{G}_{\lambda}(T^{n+1}z,y^{\ast},y^{\ast}))\\ \notag %
+&\kappa_{4}(\psi(\omega^{G}_{\lambda}(T^{n}z,y^{\ast},y^{\ast}))) %
\psi(\omega^{G}_{\lambda}(T^{n+1}z,y^{\ast},y^{\ast}))\\ \notag %
+&\kappa_{5}(\psi(\omega^{G}_{\lambda}(T^{n}z,y^{\ast},y^{\ast}))) %
\psi(\omega^{G}_{\lambda}(T^{n+1}z,y^{\ast},y^{\ast}))\\\notag %
+&\kappa_{8}(\psi(\omega^{G}_{\lambda}(T^{n}z,y^{\ast},y^{\ast}))) %
(\psi(\omega^{G}_{\lambda}(T^{n}z,y^{\ast},y^{\ast})) %
+2\psi(\omega^{G}_{\lambda}(T^{n+1}z,y^{\ast},y^{\ast})))\\\notag %
+&\kappa_{9}(\psi(\omega^{G}_{\lambda}(T^{n}z,y^{\ast},y^{\ast}))) %
\psi(\omega_{\lambda}(T^{n}z,y^{\ast},y^{\ast})) %
\\\notag%
+&\kappa_{10}(\psi(\omega^{G}_{\lambda}(T^{n}z,y^{\ast},y^{\ast}))) %
\psi(\omega^{G}_{\lambda}(T^{n+1}z,y^{\ast},y^{\ast}))\\\notag %
+&\kappa_{11}(\psi(\omega^{G}_{\lambda}(T^{n}z,y^{\ast},y^{\ast}))) %
\psi(\omega^{G}_{\lambda}(T^{n}z,y^{\ast},y^{\ast}))\\\notag %
=&\kappa_{1}(\psi(\omega^{G}_{\lambda}(T^{n}z,y^{\ast},y^{\ast}))) %
\psi(\omega^{G}_{\lambda}(T^{n}z,y^{\ast},y^{\ast}))\\\notag %
+&\kappa_{2}(\psi(\omega^{G}_{\lambda}(T^{n}z,y^{\ast},y^{\ast}))) %
\psi(\omega^{G}_{\lambda}(T^{n}z,y^{\ast},y^{\ast}))\\\notag %
+&2\kappa_{2}(\psi(\omega^{G}_{\lambda}(T^{n}z,y^{\ast},y^{\ast}))) %
\psi(\omega^{G}_{\lambda}(T^{n+1}z,y^{\ast},y^{\ast})\\ \notag %
+&\kappa_{4}(\psi(\omega^{G}_{\lambda}(T^{n}z,y^{\ast},y^{\ast}))) %
\psi(\omega^{G}_{\lambda}(T^{n+1}z,y^{\ast},y^{\ast}))\\ \notag %
+&\kappa_{5}(\psi(\omega^{G}_{\lambda}(T^{n}z,y^{\ast},y^{\ast}))) %
\psi(\omega^{G}_{\lambda}(T^{n+1}z,y^{\ast},y^{\ast}))\\\notag %
+&\kappa_{8}(\psi(\omega^{G}_{\lambda}(T^{n}z,y^{\ast},y^{\ast}))) %
\psi(\omega^{G}_{\lambda}(T^{n}z,y^{\ast},y^{\ast}))\\\notag %
+&2\kappa_{8}(\psi(\omega^{G}_{\lambda}(T^{n}z,y^{\ast},y^{\ast}))) %
\psi(\omega^{G}_{\lambda}(T^{n+1}z,y^{\ast},y^{\ast}))\\\notag %
+&\kappa_{9}(\psi(\omega^{G}_{\lambda}(T^{n}z,y^{\ast},y^{\ast}))) %
\psi(\omega_{\lambda}(T^{n}z,y^{\ast},y^{\ast})) %
\\\notag%
+&\kappa_{10}(\psi(\omega^{G}_{\lambda}(T^{n}z,y^{\ast},y^{\ast}))) %
\psi(\omega^{G}_{\lambda}(T^{n+1}z,y^{\ast},y^{\ast}))\\ %
+&\kappa_{11}(\psi(\omega^{G}_{\lambda}(T^{n}z,y^{\ast},y^{\ast}))) %
\psi(\omega^{G}_{\lambda}(T^{n}z,y^{\ast},y^{\ast})), %
\end{align}%

so that %
	
\begin{eqnarray}%
&&\bigg(1-\bigg(2\kappa_{2}(\psi(\omega^{G}_{ %
\lambda}(T^{n}z,y^{\ast},y^{\ast})))\notag %
+\kappa_{4}(\psi(\omega^{G}_{\lambda} %
(T^{n}z,y^{\ast},y^{\ast})))\notag %
+\kappa_{5}(\psi(\omega^{G}_{\lambda} %
(T^{n}z,y^{\ast},y^{\ast})))\\ %
&&+2\kappa_{8}(\psi(\omega^{G}_{\lambda} %
(T^{n}z,y^{\ast},y^{\ast})))\notag %
+\kappa_{10}(\psi(\omega^{G}_{\lambda} %
(T^{n}z,y^{\ast},y^{\ast})))\bigg)\bigg)\notag %
\psi(\omega^{G}_{\lambda}(T^{n+1}z, %
 y^{\ast},y^{\ast}))\\\notag %
\end{eqnarray}%

\begin{eqnarray}%
\le\notag %
&&\bigg(\kappa_{1}(\psi( %
\omega^{G}_{\lambda}(T^{n}, %
y^{\ast},y^{\ast})))\notag %
+\kappa_{2}(\psi(\omega^{G}_{ %
\lambda}(T^{n}z,y^{\ast},y^{\ast})))\notag %
+\kappa_{8}(\psi(\omega^{G}_{ %
\lambda}(T^{n}z,y^{\ast},y^{\ast})))\\ %
&&+\kappa_{9}(\psi(\omega^{G}_{ %
\lambda}(T^{n}z,y^{\ast},y^{\ast}))) %
+\kappa_{11}(\psi(\omega^{G}_{ %
\lambda}(T^{n}z,y^{\ast},y^{\ast})))\bigg) %
\psi(\omega^{G}_{\lambda}(T^{n}z, %
y^{\ast},y^{\ast})). %
\end{eqnarray}%
	
Therefore, %

\begin{align} %
\psi(\omega^{G}_{\lambda}( %
T^{n+1}z,y^{\ast},y^{\ast}))\notag %
\le& \tau\psi(\omega^{G}_{ %
\lambda}(T^{n}z,y^{\ast},y^{\ast}))\\\notag %
\le&\psi(\omega^{G}_{\lambda} %
(T^{n}z,y^{\ast},y^{\ast}))\\\notag %
\le&\psi(\omega^{G}_{\lambda} %
(T^{n-1}z,y^{\ast},y^{\ast}))\\\notag %
\vdots&\\\notag %
\le&\psi(\omega^{G}_{\lambda}(z,y^{\ast},y^{\ast}))\\ %
< &\infty, %
\end{align}%

where $\tau=\dfrac{\tau_{1}}{\tau_{2}}.$ %

\begin{eqnarray}%
\tau_{1}:=&& 1-\bigg(2\kappa_{2}( %
\psi(\omega^{G}_{\lambda}(T^{n}z, %
y^{\ast},y^{\ast})))\notag %
+\kappa_{4}(\psi(\omega^{G}_{ %
\lambda}(T^{n}z,y^{\ast},y^{\ast})))\notag %
+\kappa_{5}(\psi(\omega^{G}_{ %
\lambda}(T^{n}z,y^{\ast},y^{\ast})))\\ %
&&+2\kappa_{8}(\psi(\omega^{G}_{ %
\lambda}(T^{n}z,y^{\ast},y^{\ast})))\notag %
+\kappa_{10}(\psi(\omega^{G}_{ %
\lambda}(T^{n}z,y^{\ast},y^{\ast})))\bigg), %
\end{eqnarray}%

and %

\begin{eqnarray} %
\tau_{2}:=&&\bigg(\kappa_{1}( %
\psi(\omega^{G}_{\lambda}(T^{n}, %
y^{\ast},y^{\ast})))\notag %
+\kappa_{2}(\psi(\omega^{G}_{ %
\lambda}(T^{n}z,y^{\ast},y^{\ast})))\notag %
+\kappa_{8}(\psi(\omega^{G}_{ %
\lambda}(T^{n}z,y^{\ast},y^{\ast})))\\ %
&&+\kappa_{9}(\psi(\omega^{G}_{ %
\lambda}(T^{n}z,y^{\ast},y^{\ast}))) %
+\kappa_{11}(\psi(\omega^{G}_{ %
\lambda}(T^{n}z,y^{\ast},y^{\ast})))\bigg). %
\end{eqnarray}%
	
Therefore, $\{\psi(\omega^{G}_{\lambda} %
(T^{n+1}z,y^{\ast},y^{\ast}))\}_{n\ge 1}$ %
is non-increasing sequence and bounded below. %
So, it  converges to some real number  %
$\sigma\ge 0$ (say). For $\sigma>0$. Recall that  %
$\lim\limits_{n\rightarrow\infty}\omega^{G}_{ %
\lambda}(T^{n}x_{0},T^{n+1}x_{0},T^{n+1}x_{0})=0 $ %
for all $\lambda>0$. Because the sequence $\{T^{n}x_{0}\}_{n\in\mathbb{N}}$ %
is asymptotically $T$-regular on some point $x_{0}\in X_{\omega^{G}}.$ %
Thus, letting $n\rightarrow %
\infty$ in the inequality \ref{C}, %
and using the fact that $\psi(\omega^{G}_{ %
\lambda}(T^{n}z,y^{\ast},y^{\ast}))\leq  %
\psi(\omega^{G}_{\lambda}(T^{n+1}z,y^{\ast},y^{\ast}))$, %
we have  that %
\begin{align}%
1\notag %
&\le\lim\limits_{n\rightarrow\infty}\inf\biggl\{ %
\kappa_{1}(\psi(\omega^{G}_{\lambda} %
(T^{n}z,y^{\ast},y^{\ast})))\notag %
+\kappa_{4}(\psi(\omega^{G}_{\lambda} %
(T^{n}z,y^{\ast},y^{\ast})))\notag %
+\kappa_{5}(\psi(\omega^{G}_{\lambda} %
(T^{n}z,y^{\ast},y^{\ast})))\\\notag %
&+\kappa_{9}(\psi(\omega^{G}_{\lambda} %
(T^{n}z,y^{\ast},y^{\ast})))\notag %
+\kappa_{10}(\psi(\omega^{G}_{\lambda} %
(T^{n}z,y^{\ast},y^{\ast}))) %
+\kappa_{11}(\psi(\omega^{G}_{\lambda} %
(T^{n}z,y^{\ast},y^{\ast})))\biggr\}. %
\end{align}%

Therefore, $\lim\limits_{n\rightarrow\infty} %
\psi(\omega^{G}_{\lambda}(T^{n}z,y^{\ast},y^{\ast}))=0$, %
for which by condition 4 of Proposition %
\ref{2.3}, we have that 

\begin{equation} %
\lim\limits_{n\rightarrow\infty}\omega^{G}_{\lambda} %
(T^{n}z,y^{\ast},y^{\ast}))=0 ~\forall~\lambda>0. %
\end{equation}%

Hence $T^{n}z\rightarrow y^{\ast}$ as %
$n\rightarrow\infty$. %
\newline
(f)
Now, suppose, if possible, that %
$\lim\limits_{n\rightarrow\infty}T^{n}z$ %
exists and not unique. %
Let $\lim\limits_{n\rightarrow\infty}T^{n}z=x^{\ast}$ %
and $\lim\limits_{n\rightarrow\infty}T^{n}z %
=y^{\ast}$ as we have seen above, %
where $x^{\ast}\neq y^{\ast}$. %
For each $\lambda>0$, $x^{\ast}\neq y^{\ast} %
\Rightarrow\omega^{G}_{\lambda}(x^{\ast}, %
y^{\ast},y^{\ast})>0$. If we take $\psi( %
\epsilon_{1})=\frac{1}{3}\psi(\omega_{\lambda}(x^{\ast}, %
y^{\ast},y^{\ast}))>0,$ %
then for $\lambda>0$, %
$\lim\limits_{n\rightarrow\infty}T^{n}z %
=x^{\ast}\Rightarrow$ given $\epsilon_{1}>0, %
~\exists~ m_{1}\in\mathbb{N}$ such that  %
$\psi(\omega^{G}_{\frac{\lambda}{2}}(x^{\ast},T^{n}z,T^{n}z)) %
<\psi(\epsilon_{1})$ for $n>m_{1}$. %
Again, %
$\lim\limits_{n\rightarrow\infty}T^{n}z=y^{\ast} %
\Rightarrow$ given $\epsilon_{1}>0,\exists~ m_{2} %
\in\mathbb{N}$ such that $\psi(\omega^{G}_{\frac{ %
\lambda}{4}}(y^{\ast},T^{n}z,T^{n}z))<\psi(\epsilon_{1})$ %
for $n>m_{2}$. %
Set $m=\max\{m_{1},m_{2}\}$, then for $n\geq m$, %
by condition 6 of Proposition \ref{2.2}, we have %

\begin{eqnarray}%
\psi(\omega^{G}_{\lambda}(x^{\ast}, %
y^{\ast},y^{\ast}))\notag %
&\le&\psi(\omega^{G}_{\frac{\lambda}{2}} %
(x^{\ast},T^{n}z,T^{n}z) %
+2\omega^{G}_{\frac{\lambda}{4}} %
(y^{\ast},T^{n}z,T^{n}z))\\\notag %
&\le&\psi(\omega^{G}_{\frac{\lambda}{2}} %
(x^{\ast},T^{n}z,T^{n}z)) %
+2\psi(\omega^{G}_{\frac{\lambda}{4}} %
(y^{\ast},T^{n}z,T^{n}z))\\\notag %
&<&\psi(\epsilon_{1})+2\psi(\epsilon_{1})\\\notag %
&=&3\psi(\epsilon_{1}), %
\end{eqnarray}%

which we can see clearly that $\psi(\omega^{G}_{\lambda} %
(x^{\ast}, y^{\ast},y^{\ast}))< %
\psi(\omega^{G}_{\lambda}(x^{\ast}, y^{\ast},y^{\ast}))$ %
for all $\lambda>0$. This is a contradiction. %
Hence, $x^{\ast}= y^{\ast}$. %
Therefore the fixed point of $T$ is unique and %
 hence, the proof is completed. %
		
\end{proof}%

\begin{corollary}\label{aaa}%
Let $(X,\omega^{G},\preceq)$ be  %
a complete preordered %
modular G-metric space. %
Let $T:X_{\omega^{G}}\rightarrow X_{\omega^{G}}$ %
be a non-decreasing, asymptotically $T$-regular map on %
some point $x_{0}\in X_{\omega^{G}}$ %
such that for each $\lambda>0$, %
there is $\nu(\lambda)\in %
\lbrack 0,\lambda)$ such that %
the following conditions hold: %
\begin{enumerate}%
\item[(1)]	%
\begin{align}\label{bb} %
\psi(\omega^{G}_{\lambda}(Tx,Ty,Ty))\notag %
\le& %
\kappa_{1}(\psi(\omega^{G}_{\lambda}(x,y,y))) %
\psi(\omega^{G}_{\lambda+\nu(\lambda)}(x,y,y))\\ \notag %
+&\kappa_{2}(\psi(\omega^{G}_{\lambda}(x,y,y))) %
\psi(\omega^{G}_{\lambda}(x,Tx,Tx))\\ %
+& \kappa_{3}(\psi(\omega^{G}_{\lambda}(x,y,y))) %
\psi(\omega^{G}_{\lambda}(y,Ty,Ty)), %
\end{align}%
where $\psi\in \bar{\Psi}$ and $\{\kappa_{1},\kappa_{2}, %
\kappa_{3}\}\in \mathcal{S}_{Ger}$ with $\kappa_{1}(t) %
+2\max\{\sup_{t\geq 0}\kappa_{2}(t), \sup_{t\geq 0} %
\kappa_{3}(t)\}<1,$ %
and distinct $x,y\in X_{\omega^{G}}.$ %
Assume that if a non-decreasing sequence $\{x_{n}\}_{n\in %
\mathbb{N}}$ converges to $x^{\ast}$, %
then $x_{n}\preceq x^{\ast}$ for each $n\in\mathbb{N}$; %

\item[(2)]%
if $\psi$ is sub-additive and for any $x,y\in X_{\omega^{G}}$, %
there exists $z\in X_{\omega^{G}}$ with $z\preceq Tz$ and %
$\omega^{G}_{\lambda}(z,Tz,Tz)$ is finite  for all $\lambda>0$ %
such that $z$ is comparable to both $x~{\rm and}~y$. %
Then $T$ has a fixed point in $x^{\ast}\in X_{\omega^{G}}$ and %
the sequence define by $\{T^{n}x_{0}\}_{n\ge 1}$ %
converges to $x^{\ast}$. %
Furthermore, the fixed point of $T$ is unique. %
\end{enumerate}%
\end{corollary}%

\begin{proof} %
Since $\{\kappa_{1},\kappa_{2}, \kappa_{3}\}\in %
\mathcal{S}_{Ger}$, %
then by \autoref{3.1},  %
fixed point of $T$ is unique. %
	
\end{proof}%

We have the following variant form on \autoref{3.1} %

\begin{theorem}\label{3.2}%
Let $(X,\omega^{G},\preceq)$ be a %
complete preordered modular G-metric space. %
Let $T:X_{\omega^{G}}\rightarrow X_{\omega^{G}}$ %
be a non-decreasing, asymptotically $T$-regular map on %
some point $x_{0}\in X_{\omega^{G}}$ %
such that for each $\lambda>0$ and   %
for some positive integer $m\ge 1$, %
there is $\nu(\lambda)\in [0,\lambda)$  %
such that the following conditions hold: %
\begin{align}\label{7}%
(1)~~ f_{m}\notag %
\le& %
\kappa_{1}(\psi(\omega^{G}_{\lambda}(x,y,y))) %
\psi(\omega^{G}_{\lambda+\nu(\lambda)}(x,y,y)) \notag %
+\kappa_{2}(\psi(\omega^{G}_{\lambda}(x,y,y))) %
\psi(\omega^{G}_{\lambda}(x,T^{m}x,T^{m}x))\\ \notag %
+& \kappa_{3}(\psi(\omega^{G}_{\lambda}(x,y,y))) %
\psi(\omega^{G}_{\lambda}(y,T^{m}y,T^{m}y)) \notag %
+\kappa_{4}(\psi(\omega^{G}_{\lambda}(x,y,y))) %
\psi(\omega^{G}_{\lambda}(T^{m}x,T^{m}y,T^{m}y))\\ \notag %
+&\kappa_{5}(\psi(\omega^{G}_{\lambda}(x,y,y))) %
\psi(\omega^{G}_{\lambda}(T^{m}x,y,y)) \notag %
+\kappa_{6}(\psi(\omega^{G}_{\lambda}(x,y,y)))\\\notag %
\times&\psi\Set{\frac{\omega^{G}_{\lambda}(y,T^{m}y,T^{m}y) %
[1+\omega^{G}_{\lambda}(x,y,y)]}{1 %
+\omega^{G}_{\lambda}(x,T^{m}x,T^{m}x)}}\\ \notag %
+&\kappa_{7}(\psi(\omega^{G}_{\lambda}(x,y,y)))\\\notag %
\times&\psi\Set{\frac{\omega^{G}_{\lambda}(y,T^{m}y,T^{m}y) %
[1+\omega^{G}_{\lambda}(x,T^{m}x,T^{m}x)]}{1 %
+\omega^{G}_{\lambda}(x,y,y)}}\\ \notag %
+&\kappa_{8}(\psi(\omega^{G}_{\lambda}(x,y,y)))\\\notag %
\times&\psi\Set{\frac{\omega^{G}_{\lambda}(x,T^{m}x,T^{m}x) %
[1+\omega^{G}_{\lambda}(x,y,y)]}{1+ %
\omega^{G}_{\lambda}(y,T^{m}x,T^{m}x) %
+\omega^{G}_{\lambda}(x,y,y)}}\\ \notag %
+&\kappa_{9}(\psi(\omega^{G}_{\lambda}(x,y,y)))\\\notag %
\times&\psi\Set{\frac{\omega^{G}_{\lambda}(x,y,y) %
[1+\omega^{G}_{\lambda}(y,T^{m}y,T^{m}y)]} %
{1+\omega^{G}_{\lambda}(T^{m}x,T^{m}x,y) %
+\omega^{G}_{\lambda}(T^{m}x,T^{m}y,T^{m}y)}}\\ \notag %
+&\kappa_{10}(\psi(\omega^{G}_{\lambda}(x,y,y)))\\\notag %
\times&\psi\Set{\frac{\omega^{G}_{\lambda}(T^{m}x,T^{m}y,T^{m}y) %
[1+\omega^{G}_{\lambda}(x,T^{m}y,T^{m}y)]} %
{1+\omega^{G}_{\lambda}(T^{m}x,T^{m}x,y) %
+\omega^{G}_{\lambda}(x,T^{m}y,T^{m}y)}}\\ \notag %
+&\kappa_{11}(\psi(\omega^{G}_{\lambda}(x,y,y)))\\ %
\times&\psi\Set{\frac{\omega^{G}_{\lambda}(T^{m}x,T^{m}y,T^{m}y) %
[1+\omega^{G}_{\lambda}(y,T^{m}y,T^{m}y) %
+\omega^{G}_{\lambda}(T^{m}x,T^{m}x,y)]} %
{1+\omega^{G}_{\lambda}(T^{m}x,T^{m}x,y) %
+\omega^{G}_{\lambda}(T^{m}x,T^{m}y,T^{m}y)}}, %
\end{align}%

where $f_{m}=\psi(\omega^{G}_{\lambda}(T^{m}x,T^{m}y,T^{m}y))$ %
and $\psi\in \bar{\Psi}$ , $\{\kappa_{1},\kappa_{2}, \cdots, %
\kappa_{11}\}\in \mathcal{S}_{Ger}$ and %
\newline%
$\max\{\sup_{t\geq 0}\kappa_{1}(t),\sup_{t\geq 0}\kappa_{2}(t), \cdots, %
\sup_{t\geq 0}\kappa_{11}(t)\}<1.$ %
Assume that if a non-decreasing sequence $\{x_{n}\}_{n\in\mathbb{N}}$ %
converges to $x^{\ast}$, then $x_{n}\preceq x^{\ast}$ for %
each $n\in\mathbb{N}$, %
 \newline%
$(2)$ if $\psi$ is sub-additive and for any $x,y\in X_{\omega^{G}}$, %
there exists $z\in X_{\omega^{G}}$ with $z\preceq Tz$ and %
$\omega^{G}_{\lambda}(z,Tz,Tz)$ is finite  for all $\lambda>0$ such %
that $z$ is comparable to both $x~{\rm and}~y$. %
Then $T$ has a  fixed point $x^{\ast}\in X_{\omega^{G}}$ for %
some positive integer $m\ge 1$ in $X_{\omega^{G}}$ %
and the sequence define by $\{T^{n}x_{0}\}_{n\ge 1}$  %
converges to $x_{\ast}$. %
Moreover the fixed point of $T$ is unique. %
\end{theorem}%

\begin{proof} %
By \autoref{3.1}, $T^{m}$ has a  fixed point say %
$u_{\ast}\in X_{\omega^{G}}$ for some positive %
integer $m\ge 1$ by using inequality \ref{7}. %
Now $T^{m}(Tu_{\ast})=T^{m+1}u_{\ast}=T(T^{m} %
u_{\ast})=Tu_{\ast}$, %
so $Tu_{\ast}$ is a fixed point of $T^{m}$. %
By the uniqueness of fixed point of $T^{m}$, %
we have $Tu_{\ast}=u_{\ast}$. Therefore, $u_{\ast}$ %
is a fixed point of $T$. Since fixed point of $T$ is %
also fixed point of $T^{m}$, hence $T$ has a unique %
fixed point in $X_{\omega^{G}}$. %

\end{proof}%

\begin{theorem}\label{qq}%
Let $(X,\omega^{G},\preceq)$ be a complete %
preordered modular G-metric space. %
Let $T:X_{\omega^{G}}\rightarrow X_{\omega^{G}}$ %
be a non-decreasing, asymptotically $T$-regular map on %
some point $x_{0}\in X_{\omega^{G}}$ %
such that for each $\lambda>0$ and  for  %
some positive integer $m\ge 1$, %
there is $\nu(\lambda)\in [0,\lambda)$  %
such that the following conditions hold: %
\begin{align}\label{aa}%
(1)~\psi(\omega^{G}_{\lambda}(Tx,Ty,Tz))\notag %
\le& %
\kappa_{1}(\psi(\omega^{G}_{\lambda}(x,y,z))) %
\psi(\omega^{G}_{\lambda+\nu(\lambda)}(x,y,z)) \notag%
+\kappa_{2}(\psi(\omega^{G}_{\lambda}(x,y,z))) %
\psi(\omega^{G}_{\lambda}(x,Tx,Tx))\\ \notag %
+& \kappa_{3}(\psi(\omega^{G}_{\lambda}(x,y,z))) %
\psi(\omega^{G}_{\lambda}(y,Ty,Ty)) \notag %
+\kappa_{4}(\psi(\omega^{G}_{\lambda}(x,y,z))) %
\psi(\omega^{G}_{\lambda}(z,Tz,Tz))\\ \notag %
+&\kappa_{5}(\psi(\omega^{G}_{\lambda}(x,y,z)) %
\psi(\omega^{G}_{\lambda}(Tx,y,z)) \notag %
+\kappa_{6}(\psi(\omega^{G}_{\lambda}(x,y,z)))\\\notag %
\times&\psi\Set{\frac{\omega^{G}_{\lambda}(y,Ty,Ty) %
[1+\omega^{G}_{\lambda}(x,y,z)]}{1+\omega^{G}_{ %
\lambda}(x,Tx,Tx)}}\\ \notag %
+&\kappa_{7}(\psi(\omega^{G}_{\lambda}(x,y,z)))\\\notag %
\times&\psi\Set{\frac{\omega^{G}_{\lambda}(y,Ty,Ty) %
[1+\omega^{G}_{\lambda}(x,Tx,Tx)]}{1+\omega^{G}_{ %
\lambda}(x,y,z)}}\\ \notag %
+&\kappa_{8}(\psi(\omega^{G}_{\lambda}(x,y,z)))\\\notag %
\times&\psi\Set{\frac{\omega^{G}_{\lambda}(x,Tx,Tx) %
[1+\omega^{G}_{\lambda}(x,y,z)]}{1+\omega^{G}_{ %
\lambda}(y,Tx,Tx) %
+\omega^{G}_{\lambda}(x,y,z)}}\\ \notag %
+&\kappa_{9}(\psi(\omega^{G}_{\lambda}(x,y,z)))\\\notag %
\times&\psi\Set{\frac{\omega^{G}_{\lambda}(x,y,z) %
[1+\omega^{G}_{\lambda}(y,Ty,Ty)]}{1+\omega^{G}_{ %
\lambda}(Tx,Tx,y) %
+\omega^{G}_{\lambda}(Tx,Ty,Ty)}}\\ \notag %
+&\kappa_{10}(\psi(\omega^{G}_{\lambda}(x,y,z)))\\\notag %
\times&\psi\Set{\frac{\omega^{G}_{\lambda}(Tx,Ty,Ty) %
[1+\omega^{G}_{\lambda}(x,Ty,Ty)]}{1+\omega^{G}_{ %
\lambda}(Tx,Tx,y) %
+\omega^{G}_{\lambda}(x,Ty,Ty)}}\\ \notag %
+&\kappa_{11}(\psi(\omega^{G}_{\lambda}(x,y,z)))\\ %
\times&\psi\Set{\frac{\omega^{G}_{\lambda}(Tx,Ty,Ty) %
[1+\omega^{G}_{\lambda}(y,Ty,Ty)+\omega^{G}_{ %
\lambda}(Tx,Tx,y)]} %
{1+\omega^{G}_{\lambda}(Tx,Tx,y)+\omega^{G}_{ %
\lambda}(Tx,Ty,Ty)}}, %
\end{align}%

where $\psi\in \bar{\Psi}$ , $\{\kappa_{1}, %
\kappa_{2}, \cdots, \kappa_{11}\} %
\in \mathcal{S}_{Ger}$ and %
$\max\{\sup_{t\geq 0}\kappa_{1}(t), %
\sup_{t\geq 0}\kappa_{2}(t), \cdots, %
\sup_{t\geq 0}\kappa_{11}(t)\}<1.$ %
Assume that if a non-decreasing sequence %
$\{x_{n}\}_{n\in\mathbb{N}}$ %
converges to $x^{\ast}$, then $x_{n}\preceq x^{\ast}$ %
 for each $n\in\mathbb{N}$, %
 \newline %
$(2)$ if $\psi$ is sub-additive and for any %
$x,y,z\in X_{\omega^{G}}$, %
there exists $u\in X_{\omega^{G}}$ with %
$u\preceq Tu$ and %
$\omega^{G}_{\lambda}(u,Tu,Tu)$ is finite  %
for all $\lambda>0$ %
such that $u$ is comparable to  $x,y~{\rm and}~z$. %
Then $T$ has a fixed point in $ x_{\ast} %
\in X_{\omega^{G}}$ %
and the sequence define by $\{T^{n}x_{0}\}_{n\ge 1}$ %
 converges to $x_{\ast}$. %
Moreover the fixed point of $T$ is unique. %
\end{theorem}%

\begin{proof} %
(a)
Let $x_{0}\in X_{\omega^{G}}$ be such that %
 $x_{0}\preceq Tx_{0}$ and let %
$x_{n}=Tx_{n-1}=T^{n}x_{0}$ for all $n\in  %
\mathbb{N}.$ Regarding that $T$ is %
non-decreasing mapping, we have that $x_{0} %
\preceq Tx_{0}=x_{1}$, %
which implies that $x_{1}=Tx_{0}\preceq %
Tx_{1}=x_{2}$. Inductively, %
we have %
\begin{equation}\label{pp}%
x_{0}\preceq x_{1}\preceq x_{2}\preceq \cdots \preceq x_{n-1} %
\preceq x_{n}\preceq x_{n+1}\preceq \cdots. %
\end{equation}%
	
Assume that there exists $n_{0}\in\mathbb{N}$ %
such that $x_{n_{0}}=x_{n_{0}+1}$. Since %
$x_{n_{0}}=x_{n_{0}+1}=Tx_{n_{0}}$, then %
$x_{n_{0}}$ is the fixed point of $T$. %
Now suppose that $x_{n} \precneqq x_{n+1}$ for %
all $n\in\mathbb{N}$, thus  by inequality%
\ref{pp}, we have that %
\begin{equation}\label{ppp}%
x_{0}\prec x_{1}\prec x_{2}\prec \cdots %
\prec x_{n-1}\prec x_{n}\prec x_{n+1}\prec \cdots. %
\end{equation}%
Now for each $\lambda>0,$ and $x_{0}\prec %
Tx_{0}$ for all $n\in\mathbb{N}$ %
implies that $\omega^{G}_{\lambda}(x_{0},Tx_{0},Tx_{0})>0.$ %
Again, let $x_{0}\in X_{\omega^{G}}$  such that %
$\omega^{G}_{\lambda}(x_{0},Tx_{0},Tx_{0}) %
<\infty~\forall~\lambda>0$.  %
\newline
(b)	
Next, we show  that for all $n\in \mathbb{N}$,  %
$\omega^{G}_{\lambda}(T^{n}x_{0},T^{n+1}x_{0},T^{n+1}x_{0})= 0$ %
for all $\lambda>0$, as  $n\rightarrow\infty$. This means that %
 the sequence $\{T^{n}x_{0}\}_{n\in\mathbb{N}}$ is asymptotically
 $T$-regular on some point $x_{0}\in X_{\omega^{G}}$. Assume that, %
for each $n\in\mathbb{N}$, there exists $\lambda_{n}>0$ such that %
$\omega^{G}_{\lambda_{n}}(T^{n}x_{0},T^{n+1}x_{0},T^{n+1}x_{0}) %
\neq 0$. %
Otherwise we are done. For each $n\ge 1$, if $0<\lambda<\lambda_{n}$, %
then we have $\omega^{G}_{\lambda}(T^{n}x_{0}, %
T^{n}x_{0},T^{n+1}x_{0})\neq 0$. %
Since $T^{n}x_{0}\preceq T^{n+1}x_{0}$, we have from condition (1) %
of \autoref{qq}, then we have that $\psi(\omega^{G}_{\lambda_{n}}( %
T^{n}x_{0},T^{n+1}x_{0},T^{n+1}x_{0}))\le \psi(\omega^{G}_{\lambda}( %
T^{n}x_{0},T^{n+1}x_{0},T^{n+1}x_{0})) %
=\psi(\omega^{G}_{\lambda}(TT^{n-1}x_{0},TT^{n}x_{0},TT^{n}x_{0}))$. %
Take $x=T^{n-1}x_{0}$  and $ y=T^{n}x_{0}=z$, %
by inequality \ref{aa}, we see that  $\psi(\omega^{G}_{\lambda}( %
T^{n}x_{0},T^{n+1}x_{0},T^{n+1}x_{0}))=0.$  On taking %
limit, we have $\omega^{G}_{\lambda}(T^{n}x_{0}, %
T^{n+1}x_{0},T^{n+1}x_{0})=0$ %
as $n\rightarrow \infty.$ In fact by %
\autoref{3.1}, $T$ has a unique %
fixed in $X_{\omega^{G}}.$ %
	
\end{proof}%

\begin{corollary}\label{n}%
Let $(X,\omega^{G},\preceq)$ be a complete  %
preordered modular G-metric space. %
Let $T:X_{\omega^{G}}\rightarrow X_{\omega^{G}}$ %
be a non-decreasing, asymptotically $T$-regular map %
on some point $x_{0}\in X_{\omega^{G}}$  %
such that for each $\lambda>0$, there is %
$\nu(\lambda)\in [0,\lambda)$ such that the following %
conditions hold: %
	
\begin{align}\label{o} %
(1) \psi(\omega^{G}_{\lambda}(Tx,Ty,Tz))\notag %
\le&%
\kappa_{1}(\psi(\omega^{G}_{\lambda}(x,y,z))) %
\psi(\omega^{G}_{\lambda+\nu(\lambda)}(x,y,z)) \notag %
+\kappa_{2}(\psi(\omega^{G}_{\lambda}(x,y,z))) %
\psi(\omega^{G}_{\lambda}(x,Tx,Tx))\\ %
+& \kappa_{3}(\psi(\omega^{G}_{\lambda}(x,y,z))) %
\psi(\omega^{G}_{\lambda}(y,Ty,Ty)) %
+\kappa_{4}(\psi(\omega^{G}_{\lambda}(x,y,z))) %
\psi(\omega^{G}_{\lambda}(z,Tz,Tz)), %
\end{align}%

where $\psi\in \bar{\Psi}$ and %
$\{\kappa_{1},\kappa_{2}, \kappa_{3},\kappa_{4}\}\in \mathcal{S}_{Ger}$ with %
$\kappa_{1}(t)+2\max\{\sup_{t\geq 0}\kappa_{2}(t), %
\sup_{t\geq 0}\kappa_{3}(t),\\\sup_{t\geq 0}\kappa_{4}(t)\}<1,$ %
and $x,y,z\in X_{\omega^{G}}.$ %
Assume that if a non-decreasing sequence $\{x_{n}\}_{n\in\mathbb{N}}$ %
converges to $x$, then $x_{n}\preceq x$ for each $n\in\mathbb{N}$; %

(2)~ if $\psi$ is sub-additive and for any $x,y,z\in X_{\omega^{G}}$, %
there exists $w\in X_{\omega^{G}}$ with $w\preceq Tw$ and %
$\omega^{G}_{\lambda}(w,Tw,Tw)$ is finite  for all $\lambda>0$ %
such that $w$ is comparable to both $x~{\rm ,}~y~ and~z$. %
Then $T$ has a fixed point in $x\in X_{\omega^{G}}$ and the %
sequence define by $\{T^{n}x_{0}\}_{n\ge 1}$ converges to $x$. %
Moreover the fixed point of $T$ is unique. %
\end{corollary}%
	
\begin{proof} %
(a)
Let $x_{0}\in X_{\omega^{G}}$ be such that %
$x_{0}\preceq Tx_{0}$ and let %
$x_{n}=Tx_{n-1}=T^{n}x_{0}$ for all $n\in \mathbb{N}.$ %
Regarding that $T$ is %
non-decreasing mapping, we have that $x_{0}\preceq Tx_{0}=x_{1}$, %
which implies that $x_{1}=Tx_{0}\preceq Tx_{1}=x_{2}$. %
Inductively, %
 we have %
\begin{equation}\label{3}%
x_{0}\preceq x_{1}\preceq x_{2}\preceq \cdots \preceq x_{n-1} %
\preceq x_{n}\preceq x_{n+1}\preceq \cdots. %
\end{equation}%
	
Assume that there exists $n_{0}\in\mathbb{N}$ %
such that $x_{n_{0}}=x_{n_{0}+1}$. Since %
$x_{n_{0}}=x_{n_{0}+1}=Tx_{n_{0}}$, then %
$x_{n_{0}}$ is the fixed point of $T$. %
Now suppose that $x_{n} \precneqq x_{n+1}$ for %
all $n\in\mathbb{N}$, thus  by inequality%
\ref{3}, we have that %
\begin{equation}\label{4}%
x_{0}\prec x_{1}\prec x_{2}\prec \cdots %
\prec x_{n-1}\prec x_{n}\prec x_{n+1}\prec \cdots. %
\end{equation}%
Now for each $\lambda>0,$ and $x_{0}\prec %
Tx_{0}$ for all $n\in\mathbb{N}$ %
implies that $\omega^{G}_{\lambda}(x_{0},Tx_{0},Tx_{0})>0.$ %
Again, let $x_{0}\in X_{\omega^{G}}$  such that %
$\omega^{G}_{\lambda}(x_{0},Tx_{0},Tx_{0}) %
<\infty~\forall~\lambda>0$.  %
\newline
(b)	
Now, we show  that for all $n\in \mathbb{N}$,  %
$\omega^{G}_{\lambda}(T^{n}x_{0},T^{n+1}x_{0},T^{n+1}x_{0})= 0$ %
for all $\lambda>0$, as  $n\rightarrow\infty$. It implies that %
the sequence $\{T^{n}x_{0}\}_{n\in\mathbb{N}}$ is asymptotically %
$T$-regular on some point $x_{0}\in X_{\omega^{G}}$. Assume that, %
for each $n\in\mathbb{N}$, there exists $\lambda_{n}>0$ such that %
$\omega^{G}_{\lambda_{n}}(T^{n}x_{0},T^{n+1}x_{0},T^{n+1}x_{0}) %
\neq 0$. %
Otherwise we are done. For each $n\ge 1$, if $0<\lambda<\lambda_{n}$, %
then we have $\omega^{G}_{\lambda}(T^{n}x_{0},T^{n}x_{0},T^{n+1}x_{0})\neq 0$. %
Since $T^{n}x_{0}\preceq T^{n+1}x_{0}$, we have from inequality \ref{o} that %
$\psi(\omega^{G}_{\lambda_{n}}(T^{n}x_{0},T^{n+1}x_{0},T^{n+1}x_{0}))\le  %
\psi(\omega^{G}_{\lambda}(T^{n}x_{0},T^{n+1}x_{0},T^{n+1}x_{0}))= %
\psi(\omega^{G}_{\lambda}(TT^{n-1}x_{0},TT^{n}x_{0},TT^{n}x_{0}))$. %
Take $x=T^{n-1}x_{0}$  and $ y=T^{n}x_{0}=z$, then inequality \ref{o} %
becomes %
\begin{align}\label{p}%
\psi(\omega^{G}_{\lambda_{n}}(T^{n}x_{0}, %
T^{n+1}x_{0},T^{n+1}x_{0}))\notag %
\le& %
\psi(\omega^{G}_{\lambda}(T^{n}x_{0}, %
T^{n+1}x_{0},T^{n+1}x_{0}))\\\notag %
\le& %
\kappa_{1}(\psi(\omega^{G}_{\lambda}( %
T^{n-1}x_{0},T^{n}x_{0},T^{n}x_{0})))\notag %
\psi(\omega^{G}_{\lambda+\nu(\lambda)}( %
T^{n-1}x_{0},T^{n}x_{0},T^{n}x_{0}))\\ \notag %
+&\kappa_{2}(\psi(\omega^{G}_{\lambda} %
(T^{n-1}x_{0},T^{n}x_{0},T^{n}x_{0})))\notag %
\psi(\omega^{G}_{\lambda}(T^{n-1}x_{0}, %
TT^{n-1}x_{0},TT^{n-1}x_{0}))\\\notag %
+& \kappa_{3}(\psi(\omega^{G}_{\lambda} %
(T^{n-1}x_{0},T^{n}x_{0},T^{n}x_{0})))\notag %
\psi(\omega^{G}_{\lambda}(T^{n}x_{0}, %
TT^{n}x_{0},TT^{n}x_{0}))\\\notag %
+&\kappa_{4}(\psi(\omega^{G}_{\lambda} %
(T^{n-1}x_{0},T^{n}x_{0},T^{n}x_{0})))\notag %
\psi(\omega^{G}_{\lambda}(T^{n}x_{0}, %
TT^{n}x_{0},TT^{n}x_{0}))\\\notag %
=&  %
\kappa_{1}(\psi(\omega^{G}_{\lambda}( %
T^{n-1}x_{0},T^{n}x_{0},T^{n}x_{0})))\notag %
\psi(\omega^{G}_{\lambda+\nu(\lambda)}( %
T^{n-1}x_{0},T^{n}x_{0},T^{n}x_{0}))\\ \notag %
+&\kappa_{2}(\psi(\omega^{G}_{\lambda}( %
T^{n-1}x_{0},T^{n}x_{0},T^{n}x_{0})))\notag %
\psi(\omega^{G}_{\lambda}(T^{n-1}x_{0}, %
T^{n}x_{0},T^{n}x_{0}))\\\notag %
+& \kappa_{3}(\psi(\omega^{G}_{\lambda}( %
T^{n-1}x_{0},T^{n}x_{0},T^{n}x_{0})))\notag %
\psi(\omega^{G}_{\lambda}(T^{n}x_{0}, %
T^{n+1}x_{0},T^{n+1}x_{0}))\\\notag %
+&\kappa_{4}(\psi(\omega^{G}_{\lambda}( %
T^{n-1}x_{0},T^{n}x_{0},T^{n}x_{0})))\notag %
\psi(\omega^{G}_{\lambda}(T^{n}x_{0}, %
T^{n+1}x_{0},T^{n+1}x_{0}))\\\notag %
\le& %
\kappa_{1}(\psi(\omega^{G}_{\lambda}( %
T^{n-1}x_{0},T^{n}x_{0},T^{n}x_{0})))\notag %
\psi(\omega^{G}_{\lambda}(T^{n-1}x_{0}, %
T^{n}x_{0},T^{n}x_{0}))\\ \notag %
+&\kappa_{2}(\psi(\omega^{G}_{\lambda}( %
T^{n-1}x_{0},T^{n}x_{0},T^{n}x_{0})))\notag %
\psi(\omega^{G}_{\lambda}(T^{n-1}x_{0}, %
T^{n}x_{0},T^{n}x_{0}))\\\notag %
+& \kappa_{3}(\psi(\omega^{G}_{\lambda}( %
T^{n-1}x_{0},T^{n}x_{0},T^{n}x_{0})))\notag %
\psi(\omega^{G}_{\lambda}(T^{n}x_{0}, %
T^{n+1}x_{0},T^{n+1}x_{0}))\\ %
+&\kappa_{4}(\psi(\omega^{G}_{\lambda}( %
T^{n-1}x_{0},T^{n}x_{0},T^{n}x_{0}))) %
\psi(\omega^{G}_{\lambda}(T^{n}x_{0}, %
T^{n+1}x_{0},T^{n+1}x_{0})), %
\end{align}%

which implies that %

\begin{align}\label{q} %
\psi(\omega^{G}_{\lambda}(T^{n}x_{0}, %
T^{n+1}x_{0},T^{n+1}x_{0}))\notag %
\le&h\psi(\omega^{G}_{\lambda} %
(T^{n-1}x_{0},T^{n}x_{0},T^{n}x_{0}))\\\notag %
\le& %
\psi(\omega^{G}_{\lambda}(T^{n-1}x_{0}, %
T^{n}x_{0},T^{n}x_{0}))\\\notag %
\vdots\\\notag %
\le& %
\psi(\omega^{G}_{\lambda}(x_{0},Tx_{0},Tx_{0}))\\ %
<&\infty, %
\end{align}%
where %
\begin{equation}%
h:=\dfrac{\kappa_{1}(\psi(\omega^{G}_{\lambda}%
(T^{n-1}x_{0},T^{n}x_{0},T^{n}x_{0}))) %
+\kappa_{2}(\psi(\omega^{G}_{\lambda} %
(T^{n-1}x_{0},T^{n}x_{0},T^{n}x_{0})))}{1-(\kappa_{3} %
(\psi(\omega^{G}_{\lambda}(T^{n-1}x_{0}, %
T^{n}x_{0},T^{n}x_{0})))+\kappa_{4}(\psi( %
\omega^{G}_{\lambda}(T^{n-1}x_{0},T^{n}x_{0},T^{n}x_{0}))))}. %
\end{equation}	%
Therefore, $\{\psi(\omega^{G}_{\lambda}(T^{n}x_{0}, %
T^{n+1}x_{0},T^{n+1}x_{0}))\}_{n\ge 1}$ is non-increasing  %
and bounded below. So, the sequence converges to some  %
real number $M\ge 0$, by \autoref{3.1}, the proof is completed. %

\end{proof}%
	
\begin{theorem}\label{nn}%
Let $(X,\omega^{G},\preceq)$ be a complete preordered modular %
G-metric space. %
Let $T:X_{\omega^{G}}\rightarrow X_{\omega^{G}}$ %
be a non-decreasing, asymptotically $T$-regular map on %
some point $x_{0}\in X_{\omega^{G}}$%
such that for each $\lambda>0$ and some positive integer $m\ge 1$, %
there is $\nu(\lambda)\in [0,\lambda)$ such that the  %
following conditions hold: %
\begin{enumerate}%
\item[(1)]	
\begin{align}\label{oo}%
\psi(\omega^{G}_{\lambda}(T^{m}x,T^{m}y,T^{m}z))\notag %
\le& %
\kappa_{1}(\psi(\omega^{G}_{\lambda}(x,y,z))) %
\psi(\omega^{G}_{\lambda+\nu(\lambda)}(x,y,z))\\ \notag %
+&\kappa_{2}(\psi(\omega^{G}_{\lambda}(x,y,z))) %
\psi(\omega^{G}_{\lambda}(x,T^{m}x,T^{m}x))\\\notag %
+& \kappa_{3}(\psi(\omega^{G}_{\lambda}(x,y,z))) %
\psi(\omega^{G}_{\lambda}(y,T^{m}y,T^{m}y))\\ %
+&\kappa_{4}(\psi(\omega^{G}_{\lambda}(x,y,z))) %
\psi(\omega^{G}_{\lambda}(z,T^{m}z,T^{m}z)), %
\end{align}%
where $\psi\in \bar{\Psi}$ and $\{\kappa_{1}, %
\kappa_{2}, \kappa_{3},\kappa_{4}\}\in \mathcal{S}_{Ger}$ with %
$\kappa_{1}(t)+2\max\{\sup_{t\geq 0}\kappa_{2}(t), %
\sup_{t\geq 0}\kappa_{3}(t),\\\sup_{t\geq 0}\kappa_{4}(t)\}<1,$ %
and $x,y,z\in X_{\omega^{G}}.$ %
Assume that if a non-decreasing sequence %
$\{x_{n}\}_{n\in\mathbb{N}}$ converges to $x$, then $x_{n} %
\preceq x$ for each $n\in\mathbb{N}$; %
\item[(2)] if $\psi$ is sub-additive and for any %
$x,y,z\in X_{\omega^{G}}$, there exists $w\in X_{\omega^{G}}$ %
with $w\preceq Tw$ and $\omega^{G}_{\lambda}(w,Tw,Tw)$ is %
finite  for all $\lambda>0$ such that $w$ is comparable to %
both $x~{\rm ,}~y~ and~z$. %
Then $T$ has a  fixed point $x\in X_{\omega^{G}}$ for some positive %
integer $m\ge 1$ in $X_{\omega^{G}}$ and the sequence define %
by $\{T^{n}x_{0}\}_{n\ge 1}$ converges to $x$. %
Furthermore, the fixed point of $T$ is unique. %
\end{enumerate}%
\end{theorem}%

\begin{proof} %
By Corollary \ref{n}, $T^{m}$ has %
a  fixed point say $u\in X_{\omega^{G}}$ %
by using inequality \ref{oo} for some %
positive integer $m\ge 1$. Now %
$T^{m}(Tu)=T^{m+1}u=T(T^{m}u)=Tu$, %
so $Tu$ is a fixed point of $T^{m}$. %
By the uniqueness of fixed point of %
$T^{m}$, we have $Tu=u$. Therefore, %
$u$ is a fixed point of $T$. Since %
fixed point of $T$ is also fixed %
point of $T^{m}$, hence $T$ has a %
unique fixed point in $X_{\omega^{G}}$. %

\end{proof}%

\begin{theorem}\label{} %
Let $(X,\omega^{G},\preceq)$ be a complete  preordered %
modular G-metric space. %
Let $T:X_{\omega^{G}}\rightarrow X_{\omega^{G}}$ %
be a non-decreasing, asymptotically $T$-regular map on %
some point $x_{0}\in X_{\omega^{G}}$
such that for each $\lambda>0$ %
and  for some positive integer $m\ge 1$, there is %
$\nu(\lambda)\in [0,\lambda)$ such that the following %
conditions hold: %
\begin{align}\label{k}
(1)~ h_{n}\notag%
\le&%
\kappa_{1}(\psi(\omega^{G}_{\lambda}(x,y,z)))%
\psi(\omega^{G}_{\lambda+\nu(\lambda)}(x,y,z)) \notag%
+\kappa_{2}(\psi(\omega^{G}_{\lambda}(x,y,z)))%
\psi(\omega^{G}_{\lambda}(x,T^{m}x,T^{m}x))\\ \notag%
+& \kappa_{3}(\psi(\omega^{G}_{\lambda}(x,y,z)))%
\psi(\omega^{G}_{\lambda}(y,T^{m}y,T^{m}y)) \notag%
+\kappa_{4}(\psi(\omega^{G}_{\lambda}(x,y,z)))%
\psi(\omega^{G}_{\lambda}(z,T^{m}z,T^{m}z))\\ \notag%
+&\kappa_{5}(\psi(\omega^{G}_{\lambda}(x,y,z)))%
\psi(\omega^{G}_{\lambda}(T^{m}x,y,z)) \notag%
+\kappa_{6}(\psi(\omega^{G}_{\lambda}(x,y,z)))\\\notag%
\times&\psi\Set{\frac{\omega^{G}_{\lambda}(y,T^{m}y,T^{m}y)%
[1+\omega^{G}_{\lambda}(x,y,z)]}%
{1+\omega^{G}_{\lambda}(x,T^{m}x,T^{m}x)}}\\ \notag%
+&\kappa_{7}(\psi(\omega^{G}_{\lambda}(x,y,z)))\\\notag%
\times&\psi\Set{\frac{\omega^{G}_{\lambda}(y,T^{m}y,T^{m}y)%
[1+\omega^{G}_{\lambda}(x,T^{m}x,T^{m}x)]}%
{1+\omega^{G}_{\lambda}(x,y,z)}}\\ \notag%
+&\kappa_{8}(\psi(\omega^{G}_{\lambda}(x,y,z)))\\\notag%
\times&\psi\Set{\frac{\omega^{G}_{\lambda}(x,T^{m}x,T^{m}x)%
[1+\omega^{G}_{\lambda}(x,y,z)]}%
{1+\omega^{G}_{\lambda}(y,T^{m}x,T^{m}x)%
+\omega^{G}_{\lambda}(x,y,z)}}\\ \notag%
+&\kappa_{9}(\psi(\omega^{G}_{\lambda}(x,y,z)))\\\notag%
\times&\psi\Set{\frac{\omega^{G}_{\lambda}(x,y,z)%
[1+\omega^{G}_{\lambda}(y,T^{m}y,T^{m}y)]}%
{1+\omega^{G}_{\lambda}(T^{m}x,T^{m}x,y)%
+\omega^{G}_{\lambda}(T^{m}x,T^{m}y,T^{m}y)}}\\ \notag%
+&\kappa_{10}(\psi(\omega^{G}_{\lambda}(x,y,z)))\\\notag%
\times&\psi\Set{\frac{\omega^{G}_{\lambda}%
(T^{m}x,T^{m}y,T^{m}y)[1+\omega^{G}_{\lambda}(x,T^{m}y,T^{m}y)]}%
{1+\omega^{G}_{\lambda}(T^{m}x,T^{m}x,y)%
+\omega^{G}_{\lambda}(x,T^{m}y,T^{m}y)}}\\ \notag%
+&\kappa_{11}(\psi(\omega^{G}_{\lambda}(x,y,z)))\\%
\times&\psi\Set{\frac{\omega^{G}_{\lambda}%
(T^{m}x,T^{m}y,T^{m}y)%
[1+\omega^{G}_{\lambda}(y,T^{m}y,T^{m}y)%
+\omega^{G}_{\lambda}(T^{m}x,T^{m}x,y)]}%
{1+\omega^{G}_{\lambda}(T^{m}x,T^{m}x,y)%
+\omega^{G}_{\lambda}(T^{m}x,T^{m}y,T^{m}y)}},%
\end{align}%
where $h_{n}=\psi(\omega^{G}_{\lambda}(T^{m}x,T^{m}y,T^{m}z))$, %
$\psi\in \bar{\Psi}$ , $\{\kappa_{1},\kappa_{2}, \cdots, \kappa_{11}\} %
\in \mathcal{S}_{Ger}$ and %
$\max\{\sup_{t\geq 0}\kappa_{1}(t),\\\sup_{t\geq 0}\kappa_{2}(t), %
\cdots, \sup_{t\geq 0}\kappa_{11}(t)\}<1.$ %
Assume that if a non-decreasing sequence $\{x_{n}\}_{n\in\mathbb{N}}$ %
converges to $x^{\ast}$, then $x_{n}\preceq x^{\ast}$ for each %
$n\in\mathbb{N}$,%
\newline %
$(2)$ if $\psi$ is sub-additive and for any $x,y,z\in X_{\omega^{G}}$, %
there exists $u\in X_{\omega^{G}}$ with $u\preceq Tu$ and %
$\omega^{G}_{\lambda}(u,Tu,Tu)$ is finite  for all $\lambda>0$ %
such that $u$ is comparable to  $x,y~{\rm and}~z$. %
Then $T$ has a  fixed point $x^{\ast}\in X_{\omega^{G}}$ %
for some positive integer %
$m\ge 1$ in $X_{\omega^{G}}$ and the sequence define by %
$\{T^{n}x_{0}\}_{n\ge 1}$ converges to $x_{\ast}$. %
Furthermore, the fixed point of $T$ is unique. %
\end{theorem} %

\begin{proof} %
By \autoref{qq}, $T^{m}$ has a  fixed point %
say $w_{\ast}\in X_{\omega^{G}}$ by using inequality %
\ref{k} for some positive integer $m\ge 1$. Now $T^{m}(Tw_{\ast}) %
=T^{m+1}w_{\ast}=T(T^{m}w_{\ast})=Tw_{\ast}$, %
so $Tw_{\ast}$ is a fixed point of $T^{m}$. %
By the uniqueness of fixed point of $T^{m}$, we have %
$Tw_{\ast}=w_{\ast}$. Therefore, $w_{\ast}$ is a %
fixed point of $T$. Since fixed point of $T$ is also %
fixed point of $T^{m}$, hence $T$ has a unique fixed %
point in $X_{\omega^{G}}$. %

\end{proof}%


\section{Applications to nonlinear integral equations} %
We construct a system of nonlinear integral %
equations that will satisfies the conditions of %
\autoref{3.1} as follows. %

Consider the following nonlinear integral equations %

\begin{equation}\label{L} %
u(t)=\int_{0}^{\Omega}H_{1}(t,s,u(s))ds+\mu_{1}(t), %
\end{equation} %
and %

\begin{equation}\label{LL} %
u(t)=\int_{0}^{\Omega}H_{2}(t,s,u(s))ds+\mu_{2}(t),~~\forall~~t\in[0,\Omega], %
\end{equation} %

where $H_{1},H_{2}:[0,\Omega]\times[0,\Omega]\times\mathbb{R}\rightarrow %
\mathbb{R}$ and $\mu:\mathbb{R}\to\mathbb{R}$ are continuous, $\Omega>0$. %
For all $t,s\in[0,\Omega]$, take %

\begin{equation} %
	H_{1}(t,s,u(t))\le %
H_{2}\biggl(t,s,\int_{0}^{\Omega}H_{1}(s,z,u(z))ds+\mu_{1}(s)\biggr) %
\end{equation} %

and %
 \begin{equation}%
H_{2}(t,s,u(t))\le H_{1}\biggl(t,s,\int_{0}^{\Omega}H_{2}(s,z,u(z))ds %
+\mu_{2}(s)\biggr). %
\end{equation} %

Again let $X_{\omega^{G}}=C[0,\Omega]$ be the %
space of continuous functions %
defined on $[0,\Omega]$. $X_{\omega^{G}}=C[0,\Omega]$  %
can also be equipped with a %
preorder  $\preceq$ given by for all $x,y\in %
 X_{\omega^{G}}$ %
given by $x\preceq y\iff x(t)\le y(t)$ for all %
 $t\in [0,\Omega]$. %
Suppose that there exists a continuous functions %
$B:[0,\Omega]\times[0,\Omega]\to\mathbb{R}_{+}$, %
 $D_{k}:[0,\Omega]\times[0,\Omega]\rightarrow %
\mathbb{R}_{+}$ and $\kappa_{k}:[0,\infty)\rightarrow [0,1)$  %
for all $k\ge 1$ such that %

\begin{equation} %
\norm{H_{1}(t,s,x)-H_{2}(t,s,y)}\le \sum_{k=1}^{n}B(t,s) %
\dfrac{\sin\norm{\frac{1}{\lambda}(x-y)}}{ %
\norm{\frac{1}{\lambda}(x-y)}}D_{k}(x,y), %
\end{equation} %

where $$\sup_{t\in [0,\Omega]}\int_{0}^{\Omega} %
B^2(t,s)ds\le \dfrac{1}{\Omega},$$ $\lambda>0$.  %

Now for any $\lambda>0$, we define %

\begin{equation}\label{0}%
\omega^{G}_{\lambda}(x,y,z):=\frac{1}{2(1+\lambda)} %
\sup_{t\in[0,\Omega]}\{\norm{x(t)-y(t)}+ %
\norm{y(t)-z(t)}+\norm{x(t)-z(t)}\}, %
\end{equation} %

so that %

\begin{equation}\label{00}%
\omega^{G}_{\lambda}(x,y,y):=\frac{1}{(1+\lambda)} %
\sup_{t\in[0,\Omega]}\{\norm{x(t)-y(t)}\}. %
\end{equation}%

In fact \autoref{0} and \autoref{00} satisfies all %
the conditions in Definition \ref{2.1} endowed with %
 $X_{\omega^{G}}=(X_{\omega^{G}},\omega^{G}) %
=C([0,\Omega],[0,\Omega]).$ %

Let %
\begin{equation}\label{12}%
G_{u}=\int_{0}^{\Omega}H_{1}(t,s,u(s))ds-\mu_{1}(t), %
\end{equation}
and \begin{equation}\label{13}
G_{v}=\int_{0}^{\Omega}H_{2}(t,s,v(s))ds-\mu_{2}(t), %
\end{equation}

Now, for distinct $x,y\in X_{\omega^{G}}$, let %
 $W=\frac{1}{1+\lambda}\sup_{t\in[0,\Omega]}%
\bigl\{\norm{G_{x}(t)+\mu_{1}(t)-G_{y}(t)-\mu_{2}(t)}\bigr\}>0$, so that by %
 Cauchy-Schwartz inequality, %
\begin{align}\label{4.1}%
W\notag %
=&\frac{1}{1+\lambda}\sup_{t\in[0,\Omega]} %
\biggl\{\norm{\int_{0}^{\Omega}H_{1}(t,s,x(s))ds %
-\int_{0}^{\Omega}H_{2}(t,s,y(s))ds}\biggr\}\\\notag %
\le&\frac{1}{1+\lambda}\sup_{t\in[0,\Omega]}\biggl\{\int_{0}^{\Omega} %
\norm{H_{1}(t,s,x(s))-H_{2}(t,s,y(s))}ds\biggr\}\\\notag %
\le&\frac{1}{1+\lambda}\sup_{t\in[0,\Omega]}%
\biggl\{\int_{0}^{\Omega}\sum_{k=1}^{n}B(t,s) %
\sqrt{\dfrac{\sin^{2}\norm{\frac{1}{\lambda}(x(s)-y(s))}} %
{\norm{\frac{1}{\lambda}(x(s)-y(s))}^{2}}}D_{k}(x(s),y(s))ds\biggr\}\\\notag %
\le&\frac{1}{1+\lambda}\sup_{t\in[0,\Omega]} %
\biggl\{\sum_{k=1}^{n}\int_{0}^{\Omega}B(t,s) %
\sqrt{\dfrac{\sin^{2}\norm{\frac{1}{\lambda}(x(s)-y(s))}} %
{\norm{\frac{1}{\lambda}(x(s)-y(s))}^{2}}}D_{k}(x(s),y(s))ds\biggr\}\\\notag %
\le&\frac{1}{1+\lambda}\sup_{t\in[0,\Omega]}\sum_{k=1}^{n} %
\biggl\{\bigg(\int_{0}^{\Omega}B(t,s)^{2}ds\bigg)^{\frac{1}{2}}\\\notag %
\times&\bigg(\int_{0}^{\Omega}\dfrac{\sin^{2}\norm{ %
\frac{1}{\lambda}(x(s)-y(s))}}{\norm{\frac{1}{\lambda}(x(s)-y(s))}^{2}} %
D_{k}(x(s),y(s))^{2}ds\bigg)^{\frac{1}{2}}\biggr\}\\\notag %
\le&\frac{1}{1+\lambda}\sup_{t\in[0,\Omega]}\sum_{k=1}^{n} %
\biggl\{\bigg(\dfrac{1}{\Omega}\bigg)^{\frac{1}{2}}\notag %
\bigg(\int_{0}^{\Omega}\dfrac{\sin^{2}\norm{ %
\frac{1}{\lambda}(x(s)-y(s))}}{\norm{\frac{1}{\lambda}(x(s)-y(s))}^{2}} %
D_{k}(x(s),y(s))^{2}ds\bigg)^{\frac{1}{2}}\biggr\}\\\notag %
\le&\frac{1}{1+\lambda}\sup_{t\in[0,\Omega]}\sum_{k=1}^{n} %
\biggl\{\bigg(\dfrac{1}{\Omega}\bigg)^{\frac{1}{2}}\notag %
\bigg(\dfrac{\sin\omega^{G}_{\lambda}(x,y,y)}{\omega^{G}_{\lambda}(x,y,y)} %
D_{k}(x,y)\sqrt{\Omega}\bigg)\biggr\}\\ %
=&\frac{1}{1+\lambda}\sup_{t\in[0,\Omega]}\sum_{k=1}^{n} %
\biggl\{\bigg(\dfrac{\sin\omega^{G}_{\lambda}(x,y,y)} %
{\omega^{G}_{\lambda}(x,y,y)}D_{k}(x,y)\bigg)\biggr\}. %
\end{align}%

Now, take $\omega^{G}_{\lambda}(x,y,y):=\theta$, so that %
$\frac{\sin\theta}{\theta}\rightarrow 1$ as  %
$\theta\rightarrow 0$. This agrees  %
with the Definition \ref{2.2}. Hence for all  %
$k\in\mathbb{N}$,  %
$\kappa_{k}(\theta)=\frac{\sin\theta}{\theta}$. %
Therefore, from inequality \ref{4.1}, we get %

\begin{equation}%
\frac{1}{1+\lambda}\sup_{t\in[0,\Omega]} %
\bigl\{\norm{G_{x}(t)+\mu_{1}(t)-G_{y}(t)-\mu_{2}(t)}\bigr\}\le %
\sup_{t\in[0,\Omega]}\sum_{k=1}^{n} %
\biggl\{\kappa_{k}\big(\omega^{G}_{\lambda}(x,y,y)\big)D_{k}(x,y)(t)\biggr\}. %
\end{equation}%

 For the application here, we have %
 \begin{equation}\label{z} %
 \frac{1}{1+\lambda}\sup_{t\in[0,\Omega]} %
 \bigl\{\norm{G_{x}(t)+\mu_{1}(t)-G_{y}(t)-\mu_{2}(t)}\bigr\}\le %
 \sup_{t\in[0,\Omega]}\sum_{k=1}^{11} %
 \biggl\{\kappa_{k}\big(\omega^{G}_{\lambda}(x,y,y)\big)D_{k}(x,y)(t)\biggr\}, %
 \end{equation}%

where %
\begin{equation*}%
D_{1}(x,y)(t)=\frac{1}{1+\lambda+\nu(\lambda)} %
\norm{x(t)-y(t)} %
\end{equation*}%

 for $\nu(\lambda)\in [0,\lambda).$ %

\begin{equation*}%
D_{2}(x,y)(t)=\dfrac{1}{1+\lambda} %
\norm{G_{x}(t)+\mu_{1}(t)-x(t)}. %
\end{equation*} %

\begin{equation*}%
D_{3}(x,y)(t)=\dfrac{1}{1+\lambda} %
\norm{G_{y}(t)+\mu_{2}(t)-y(t)}. %
\end{equation*}%

\begin{equation*}%
D_{4}(x,y)(t)=\dfrac{1}{1+\lambda} %
\norm{G_{x}(t)+\mu_{1}(t)-G_{y}(t)-\mu_{2}(t)} %
\end{equation*}%

\begin{equation*}%
D_{5}(x,y)(t)=\dfrac{1}{1+\lambda} %
\norm{G_{x}(t)+\mu_{1}(t)-y(t)}. %
\end{equation*}%

\begin{equation*}%
D_{6}(x,y)(t)=\dfrac{\frac{1}{1+\lambda} %
\norm{G_{y}(t)+\mu_{2}(t)-y(t)}\bigg(1+\frac{1}{1+\lambda} %
\norm{x(t)-y(t)}\bigg)}{1+\frac{1}{1+\lambda}\norm{G_{x}(t) %
+\mu_{1}(t)-x(t)}}. %
\end{equation*}%

\begin{equation*}%
D_{7}(x,y)(t)=\dfrac{\frac{1}{1+\lambda} %
\norm{G_{y}(t)+\mu_{2}(t)-y(t)}\bigg(1+ %
\frac{1}{1+\lambda}\norm{G_{x}(t)+\mu_{1}(t)-x(t)}\bigg)} %
{1+\frac{1}{1+\lambda}\norm{x(t)-y(t)}}. %
\end{equation*}%

\begin{equation*}%
D_{8}(x,y)(t)=\dfrac{\frac{1}{1+\lambda}\norm{G_{x}(t) %
+\mu_{1}(t)-x(t)}\bigg(1+\frac{1}{1+\lambda} %
\norm{x(t)-y(t)}\bigg)}{1+\frac{1}{1+\lambda} %
\bigg(\norm{G_{x}(t)+\mu_{1}(t)-y(t)}+\norm{x(t)-y(t)}\bigg)}. %
\end{equation*}%

\begin{equation*}%
D_{9}(x,y)(t)=\dfrac{\frac{1}{1+\lambda} %
\norm{x(t)-y(t)}\bigg(1+\frac{1}{1+\lambda} %
\norm{G_{y}(t)+\mu_{2}(t)-y(t)}\bigg)}{1+ %
\frac{1}{1+\lambda}\bigg(\norm{G_{x}(t) %
+\mu_{1}(t)-y(t)}+\norm{G_{x}(t)+\mu_{1}(t)-G_{y}(t) %
-\mu_{2}(t)}\bigg)}. %
\end{equation*}%

\begin{equation*}%
D_{10}(x,y)(t)=\dfrac{\frac{1}{1+\lambda} %
\norm{G_{x}(t)+\mu_{1}(t)-G_{y}(t)-\mu_{2}(t)} %
\bigg(1+\frac{1}{1+\lambda}\norm{G_{y}(t)+ %
\mu_{2}(t)-x(t)}\bigg)}{1+\frac{1}{1+\lambda} %
\bigg(\norm{G_{x}(t)+\mu_{1}(t)-y(t)}+\norm{G_{y}(t) %
-\mu_{2}(t)-x(t)}\bigg)}. %
\end{equation*} %

$D_{11}(x,y)(t)=$ %
\begin{equation*}%
\dfrac{\frac{1}{1+\lambda}\norm{G_{x}(t) %
+\mu_{1}(t)-G_{y}(t)-\mu_{2}(t)}\bigg(1 %
+\frac{1}{1+\lambda}\norm{G_{y}(t)+\mu_{2}(t)-y(t)} %
+\norm{G_{x}(t)+\mu_{1}(t)-y(t)}\bigg)}{1+\frac{1} %
{1+\lambda}\bigg(\norm{G_{x}(t)+\mu_{1}(t)-y(t)} %
+\norm{G_{x}(t)+\mu_{1}(t)-G_{y}(t)-\mu_{2}(t)}\bigg)} %
\end{equation*}%


\begin{theorem}%
Let $X_{\omega^{G}}=C([0,\Omega],[0,\Omega])$ be a %
 complete preordered modular G-metric space  and  %
 $\omega^{G}:(0,\infty)\times X_{\omega^{G}}\times  %
 X_{\omega^{G}}\times X_{\omega^{G}}\to [0,\Omega]\cup\{\infty\}$  %
 be define by %
 \begin{equation}%
\omega^{G}_{\lambda}(x,y,y):=\frac{1}{(1+\lambda)} %
\sup_{t\in[0,\Omega]}\{\norm{x(t)-y(t)}\}, \lambda>0. %
\end{equation}%
 Let $G_{x}, G_{y}:[0,\Omega]\times[0,\Omega] %
 \times\mathbb{R}\rightarrow %
 \mathbb{R}$ and $\mu:\mathbb{R}\to\mathbb{R}$  %
 are continuous, $\Omega>0$. %
 For all $t,s\in[0,\Omega]$, are such that %
 $G_{x},G_{y}\in X_{\omega^{G}}$ for all %
 $x,y\in X_{\omega^{G}}$ and $G_{x}$, $G_{y}$ %
 satisfying \autoref{12} and \autoref{13} %
 respectively for all  $t,s\in[0,\Omega]$. %
 Suppose that there exists non-negative real numbers %
  $\kappa_{1},\kappa_{2},\cdots,\kappa_{11}>0$ with %
   $\sum_{k=1}^{11}\kappa_{k}(\psi(\omega^{G}_{ %
 \lambda}(x,y,y)))\in\mathcal{S}_{Ger}$ and %
 $\max\{\sup_{t\geq 0}\kappa_{1}(t),\sup_{t\geq 0}\kappa_{2}(t), %
 \cdots, \\\sup_{t\geq 0}\kappa_{11}(t)\}<1.$ %
 for all $\lambda>0$, $\psi\in \bar{\Psi}$ such %
 that inequality \ref{z} is satisfied for every  %
 $x,y\in X_{\omega^{G}}.$ Moreover if $\psi$ is %
 sub-additive and for any $x,y\in X_{\omega^{G}}$,  %
 there exists $u\in X_{\omega^{G}}$  such that $u$  %
 is comparable to both $x~{\rm and}~y$. %
Then the system of integral Equations \ref{L} and \ref{LL} %
have a unique solution in $X_{\omega^{G}}.$ %
\end{theorem}%

\begin{proof} %
Let $T:X_{\omega^{G}}\rightarrow X_{\omega^{G}}$ be non-decreasing %
and asymptotically $T$-regular map on some point $x_{0}\in X_{\omega^{G}}$ %
 define by %
$Tx=G_{x}+\mu_{1}$ and $Ty=G_{y}+\mu_{2}$. Then for $\lambda>0$, %
there exists $u\in X_{\omega^{G}}$ %
with $u\preceq Tu$ and $\omega^{G}_{\lambda}(u,Tu,Tu)$ is finite   %
for all $\lambda>0$. %
So from inequality  %
\ref{z}, we get by noticing that $\psi$ is continuous %
and sub-additive and there exists $\nu\in [0,\lambda)$ such that %
			
\begin{align}\label{}%
\psi(\omega^{G}_{\lambda}(Tx,Ty,Ty))\notag%
\le&%
\kappa_{1}(\psi(\omega^{G}_{\lambda}(x,y,y)))%
\psi(\omega^{G}_{\lambda+\nu(\lambda)}(x,y,y)) \notag%
+\kappa_{2}(\psi(\omega^{G}_{\lambda}(x,y,y)))%
\psi(\omega^{G}_{\lambda}(x,Tx,Tx))\\ \notag%
+& \kappa_{3}(\psi(\omega^{G}_{\lambda}(x,y,y)))%
\psi(\omega^{G}_{\lambda}(y,Ty,Ty)) \notag%
+\kappa_{4}(\psi(\omega^{G}_{\lambda}(x,y,y)))%
\psi(\omega^{G}_{\lambda}(Tx,Ty,Ty))\\ \notag%
+&\kappa_{5}(\psi(\omega^{G}_{\lambda}(x,y,y)))%
\psi(\omega^{G}_{\lambda}(Tx,y,y))\\ \notag%
+&\kappa_{6}(\psi(\omega^{G}_{\lambda}(x,y,y)))\notag%
\psi\Set{\frac{\omega^{G}_{\lambda}(y,Ty,Ty)%
[1+\omega^{G}_{\lambda}(x,y,y)]}{1+\omega^{G}_{\lambda}(x,Tx,Tx)}}\\ \notag%
+&\kappa_{7}(\psi(\omega^{G}_{\lambda}(x,y,y)))\notag%
\psi\Set{\frac{\omega^{G}_{\lambda}(y,Ty,Ty)%
[1+\omega^{G}_{\lambda}(x,Tx,Tx)]}{1+\omega^{G}_{\lambda}(x,y,y)}}\\ \notag%
+&\kappa_{8}(\psi(\omega^{G}_{\lambda}(x,y,y)))\notag%
\psi\Set{\frac{\omega^{G}_{\lambda}(x,Tx,Tx)%
[1+\omega^{G}_{\lambda}(x,y,y)]}{1+\omega^{G}_{\lambda}(y,Tx,Tx)%
+\omega^{G}_{\lambda}(x,y,y)}}\\ \notag%
+&\kappa_{9}(\psi(\omega^{G}_{\lambda}(x,y,y)))\notag%
\psi\Set{\frac{\omega^{G}_{\lambda}(x,y,y)%
[1+\omega^{G}_{\lambda}(y,Ty,Ty)]}{1+\omega^{G}_{\lambda}(Tx,Tx,y)%
+\omega^{G}_{\lambda}(Tx,Ty,Ty)}}\\ \notag%
+&\kappa_{10}(\psi(\omega^{G}_{\lambda}(x,y,y)))\notag%
\psi\Set{\frac{\omega^{G}_{\lambda}(Tx,Ty,Ty)%
[1+\omega^{G}_{\lambda}(x,Ty,Ty)]}{1+\omega^{G}_{\lambda}(Tx,Tx,y)%
+\omega^{G}_{\lambda}(x,Ty,Ty)}}\\ \notag%
+&\kappa_{11}(\psi(\omega^{G}_{\lambda}(x,y,y)))\\%
\times&\psi\Set{\frac{\omega^{G}_{\lambda}(Tx,Ty,Ty)%
[1+\omega^{G}_{\lambda}(y,Ty,Ty)+\omega^{G}_{\lambda}(Tx,Tx,y)]}%
{1+\omega^{G}_{\lambda}(Tx,Tx,y)+\omega^{G}_{\lambda}(Tx,Ty,Ty)}}.%
\end{align}
By \autoref{3.1}, then the system of integral Equations \ref{L} and \ref{LL}
have a unique solution in $X_{\omega^{G}}.$	
\end{proof}%

\begin{corollary} %
	Let $X_{\omega^{G}}=C([0,\Omega],[0,\Omega])$ be a %
	complete preordered modular G-metric space  and %
	 $\omega^{G}:(0,\infty)\times X_{\omega^{G}}\times %
	  X_{\omega^{G}}\times X_{\omega^{G}}\to [0,\Omega] %
	  \cup\{\infty\}$ be define by\begin{equation} %
	\omega^{G}_{\lambda}(x,y,y):=\frac{1}{(1+\lambda)} %
	\sup_{t\in[0,\Omega]}\{\norm{x(t)-y(t)}\}, \lambda>0. %
	\end{equation} %
	Let $G_{x}, G_{y}:[0,\Omega]\times[0,\Omega] %
	\times\mathbb{R}\rightarrow %
	\mathbb{R}$ and $\mu:\mathbb{R}\to\mathbb{R}$ are %
	 continuous, $\Omega>0$. %
	For all $t,s\in[0,\Omega]$, are such that  %
	$G_{x},G_{y}\in X_{\omega^{G}}$ for all  %
	$x,y\in X_{\omega^{G}}$ and $G_{x}$, $G_{y}$  %
	satisfying \autoref{12} and \autoref{13} %
	 respectively for all  $t,s\in[0,\Omega]$. %
	 Suppose that there exists non-negative real numbers  %
	  $\kappa_{1},\kappa_{2},\kappa_{3}>0$ with %
	 $\sum_{k=1}^{3}\kappa_{k}(\psi(\omega^{G}_{ %
	 	\lambda}(x,y,y)))\in\mathcal{S}_{Ger}$ %
	and $\kappa_{1}(t) %
	+2\max\{\sup_{t\geq 0}\kappa_{2}(t), \sup_{t\geq 0} %
	\kappa_{3}(t)\}<1,$ %
	for distinct $x,y\in X_{\omega^{G}}$ and for all %
	 $\lambda>0$, $\psi\in \bar{\Psi}$ such that %
	 inequality \ref{z} is satisfied for every %
	 $x,y\in X_{\omega^{G}}.$ Moreover if $\psi$ is %
	 sub-additive and for any $x,y\in X_{\omega^{G}}$, %
	 there exists $u\in X_{\omega^{G}}$ such that $u$ is %
	 comparable to both $x~{\rm and}~y$. %
	Then the system of integral Equations \ref{L} and \ref{LL} %
	have a unique solution in $X_{\omega^{G}}.$ %
\end{corollary}	%

\begin{proof} %
Let $T:X_{\omega^{G}}\rightarrow X_{\omega^{G}}$  %
be non-decreasing %
and asymptotically $T$-regular map on some point  %
$x_{0}\in X_{\omega^{G}}$ %
define by %
$Tx=G_{x}+\mu_{1}$ and $Ty=G_{y}+\mu_{2}$. %
Then for $\lambda>0$, %
there exists $u\in X_{\omega^{G}}$ %
with $u\preceq Tu$ and $\omega^{G}_{\lambda}(u,Tu,Tu)$  %
is finite  for all $\lambda>0$. %
So from inequality \ref{z}, we have %
$\omega^{G}_{\lambda}(Tx,Ty,Ty)=\frac{1} %
{1+\lambda}\sup_{t\in[0,\Omega]} %
\bigl\{\norm{G_{x}(t)+\mu_{1}(t)-G_{y}(t)- %
	\mu_{2}(t)}\bigr\}$, %
$\omega^{G}_{\lambda}(x,y,y)=\frac{1}{1+ %
\lambda+\nu(\lambda)}\norm{x(t)-y(t)}$, %
for $\nu(\lambda)\in [0,\lambda).$  %
$\omega^{G}_{\lambda}(x,Tx,Tx)=\dfrac{1}{1+\lambda} %
\norm{G_{x}(t)+\mu_{1}(t)-x(t)},$ and %
$\omega^{G}_{\lambda}(y,Ty,Ty)=\dfrac{1}{1+\lambda} %
\norm{G_{y}(t)+\mu_{2}(t)-y(t)}.$ %
Since $\psi$ is continuous %
and sub-additive, then by inequality \ref{z} for  %
$k=1,2,3$, we have %
\begin{align}\label{}%
\psi(\omega^{G}_{\lambda}(Tx,Ty,Ty))\notag %
\le&%
\kappa_{1}(\psi(\omega^{G}_{\lambda}(x,y,y))) %
\psi(\omega^{G}_{\lambda+\nu(\lambda)}(x,y,y)) \notag %
+\kappa_{2}(\psi(\omega^{G}_{\lambda}(x,y,y))) %
\psi(\omega^{G}_{\lambda}(x,Tx,Tx))\\ \notag %
+& \kappa_{3}(\psi(\omega^{G}_{\lambda}(x,y,y))) %
\psi(\omega^{G}_{\lambda}(y,Ty,Ty)). %
\end{align}%
By Corollary \ref{aaa}, then the system of  %
integral Equations \ref{L} and \ref{LL} %
have a unique solution in $X_{\omega^{G}}.$ %
\end{proof}	%
\newline	
\textbf{Data Availability}\newline
The data used to support the findings of this study are included within the
article.\newline
\newline
\textbf{Conflicts of Interest}\newline
The authors declare no conflict of interest.\newline
\newline
\textbf{Authors Contributions}\newline
All authors contributed equally to the writing of this paper.

%


\end{document}